\documentclass[11pt,english]{extarticle}
\usepackage[latin9]{inputenc}
\usepackage{geometry}
\usepackage{color}
\usepackage{babel}
\usepackage{array}
\usepackage{mathrsfs}
\usepackage{enumitem}
\usepackage{amsmath}
\usepackage{amsthm}
\usepackage{amssymb}
\usepackage{stmaryrd}
\usepackage{graphicx}
\PassOptionsToPackage{normalem}{ulem}
\usepackage{ulem}

\usepackage{hyperref}

\hypersetup{
	unicode=true,pdfusetitle,
bookmarks=true,bookmarksnumbered=false,bookmarksopen=false,
breaklinks=false,pdfborder={0 0 0},pdfborderstyle={},backref=false,colorlinks=true
}

\usepackage{epstopdf}

\makeatletter

\providecommand{\tabularnewline}{\\}

\numberwithin{equation}{section}
\numberwithin{figure}{section}
  \theoremstyle{plain}
  \newtheorem{lem}{\protect\lemmaname}
  \theoremstyle{plain}
  \newtheorem{assumption}{\protect\assumptionname}
\theoremstyle{plain}
\newtheorem{thm}{\protect\theoremname}
  \theoremstyle{remark}
  \newtheorem{rem}{\protect\remarkname}
  \theoremstyle{definition}
  \newtheorem{defn}{\protect\definitionname}

\usepackage{dsfont}
\usepackage{relsize}
\usepackage{subdepth}
\allowdisplaybreaks

\DeclareFontFamily{OT1}{pzc}{}
\DeclareFontShape{OT1}{pzc}{m}{it}{<-> s * [1.200] pzcmi7t}{}
\DeclareMathAlphabet{\mathpzc}{OT1}{pzc}{m}{it}

\usepackage[noadjust]{cite}
\usepackage{filecontents}

\usepackage{stfloats}

\usepackage{babel}

\renewcommand\footnotemark{}

\usepackage{bigints}

\usepackage{scalerel}

\def\scaleint#1{\vcenter{\hbox{\scaleto[4ex]{\displaystyle\int}{#1}}}}

\@ifundefined{showcaptionsetup}{}{%
\PassOptionsToPackage{caption=false}{subfig}}
\usepackage{subfig}
\makeatother

\providecommand{\assumptionname}{Assumption}
\providecommand{\definitionname}{Definition}
\providecommand{\lemmaname}{Lemma}
\providecommand{\remarkname}{Remark}
\providecommand{\theoremname}{Theorem}

\begin{document}

\title{\textbf{Spatially Controlled Relay Beamforming:}\\
\textbf{2-Stage Optimal Policies}}

\author{Dionysios S. Kalogerias\thanks{The Authors are with the Department of Electrical \& Computer Engineering,
Rutgers, The State University of New Jersey, 94 Brett Rd, Piscataway,
NJ 08854, USA. e-mail: \{d.kalogerias, athinap\}@rutgers.edu.}\thanks{This work is supported by the National Science Foundation (NSF) under
Grants CCF-1526908 \& CNS-1239188.}\thanks{Also, this work constitutes an extended preprint of a two part paper
(soon to be) submitted for publication to the IEEE Transactions on
Signal Processing in Spring/Summer 2017.} and Athina P. Petropulu}

\date{May 2017}

\maketitle
\textbf{\vspace{-30pt}
}
\begin{abstract}
The problem of enhancing Quality-of-Service (QoS) in power constrained,
mobile relay beamforming networks, by optimally and dynamically controlling
the motion of the relaying nodes, is considered, in a dynamic channel
environment. We assume a time slotted system, where the relays update
their positions before the beginning of each time slot. Modeling the
wireless channel as a Gaussian spatiotemporal stochastic field, we
propose a novel $2$-stage stochastic programming problem formulation
for optimally specifying the positions of the relays at each time
slot, such that the expected QoS of the network is maximized, based
on causal Channel State Information (CSI) and under a total relay
transmit power budget. This results in a schema where, at each time
slot, the relays, apart from optimally beamforming to the destination,
also optimally, predictively decide their positions at the next time
slot, based on causally accumulated experience. Exploiting either
the Method of Statistical Differentials, or the multidimensional Gauss-Hermite
Quadrature Rule, the stochastic program considered is shown to be
approximately equivalent to a set of simple subproblems, which are
solved in a distributed fashion, one at each relay. Optimality and
performance of the proposed spatially controlled system are also effectively
assessed, under a rigorous technical framework; strict optimality
is rigorously demonstrated via the development of a version of the
Fundamental Lemma of Stochastic Control, and, performance-wise, it
is shown that, quite interestingly, the optimal average network QoS
exhibits an increasing trend across time slots, despite our myopic
problem formulation. Numerical simulations are presented, experimentally
corroborating the success of the proposed approach and the validity
of our theoretical predictions.\textbf{\vspace{15pt}
}
\end{abstract}
\textbf{\textit{$\quad$}}\textbf{Keywords.} Spatially Controlled
Relay Beamforming, Mobile Relay Beamforming, Network Mobility Control,
Network Utility Optimization, QoS Maximization, Motion Control, Distributed
Cooperative Networks, Stochastic Programming.

\newpage{}

\tableofcontents{}

\section{Introduction}

Distributed, networked communication systems, such as relay beamforming
networks \cite{Havary_BEAM_2008,Havary_BEAM_2010,Beamforming_1_2009,Beamforming_2_2009,Beamforming_6_Petropulu2011,Beamforming_7_Petropulu2011,Beamforming_8_Petropulu2012}
(e.g., Amplify \& Forward (AF)) are typically designed without explicitly
considering how the positions of the networking nodes might affect
the quality of the communication. Optimum physical placement of assisting
networking nodes, which could potentially improve the quality of the
communication, does not constitute a clear network design aspect.
However, in most practical settings in physical layer communications,
the Channel State Information (CSI) observed by each networking node,
per channel use, although (modeled as) random, it is both spatially
and temporally correlated. It is, therefore, reasonable to ask if
and how system performance could be improved by controlling the positions
of certain network nodes, based on causal side information,
and exploiting the spatiotemporal dependencies of the wireless medium. 

Recently, autonomous node mobility has been proposed as an effective
means to further enhance performance in various distributed network
settings. In \cite{NikosBeam-2}, optimal transmit AF beamforming
has been combined with potential field based relay mobility control
in multiuser cooperative networks, in order to minimize relay transmit
power, while meeting certain Quality-of-Service (QoS) constraints.
In \cite{KalPet-Jammers-2013}, in the framework of information theoretic
physical layer security, decentralized jammer motion control has been
jointly combined with noise nulling and cooperative jamming, maximizing
the network secrecy rate. In \cite{KalPet-Mobi-2014}, optimal relay
positioning has been studied in systems where multiple relays deliver
information to a destination, in the presence of an eavesdropper,
with a goal of maximizing or achieving a target level of ergodic secrecy.

In the complementary context of communication aware (comm-aware) robotics,
node mobility has been exploited in distributed robotic networks,
in order to enhance system performance, in terms of maintaining reliable,
in-network communication connectivity \cite{Comm_Aware_1,Comm_Aware_3,Comm_Aware_6,Comm_Aware_7},
and optimizing network energy management \cite{Comm_Aware_2}. Networked
node motion control has also been exploited in special purpose applications,
such as networked robotic surveillance \cite{Comm_Aware_4} and target
tracking \cite{Comm_Aware_5}.

In \cite{NikosBeam-2,KalPet-Jammers-2013,KalPet-Mobi-2014}, the links
among the nodes of the network (or the related statistics) are assumed
to be available in the form of static channel maps, during the whole
motion of the jammers/relays. However, this is an oversimplifying
assumption in scenarios where the channels change significantly in
time and space \cite{Giannakis_Spatial1_2011,Giannakis_Spatial2_2011,MostofiSpatial2012}.

In this paper, we try to overcome this major limitation, and we consider
the problem of optimally and dynamically updating relay positions
in one source/destination relay beamforming networks, \textit{in a
dynamic channel environment}. Different from \cite{NikosBeam-2,KalPet-Jammers-2013,KalPet-Mobi-2014},
we model the wireless channel as a \textit{spatiotemporal stochastic
field}; this approach may be seen as a versatile extension of a realistic,
commonly employed ``log-normal'' channel model \cite{MostofiSpatial2012}.
We then propose a $2$-\textit{stage stochastic programming} problem
formulation, optimally specifying the positions of the relays \textit{at
each time slot}, such that the \textit{Signal-to-Interference+Noise
Ratio (SINR) }or \textit{QoS} at the destination, at the same time
slot, is maximized \textit{on average}, based on \textit{causal CSI},
and subject to a \textit{total power constraint at the relays}. At
each time slot, the relays not only beamform to the destination, but
also optimally, \textit{predictively} decide their positions at the
next time slot, based on their experience (causal actions and channel
observations). This novel, cyber-physical system approach to relay
beamforming is termed as \textit{Spatially Controlled Relay Beamforming}.

Exploiting the assumed stochastic channel structure, it is first shown
that the proposed optimal motion control problem is equivalent to
a set of simpler, two dimensional subproblems, which can be solved
in a \textit{distributed fashion}, one at each relay, \textit{without
the need for intermediate exchange of messages }among the relays.
However, each of the objectives of the aforementioned subproblems
involves the evaluation of a conditional expectation of a well defined
ratio of almost surely positive random variables, which is \textit{impossible
to perform analytically}, calling for the development of easily implementable
approximations to each of the original problems. Two such heuristics
are considered. The first is based on the so-called \textit{Method
of Statistical Differentials} \cite{Survival_1990Elandt}, whereas
the second constitutes a brute force approach, based on the \textit{multidimensional
Gauss-Hermite Quadrature Rule}, a readily available routine for numerical
integration. In both cases, the original problem objective is replaced
by the respective approximation, which, in both cases, is shown to
be easily computed via simple, closed form expressions. The computational
complexity of both approaches is also discussed and characterized.
Subsequently, we present an important result, along with the respective
detailed technical development, characterizing the performance of
the proposed system, \textit{across time slots} (Theorems \ref{lem:QoS_INCREASES}
and \ref{prop:QoS-Increases}). In a nutshell, this result states
that, although our \textit{problem objective is itself myopic} at
each time slot, the \textit{expected network QoS exhibits an increasing
trend across time slots} (in other words, the expected QoS increases
in time, within a small positive slack), under optimal decision making
at the relays. Lastly, we present representative numerical simulations,
experimentally confirming both the efficacy and feasibility of the
proposed approach, as well as the validity of our theoretical predictions.

During exposition of the proposed spatially controlled relay beamforming
system, we concurrently develop and utilize a rigorous discussion
concerning the optimality of our approach, and with interesting results
(Section \ref{subsec:Appendix-B:-Measurability} / Appendix B). Clearly,
our problem formulation is challenging; it involves a \textit{variational}
stochastic optimization problem, where, at each time slot, the decision
variable, \textit{a function }of the so far available useful information
in the system (also called a \textit{policy}, or a \textit{decision
rule}), \textit{constitutes itself the spatial coordinates, from which
every network relay will observe the underlying spatiotemporal channel
field, at the next time slot}. In other words, our formulation requires
solving an (\textit{myopic}, in particular) \textit{optimal spatial
field sampling problem, in a dynamic fashion}. Such a problem raises
certain fundamental questions, not only related to our proposed spatially
controlled beamforming formulation, but also to a large class of variational
stochastic programs of similar structure.

In this respect, our contributions are partially driven by assuming
an underlying complete base probability space of otherwise \textit{arbitrary
structure}, generating all random phenomena considered in this work.
Under this general setting, we explicitly identify sufficient conditions,
which guarantee the validity of the so-called \textit{substitution
rule for conditional expectations}, specialized to such expectations
of random spatial (in general) fields/functions \textit{with an also
random spatial parameter}, relative to some $\sigma$-algebra, which
makes the latter parameter measurable (fixed) (Definition \ref{def:SUB}
\& Theorem \ref{thm:REP_EXP}). General validity of the substitution
rule, without imposing additional, special conditions, traces back
to the existence of regular conditional distributions, defined \textit{directly}
on the sample space of the underlying base probability space. Such
regular conditional distributions cannot be guaranteed to exist, unless
the sample space has nice topological properties, for instance, if
it is Polish \cite{Durrett2010probability}. In the context of our
spatially controlled beamforming application, such structural requirements
on the sample space, which, by assumption, is conceived as a model
of ``nature'', and generates the spatiotemporal channel field sampled
by the relays, are simply not reasonable. Considering this, our first
contribution is to show that it is possible to guarantee the validity
of the form of the substitution rule under consideration by imposing
conditions on the topological structure of the involved random field,
rather than that of the sample space (a part of its domain). This
results in a rather generally applicable problem setting (Theorem
\ref{thm:REP_EXP}).

In this work, the validity of the substitution rule is ascertained
by imposing simple continuity assumptions on the random functions
involved, which, in some cases, might be considered somewhat restrictive.
Nevertheless, those assumptions can be significantly weakened, guaranteeing
the validity of the substitution rule for vastly discontinuous random
functions, including, for instance, cases with random discontinuities,
or random jumps. The development of this extended analysis, though,
is out of the scope of this paper, and will be presented elsewhere.

The validity of the substitution rule is vitally important in the
treatment of a wide class of variational stochastic programs, including
that involved in the proposed spatially controlled beamforming approach.
In particular, leveraging the power of the substitution rule, we develop
a version of the so-called \textit{Fundamental Lemma of Stochastic
Control} \textit{(FLSC)} \cite{Speyer2008STOCHASTIC,Astrom1970CONTROL,Rockafellar2009VarAn,Shapiro2009STOCH_PROG,Bertsekas_Vol_2,Bertsekas1978Stochastic}
(Lemma \ref{lem:FUND_Lemma}), which provides sufficient conditions
that permit interchange of integration (expectation) and max/minimization
in general variational (stochastic) programming settings. The FLSC
allows the initial variational problem to be \textit{exchanged} by
a related, though \textit{pointwise} (ordinary) optimization problem,
thus efficiently reducing the search over a class of functions (initial
problem) to searching over constants, which is, of course, a standard
and much more handleable optimization setting. In slightly different
ways, the FLSC is evidently utilized in relevant optimality analysis
both in Stochastic Programming \cite{Rockafellar2009VarAn,Shapiro2009STOCH_PROG},
and in Dynamic Programming \& Stochastic Optimal Control \cite{Speyer2008STOCHASTIC,Astrom1970CONTROL,Bertsekas_Vol_2,Bertsekas1978Stochastic}.

A very general version of the FLSC is given in (\cite{Rockafellar2009VarAn},
Theorem 14.60), where unconstrained variational optimization of integrals
of extended real-valued \textit{random lower semicontinuous functions}
\cite{Shapiro2009STOCH_PROG}, or, by another name, \textit{normal
integrands }\cite{Rockafellar2009VarAn}, with respect to a general
$\sigma$-finite measure, is considered. Our version of the FLSC may
be considered a useful variation of Theorem 14.60 in \cite{Rockafellar2009VarAn},
and considers \textit{constrained} variational optimization problems
involving integrals of random functions, but with respect to some
base \textit{probability} measure (that is, expectations). In our
result, via the tower property of expectations, the role of the normal
integrand in (\cite{Rockafellar2009VarAn}, Theorem 14.60) is played
by the conditional expectation of the random function considered,
relative to a $\sigma$-algebra, which makes the respective decision
variable of the problem (a function(al)) measurable. Assuming a base
probability space of arbitrary structure, this argument is justified
by assuming validity of the substitution rule, which, in turn, is
ascertained under our previously developed sufficient conditions.
Different from (\cite{Rockafellar2009VarAn}, Theorem 14.60), in our
version of the FLSC, apart from natural Borel measurability requirements,
\textit{no continuity assumptions are directly imposed} on the structure
on either the random function, or the respective conditional expectation.
In this respect, our result extends (\cite{Rockafellar2009VarAn},
Theorem 14.60), and is of independent interest. 

On the other hand, from the strongly related perspective of Stochastic
Optimal Control, our version of the FLSC may be considered as the
basic building block for further development of Bellman Equation-type,
Dynamic Programming solutions \cite{Bertsekas1978Stochastic,Bertsekas_Vol_2},
under a strictly Borel measurability framework, sufficient for our
purposes. Quite differently though, in our formulation, the respective
cost (at each stage of the problem) is itself a random function (a
spatial field), whose domain is the Cartesian product of a base space
of arbitrary topology, with another, nicely behaved Borel space, instead
of the usual Cartesian product of two Borel spaces (the spaces of
state and controls), as in the standard dynamic programming setting
\cite{Bertsekas1978Stochastic,Bertsekas_Vol_2}. Essentially, our
formulation is ``one step back'' as compared to the basic dynamic
programming model of \cite{Bertsekas1978Stochastic,Bertsekas_Vol_2},
in the sense that the cost considered herein refers \textit{directly}
back to the base space. As a result, different treatment of the problem
is required; essentially, the validity of the substitution rule for
our cost function bypasses the requirement for existence of conditional
distributions, and exploits potential nice properties of the respective
conditional cost (in our case, joint \textit{Borel} measurability).

\textit{Emphasizing on our particular problem formulation}, our functional
assumptions, which guarantee the validity of the substitution rule,
combined with the FLSC, result in a total of six sufficient conditions,
under which strict optimality via problem exchangeability is guaranteed
(conditions \textbf{C1-C6} in Lemma \ref{lem:FUND_Lemma_FINAL}).
Those conditions are subsequently shown to be satisfied specifically
for the spatially controlled beamforming problem under consideration
(verification Theorem \ref{lem:C1C4_SAT}), ensuring strict optimality
of a solution obtained by exploiting problem exchangeability.

Finally, motivated by the need to provide performance guarantees for
the proposed myopic stochastic decision making scheme (our spatially
controlled beamforming network), we introduce the concept of a \textit{linear
martingale difference generator} spatiotemporal field. We then rigorously
show that, when such fields are involved in the objective of a \textit{myopic}
stochastic sampling scheme, \textit{stagewise myopic stochastic exploration
of the involved field is, under conditions, either monotonic, or quasi-monotonic}
(that is, monotonic within some small positive slack)\textit{, either
under optimal sampling decision making, or} \textit{when retaining
the same policy next}. This result is the basis for providing performance
guarantees for the proposed spatially controlled relay beamforming
system, as briefly stated above.

\textit{Notation} \textit{(some and basic)}: Matrices and vectors
will be denoted by boldface uppercase and boldface lowercase letters,
respectively. Calligraphic letters and formal script letters will
denote sets and $\sigma$-algebras, respectively. The operators $\left(\cdot\right)^{\boldsymbol{T}}$
and $\left(\cdot\right)^{\boldsymbol{H}}$, $\lambda_{min}\left(\cdot\right)$
and $\lambda_{max}\left(\cdot\right)$ will denote transposition,
conjugate transposition, minimum and maximum eigenvalue, respectively.
The $\ell_{p}$-norm of $\boldsymbol{x}\in\mathbb{R}^{n}$ is $\left\Vert \boldsymbol{x}\right\Vert _{p}\triangleq\left(\sum_{i=1}^{n}\left|x\left(i\right)\right|^{p}\right)^{1/p}$,
for all $\mathbb{N}\ni p\ge1$. For any $\mathbb{N}\ni N\ge1$, $\mathbb{S}^{N}$,
$\mathbb{S}_{+}^{N}$, $\mathbb{S}_{++}^{N}$ will denote the sets
of symmetric, symmetric positive semidefinite and symmetric positive
definite matrices, respectively. The finite $N$-dimensional identity
operator will be denoted as ${\bf I}_{N}$. Additionally, we define
$\mathfrak{J}\triangleq\sqrt{-1}$, $\mathbb{N}^{+}\triangleq\left\{ 1,2,\ldots\right\} $,
$\mathbb{N}_{n}^{+}\triangleq\left\{ 1,2,\ldots,n\right\} $, $\mathbb{N}_{n}\triangleq\left\{ 0\right\} \cup\mathbb{N}_{n}^{+}$
and $\mathbb{N}_{n}^{m}\triangleq\mathbb{N}_{n}^{+}\setminus\mathbb{N}_{m-1}^{+}$,
for positive naturals $n>m$.

\section{System Model}

On a compact, square planar region ${\cal W}\subset\mathbb{R}^{2}$,
we consider a wireless cooperative network consisting of one source,
one destination and $R\in\mathbb{N}^{+}$ assistive relays, as shown
in Fig. \ref{fig:System_Model}. Each entity of the network is equipped
with a single antenna, being able for both information reception and
broadcasting/transmission. The source and destination are stationary
and located at ${\bf p}_{S}\in{\cal W}$ and ${\bf p}_{D}\in{\cal W}$,
respectively, whereas the relays are assumed to be mobile; each relay
$i\in\mathbb{N}_{R}^{+}$ moves along a trajectory ${\bf p}_{i}\left(t\right)\in{\cal S}\subset{\cal W}-\left\{ {\bf p}_{S},{\bf p}_{D}\right\} \subset{\cal W}$,
where, in general, $t\in\mathbb{R}_{+}$, and where ${\cal S}$ is
compact. We also define the supervector ${\bf p}\left(t\right)\triangleq\left[{\bf p}_{1}^{\boldsymbol{T}}\left(t\right)\,{\bf p}_{2}^{\boldsymbol{T}}\left(t\right)\,\ldots\,{\bf p}_{R}^{\boldsymbol{T}}\left(t\right)\right]^{\boldsymbol{T}}\in{\cal S}^{R}\subset\mathbb{R}^{2R\times1}$.
Additionally, we assume that the relays can cooperate with each other,
either by exchanging local messages, or by communicating with a local
fusion center, through a dedicated channel. Hereafter, as already
stated above, all probabilistic arguments made below presume the existence
of a complete base probability space of otherwise completely arbitrary
structure, prespecified by a triplet $\left(\Omega,\mathscr{F},{\cal P}\right)$.
This base space models a universal source of randomness, generating
all stochastic phenomena in our considerations.

Assuming that a direct link between the source and the destination
does not exist, the role of the relays is determined to be assistive
to the communication, operating in a classical, two phase AF relaying
mode. Fix a $T>0$, and \textit{divide the time interval $\left[0,T\right]$
into $N_{T}$ time slots}, with $t\in\mathbb{N}_{N_{T}}^{+}$ denoting
the respective time slot. Let $s\left(t\right)\in\mathbb{C}$, with
$\mathbb{E}\left\{ \left|s\left(t\right)\right|^{2}\right\} \equiv1$,
denote the symbol to be transmitted at time slot $t$. Also, assuming
a flat fading channel model, as well as channel reciprocity and quasistaticity
in each time slot, let the sets $\left\{ f_{i}\left(t\right)\in\mathbb{C}\right\} _{i\in\mathbb{N}_{R}^{+}}$
and $\left\{ g_{i}\left(t\right)\in\mathbb{C}\right\} _{i\in\mathbb{N}_{R}^{+}}$
contain the \textit{random, spatiotemporally varying} source-relay
and relay-destination channel gains, respectively. These are further
assumed to be \textit{evaluations} of the \textit{separable random
channel fields} or \textit{maps} $f\left({\bf p},t\right)$ and $g\left({\bf p},t\right)$,
respectively, that is, $f_{i}\left(t\right)\equiv f\left({\bf p}_{i}\left(t\right),t\right)$
and $g_{i}\left(t\right)\equiv g\left({\bf p}_{i}\left(t\right),t\right)$,
for all $i\in\mathbb{N}_{R}^{+}$ and for all $t\in\mathbb{N}_{N_{T}}^{+}$.
Then, if $P_{0}>0$ denotes the transmission power of the source,
during AF phase $1$, the signals received at the relays can be expressed
as
\begin{equation}
r_{i}\hspace{-2pt}\left(t\right)\hspace{-2pt}\triangleq\hspace{-2pt}\sqrt{P_{0}}f_{i}\hspace{-2pt}\left(t\right)s\hspace{-2pt}\left(t\right)+n_{i}\hspace{-2pt}\left(t\right)\in\mathbb{C},
\end{equation}
for all $i\in\mathbb{N}_{R}^{+}$ and for all $t\in\mathbb{N}_{N_{T}}^{+}$,
where $n_{i}\left(t\right)\in\mathbb{C}$, with $\mathbb{E}\left\{ \left|n_{i}\left(t\right)\right|^{2}\right\} \equiv\sigma^{2},$
constitutes a zero mean observation noise process at the $i$-th relay,
independent across relays.
\begin{figure}
\centering\includegraphics[scale=1.77]{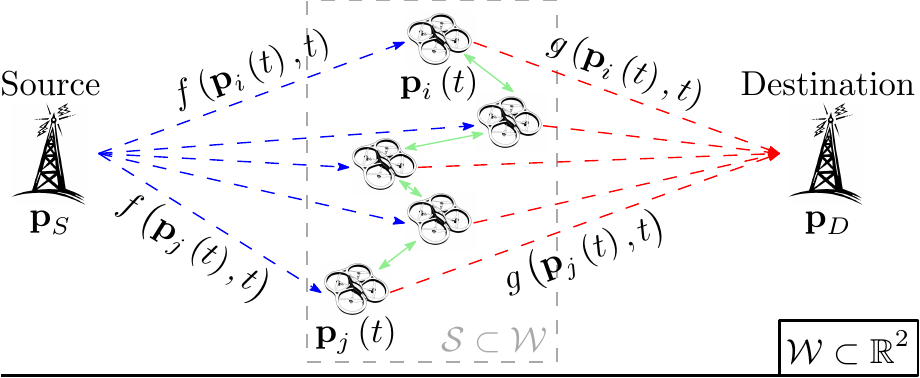}

\caption{\label{fig:System_Model}A schematic of the system model considered.}
\end{figure}
During AF phase $2$, all relays simultaneously retransmit the information
received, each modulating their received signal by a weight $w_{i}\left(t\right)\in\mathbb{C},i\in\mathbb{N}_{R}^{+}$.
The signal received at the destination can be expressed as 
\begin{flalign}
y\left(t\right) & \hspace{-2pt}\triangleq\hspace{-2pt}\sqrt{P_{0}}\hspace{-2pt}\sum_{i\in\mathbb{N}_{R}^{+}}\hspace{-2pt}w_{i}\hspace{-2pt}\left(t\right)g_{i}\hspace{-2pt}\left(t\right)r_{i}\hspace{-2pt}\left(t\right)\nonumber \\
 & \hspace{-2pt}\equiv\hspace{-2pt}\underbrace{\sqrt{P_{0}}\hspace{-2pt}\sum_{i\in\mathbb{N}_{R}^{+}}\hspace{-2pt}w_{i}\hspace{-2pt}\left(t\right)g_{i}\hspace{-2pt}\left(t\right)f_{i}\hspace{-2pt}\left(t\right)s\hspace{-2pt}\left(t\right)}_{\text{signal (transformed)}}+\underbrace{\sum_{i\in\mathbb{N}_{R}^{+}}\hspace{-2pt}w_{i}\hspace{-2pt}\left(t\right)g_{i}\hspace{-2pt}\left(t\right)n_{i}\hspace{-2pt}\left(t\right)+n_{D}\hspace{-2pt}\left(t\right)}_{\text{interference + reception noise}}\in\mathbb{C},\label{eq:MODEL}
\end{flalign}
for all $i\in\mathbb{N}_{R}^{+}$ and $t\in\mathbb{N}_{N_{T}}^{+}$,
where $n_{D}\left(t\right)\in\mathbb{C}$, with $\mathbb{E}\left\{ \left|n_{D}\left(t\right)\right|^{2}\right\} \equiv\sigma_{D}^{2},$
constitutes a zero mean, spatiotemporally white noise process at the
destination.

In the following, it is assumed that the channel fields $f\left({\bf p},t\right)$
and $g\left({\bf p},t\right)$ may be \textit{statistically dependent
both spatially and temporally}, and that, \textit{as usual}, the processes
$s\left(t\right)$, $\left[f\left({\bf p},t\right)\,g\left({\bf p},t\right)\right]$,
$n_{i}\left(t\right)$ for all $i\in\mathbb{N}_{R}^{+}$, and $n_{D}\left(t\right)$
are mutually independent. Also, we will assume that, at each time
slot $t,$ CSI $\left\{ f_{i}\left(t\right)\right\} _{i\in\mathbb{N}_{R}^{+}}$
and $\left\{ g_{i}\left(t\right)\right\} _{i\in\mathbb{N}_{R}^{+}}$
is known \textit{exactly} to all relays. This may be achieved through
pilot based estimation.

\section{\label{sec:Spatiotemporal-Wireless-Channel}Spatiotemporal Wireless
Channel Modeling}

This section introduces a general stochastic model for describing
the spatiotemporal evolution of the wireless channel. For the benefit
of the reader, a more intuitive justification of this general model
is also provided. Additionally, some extensions to the model are briefly
discussed, highlighting its versatility, along with some technical
considerations, which will be of importance later, for analyzing the
theoretical consistency of the subsequently proposed techniques.

\subsection{\label{subsec:Large-Scale-Gaussian}Large Scale Gaussian Channel
Modeling in the $dB$ Domain}

\noindent At each space-time point $\left({\bf p},t\right)\in{\cal S}\times\mathbb{N}_{N_{T}}^{+}$,
the source-relay channel field may be decomposed as the product of
three space-time varying components \cite{Goldsmith2005Wireless},
as
\begin{align}
f\left({\bf p},t\right) & \equiv\underbrace{f^{PL}\hspace{-2pt}\left({\bf p}\right)}_{\text{path loss}}\underbrace{f^{SH}\hspace{-2pt}\left({\bf p},t\right)}_{\text{shadowing}}\underbrace{f^{MF}\hspace{-2pt}\left({\bf p},t\right)}_{\text{fading}}e^{\mathfrak{J}\frac{2\pi\left\Vert {\bf p}-{\bf p}_{S}\right\Vert _{2}}{\lambda}},\label{eq:Channel_1}
\end{align}
where $\mathfrak{J}\triangleq\sqrt{-1}$ denotes the imaginary unit,
$\lambda>0$ denotes the wavelength employed for the communication,
and:
\begin{description}[style =standard, labelindent=0em , labelwidth=0.31cm, leftmargin =!]
\item [{1.}] $f^{PL}\left({\bf p}\right)\triangleq\left\Vert {\bf p}-{\bf p}_{S}\right\Vert _{2}^{-\ell/2}$
is the \textit{path loss field}, a deterministic quantity, with $\ell>0$
being the path loss exponent.\textbf{ }
\item [{2.}] $f^{SH}\left({\bf p},t\right)\in\mathbb{R}$ is the \textit{shadowing
field}, whose square is, for each $\left({\bf p},t\right)\in{\cal S}\times\mathbb{N}_{N_{T}}^{+}$,
a base-$10$ log-normal random variable with zero location. 
\item [{3.}] $f^{MF}\left({\bf p},t\right)\in\mathbb{C}$ constitutes the
\textit{multipath fading field}, a stationary process with known statistics.
\end{description}
The same decomposition holds in direct correspondence for the relay-destination
channel field, $g\left({\bf p},t\right)$. Additionally, if ``$\perp\!\!\!\perp$''
means ``is statistically independent of'', it is assumed that \cite{MostofiSpatial2012}
\begin{flalign}
\left[f^{MF}\left({\bf p},t\right)\,g^{MF}\left({\bf p},t\right)\right] & \perp\!\!\!\perp\left[f^{SH}\left({\bf p},t\right)\,g^{SH}\left({\bf p},t\right)\right]\quad\text{and}\\
f^{MF}\left({\bf p},t\right) & \perp\!\!\!\perp g^{MF}\left({\bf p},t\right).
\end{flalign}
In particular, if the phase of $f^{MF}\left({\bf p},t\right)$ is
denoted as $\phi_{f}\left({\bf p},t\right)\in\left[-\pi,\pi\right]$,
is further assumed that
\begin{equation}
\left|f^{MF}\left({\bf p},t\right)\right|\perp\!\!\!\perp\phi_{f}\left({\bf p},t\right),
\end{equation}
and the same for for $g^{MF}\left({\bf p},t\right)$. It also follows
that
\begin{equation}
\left[\left|f^{MF}\left({\bf p},t\right)\right|\,\left|g^{MF}\left({\bf p},t\right)\right|\right]\perp\!\!\!\perp\left[f^{SH}\left({\bf p},t\right)\,g^{SH}\left({\bf p},t\right)\right].
\end{equation}

We are interested in the magnitudes of both fields $f\left({\bf p},t\right)$
and $g\left({\bf p},t\right)$. Instead of working with the multiplicative
model described by (\ref{eq:Channel_1}), it is much preferable to
work in logarithmic scale. We may define the \textit{log-scale magnitude
field}
\begin{align}
F\left({\bf p},t\right) & \triangleq\alpha_{S}\left({\bf p}\right)\ell+\sigma_{S}\left({\bf p},t\right)+\xi_{S}\left({\bf p},t\right),\label{eq:Amplitude_log}
\end{align}
where we define
\begin{flalign}
-\alpha_{S}\left({\bf p}\right) & \triangleq10\log_{10}\left(\left\Vert {\bf p}-{\bf p}_{S}\right\Vert _{2}\right),\\
\sigma_{S}\left({\bf p},t\right) & \triangleq10\log_{10}\left(f^{SH}\left({\bf p},t\right)\right)^{2}\quad\text{and}\\
\xi_{S}\left({\bf p},t\right) & \triangleq10\log_{10}\left|f^{MF}\left({\bf p},t\right)\right|^{2}-\rho,\quad\text{with}\\
\rho & \triangleq\mathbb{E}\left\{ 10\log_{10}\left|f^{MF}\left({\bf p},t\right)\right|^{2}\right\} ,
\end{flalign}
for all $\left({\bf p},t\right)\in{\cal S}\times\mathbb{N}_{N_{T}}^{+}$.
It is then trivial to show that the magnitude of $f\left({\bf p},t\right)$
may be reconstructed via the \textit{bijective} formula 
\begin{align}
\left|f\left({\bf p},t\right)\right| & \equiv10^{\rho/20}\exp\left(\dfrac{\log\left(10\right)}{20}F\left({\bf p},t\right)\right),\label{eq:CONVERTER}
\end{align}
for all $\left({\bf p},t\right)\in{\cal S}\times\mathbb{N}_{N_{T}}^{+}$,
a ``trick'' that will prove very useful in the next section. Regarding
$g\left({\bf p},t\right)$, the log-scale field $G\left({\bf p},t\right)$
is defined in the same fashion, but replacing the subscript ``$S$''
by ``$D$''.

For each relay $i\in\mathbb{N}_{R}^{+}$, let us define the respective
log-scale channel magnitude processes $F_{i}\left(t\right)\triangleq F\left({\bf p}_{i}\left(t\right),t\right)$
and $G_{i}\left(t\right)\triangleq G\left({\bf p}_{i}\left(t\right),t\right)$,
for all $t\in\mathbb{N}_{N_{T}}^{+}$. Of course, we may stack all
the $F_{i}\left(t\right)$'s defined in (\ref{eq:Amplitude_log}),
resulting in the vector additive model
\begin{equation}
\boldsymbol{F}\left(t\right)\triangleq\boldsymbol{\alpha}_{S}\left({\bf p}\left(t\right)\right)\ell+\boldsymbol{\sigma}_{S}\left(t\right)+\boldsymbol{\xi}_{S}\left(t\right)\in\mathbb{R}^{R\times1},\label{eq:Vector_1}
\end{equation}
where $\boldsymbol{\alpha}_{S}\left(t\right)$, $\boldsymbol{\sigma}_{S}\left(t\right)$
and $\boldsymbol{\xi}_{S}\left(t\right)$ are defined accordingly.
We can also define $\boldsymbol{G}\left(t\right)\triangleq\boldsymbol{\alpha}_{D}\left({\bf p}\left(t\right)\right)\ell+\boldsymbol{\sigma}_{D}\left(t\right)+\boldsymbol{\xi}_{D}\left(t\right)\in\mathbb{R}^{R\times1},$
with each quantity in direct correspondence with (\ref{eq:Vector_1}).
We may also define, in the same manner, the log-scale shadowing and
multipath fading processes $\sigma_{S\left(D\right)}^{i}\left(t\right)\triangleq\sigma_{S\left(D\right)}\left({\bf p}_{i}\left(t\right),t\right)$
and $\xi_{S\left(D\right)}^{i}\left(t\right)\triangleq\xi_{S\left(D\right)}\left({\bf p}_{i}\left(t\right),t\right)$,
for all $t\in\mathbb{N}_{N_{T}}^{+}$, respectively.

Next, let us focus on the spatiotemporal dynamics of $\left\{ \left|f_{i}\left(t\right)\right|\right\} _{i}$
and $\left\{ \left|g_{i}\left(t\right)\right|\right\} _{i}$, which
are modeled through those of the shadowing components of $\left\{ F_{i}\left(t\right)\right\} _{i}$
and $\left\{ G_{i}\left(t\right)\right\} _{i}$. It is assumed that,
for any $N_{T}$ and any \textit{deterministic} ensemble of positions
of the relays in $\mathbb{N}_{N_{T}}^{+}$, say $\left\{ {\bf p}\left(t\right)\right\} _{t\in\mathbb{N}_{N_{T}}^{+}}$,
the random vector
\begin{equation}
\left[\boldsymbol{F}^{\boldsymbol{T}}\left(1\right)\,\boldsymbol{G}^{\boldsymbol{T}}\left(1\right)\,\ldots\,\boldsymbol{F}^{\boldsymbol{T}}\left(N_{T}\right)\,\boldsymbol{G}^{\boldsymbol{T}}\left(N_{T}\right)\right]^{\boldsymbol{T}}\in\mathbb{R}^{2RN_{T}\times1}\label{eq:Log_CSI_Process}
\end{equation}
is \textit{jointly Gaussian} with known means and known covariance
matrix. More specifically, on a per node basis, we let $\xi_{S\left(D\right)}^{i}\left(t\right)\overset{i.i.d.}{\sim}{\cal N}\left(0,\sigma_{\xi}^{2}\right)$
and $\sigma_{S\left(D\right)}^{i}\left(t\right)\overset{i.d.}{\sim}{\cal N}\left(0,\eta^{2}\right)$,
for all $t\in\mathbb{N}_{N_{T}}^{+}$ and $i\in\mathbb{N}_{R}^{+}$
\cite{MostofiSpatial2012,Cotton2007}. In particular, extending Gudmundson's
model \cite{Gudmundson1991} in a straightforward way, we propose
defining the spatiotemporal correlations of the shadowing part of
the channel as
\begin{align}
\mathbb{E}\left\{ \sigma_{S}^{i}\left(k\right)\sigma_{S}^{j}\left(l\right)\right\}  & \triangleq\eta^{2}\exp\left(-\frac{\left\Vert {\bf p}_{i}\left(k\right)-{\bf p}_{j}\left(l\right)\right\Vert _{2}}{\beta}-\frac{\left|k-l\right|}{\gamma}\right),\label{eq:first_EXP}
\end{align}
and correspondingly for $\left\{ \sigma_{D}^{i}\left(t\right)\right\} _{i\in\mathbb{N}_{R}^{+}}$,
and additionally,
\begin{align}
\mathbb{E}\left\{ \sigma_{S}^{i}\left(k\right)\sigma_{D}^{j}\left(l\right)\right\}  & \triangleq\mathbb{E}\left\{ \sigma_{S}^{i}\left(k\right)\sigma_{S}^{j}\left(l\right)\right\} \exp\left(-\frac{\left\Vert {\bf p}_{S}-{\bf p}_{D}\right\Vert _{2}}{\delta}\right),\label{eq:second_EXP}
\end{align}
for all $\left(i,j\right)\in\mathbb{N}_{R}^{+}\times\mathbb{N}_{R}^{+}$
and for all $\left(k,l\right)\in\mathbb{N}_{N_{T}}^{+}\times\mathbb{N}_{N_{T}}^{+}$.
In the above, $\eta^{2}>0$ and $\beta>0$ are called the \textit{shadowing
power} and the \textit{correlation distance}, respectively \cite{Gudmundson1991}.
In this fashion, we will call $\gamma>0$ and $\delta>0$ the \textit{correlation
time} and the\textit{ BS (Base Station) correlation}, respectively.
For later reference, let us define the (cross)covariance matrices
\begin{equation}
\boldsymbol{\Sigma}_{SD}\left(k,l\right)\triangleq\mathbb{E}\left\{ \boldsymbol{\sigma}_{S}\left(k\right)\boldsymbol{\sigma}_{D}^{\boldsymbol{T}}\left(l\right)\right\} +\mathds{1}_{\left\{ S\equiv D\right\} }\mathds{1}_{\left\{ k\equiv l\right\} }\sigma_{\xi}^{2}{\bf I}_{R}\in\mathbb{S}^{R},
\end{equation}
as well as
\begin{equation}
\boldsymbol{\Sigma}\left(k,l\right)\triangleq\begin{bmatrix}\boldsymbol{\Sigma}_{SS}\left(k,l\right) & \boldsymbol{\Sigma}_{SD}\left(k,l\right)\\
\boldsymbol{\Sigma}_{SD}\left(k,l\right) & \boldsymbol{\Sigma}_{DD}\left(k,l\right)
\end{bmatrix}\in\mathbb{S}^{2R},
\end{equation}
for all $\left(k,l\right)\in\mathbb{N}_{N_{T}}^{+}\times\mathbb{N}_{N_{T}}^{+}$.
Using these definitions, the covariance matrix of the joint distribution
describing (\ref{eq:Log_CSI_Process}) can be readily expressed as
\begin{equation}
\boldsymbol{\Sigma}\triangleq\begin{bmatrix}\boldsymbol{\Sigma}\left(1,1\right) & \boldsymbol{\Sigma}\left(1,2\right) & \ldots & \boldsymbol{\Sigma}\left(1,N_{T}\right)\\
\boldsymbol{\Sigma}\left(2,1\right) & \boldsymbol{\Sigma}\left(2,2\right) & \ldots & \boldsymbol{\Sigma}\left(2,N_{T}\right)\\
\vdots & \vdots & \ddots & \vdots\\
\boldsymbol{\Sigma}\left(N_{T},1\right) & \boldsymbol{\Sigma}\left(N_{T},2\right) & \cdots & \boldsymbol{\Sigma}\left(N_{T},N_{T}\right)
\end{bmatrix}\in\mathbb{S}^{2RN_{T}}.\label{eq:COV}
\end{equation}
Of course, in order for $\boldsymbol{\Sigma}$ to be a valid covariance
matrix, it must be at least positive semidefinite, that is, in $\mathbb{S}_{+}^{2RN_{T}}$.
If fact, for nearly all cases of interest, $\boldsymbol{\Sigma}$
is guaranteed to be strictly positive definite (or in $\mathbb{S}_{++}^{2RN_{T}}$),
as the following result suggests.
\begin{lem}
\textbf{\textup{(Positive (Semi)Definiteness of $\boldsymbol{\Sigma}$)\label{lem:NonSingularity}}}
For all possible \uline{deterministic} trajectories of the relays
on ${\cal S}^{R}\times\mathbb{N}_{N_{T}}^{+}$ , it is true that $\boldsymbol{\Sigma}\in\mathbb{S}_{++}^{2RN_{T}}$,
as long as $\sigma_{\xi}^{2}\neq0$. Otherwise, $\boldsymbol{\Sigma}\in\mathbb{S}_{+}^{2RN_{T}}$.
In other words, as long as multipath (small-scale) fading is present
in the channel response, the joint Gaussian distribution of the channel
vector in (\ref{eq:Log_CSI_Process}) is guaranteed to be nonsingular.
\end{lem}
\begin{proof}[Proof of Lemma \ref{lem:NonSingularity}]
See Appendix A.
\end{proof}

\subsection{Model Justification}

\begin{figure}
\centering\includegraphics[scale=1.8]{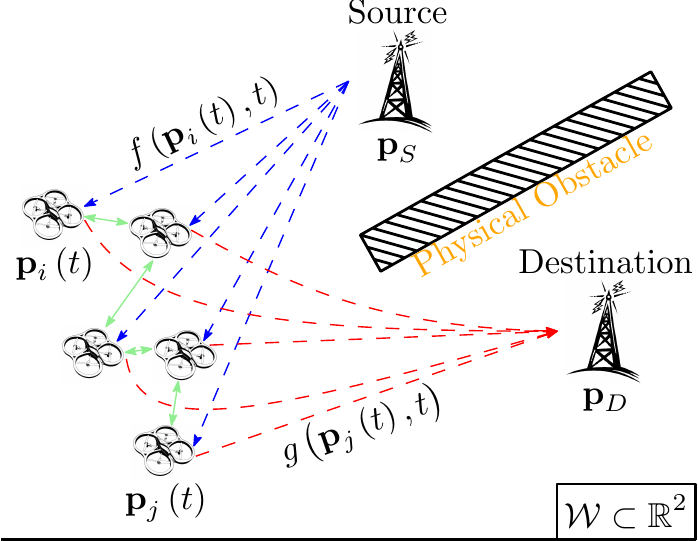}

\caption{\label{fig:System_Model_OBS-1}A case where source-relay and relay-destination
links are likely to be correlated.}
\end{figure}
As already mentioned, the spatial dependence among the source-relay
and relay-destination channel magnitudes (due to shadowing) is described
via Gudmundson's model \cite{Gudmundson1991} (position related component
in (\ref{eq:first_EXP})), which has been very popular in the literature
and also experimentally verified \cite{Gudmundson1991,Mostofi2_2012,MostofiSpatial2012}.
Second, the Laplacian type of temporal dependence among the same groups
of channel magnitudes also constitutes a reasonable choice, in the
sense that channel magnitudes are expected to be significantly correlated
only for small time lags, whereas, for larger time lags, such dependence
should decay at a fast rate. For an experimental justification of
the adopted model, see, for instance, \cite{Trappe2_2009,Channel_AR_2010}.
Also note that, this exponential temporal correlation model may result
as a reformulation of Gudmundson's model, as well. Of course, one
could use any other positive (semi)definite kernel, multiplying the
spatial correlation exponential kernel, without changing the statement
and proof of Lemma \ref{lem:NonSingularity}. Third, the incorporation
of the spherical/isotropic BS correlation term in our proposed general
model (in (\ref{eq:second_EXP})) can be justified by the the existence
of important cases where the source and destination might be close
to each other and yet no direct link may exist between them. See,
for instance, Fig. (\ref{fig:System_Model_OBS-1}), where a ``large''
physical obstacle makes the direct communication between the source
and the destination impossible. Then, relay beamforming can be exploited
in order to enable efficient communication between the source and
the destination, making intelligent use of the available resources,
in order to improve or maintain a certain QoS in the network. In such
cases, however, it is very likely that the shadowing parts of the
source-relay and relay-destination links will be spatially and/or
temporally correlated among each other, since shadowing is very much
affected by the spatial characteristics of the terrain, which, in
such cases, is common for both beamforming phases. Of course, by taking
the BS station correlation $\delta\rightarrow0$, one recovers the
generic/trivial case where the source-relay and relay-destination
links are \textit{mutually} \textit{independent}.

\subsection{\label{subsec:TECHNICAL_1}Extensions \& Some Technical Considerations}

It should be also mentioned that our general description of the wireless
channel as a spatiotemporal Gaussian field, does not limit the covariance
matrix $\boldsymbol{\Sigma}$ to be formed as in (\ref{eq:COV});
other choices for $\boldsymbol{\Sigma}$ will work fine in our subsequent
developments, as long as, \textit{for each fixed} $t\in\mathbb{N}_{N_{T}}^{+}$,
some mild conditions on the \textit{spatial interactions} of the fields\textit{
}$\sigma_{S\left(D\right)}\left({\bf p},t\right)$ and $\xi_{S\left(D\right)}\left({\bf p},t\right)$,
are satisfied. In what follows, we consider only the source-relay
fields $\sigma_{S}\left({\bf p},t\right)$ and $\xi_{S}\left({\bf p},t\right)$.
The same arguments hold for the relay-destination fields $\sigma_{D}\left({\bf p},t\right)$
and $\xi_{D}\left({\bf p},t\right)$, in direct correspondence.

Fix $t\in\mathbb{N}_{N_{T}}^{+}$. Recall that, so far, we have defined
the statistical behavior of both $\sigma_{S}\left({\bf p},t\right)$
and $\xi_{S}\left({\bf p},t\right)$ \textit{only on a per-node basis}.
However, since the spatiotemporal statistical model introduced in
Section \ref{subsec:Large-Scale-Gaussian} is assumed to be valid
for any possible trajectory of the relays in ${\cal S}^{R}\times\mathbb{N}_{N_{T}}^{+}$,
each relay is allowed to be anywhere in ${\cal S}$, at each time
slot $t$. This statistical construction \textit{induces} the statistical
structure (the \textit{laws}) of both fields $\sigma_{S}\left({\bf p},t\right)$
and $\xi_{S}\left({\bf p},t\right)$ on ${\cal S}$.

As far as $\sigma_{S}\left({\bf p},t\right)$ is concerned, it is
straightforward to see that it constitutes a Gaussian process with
zero mean, and a continuous and \textit{isotropic} covariance kernel
$\boldsymbol{\Sigma}_{\sigma}:\mathbb{R}^{2}\rightarrow\mathbb{R}$,
defined as
\begin{equation}
\boldsymbol{\Sigma}_{\sigma}\left(\boldsymbol{\tau}\right)\triangleq\eta^{2}\exp\left(-\frac{\left\Vert \boldsymbol{\tau}\right\Vert _{2}}{\beta}\right),\label{eq:S_sigma}
\end{equation}
where $\boldsymbol{\tau}\triangleq{\bf p}-{\bf q}\ge0$, for all $\left({\bf p},{\bf q}\right)\in{\cal S}^{2}$,
which agrees with the model introduced in (\ref{eq:first_EXP}), for
$k\equiv l$ (Gudmundson's model). Thus, $\sigma_{S}\left({\bf p},t\right)$
is a well defined random field.

However, this is not the case with $\xi_{S}\left({\bf p},t\right)$.
Under no additional restrictions, $\xi_{S}\left({\bf p},t\right)$
and $\xi_{S}\left({\bf q},t\right)$ are implicitly assumed to be
independent for all $\left({\bf p},{\bf q}\right)\in{\cal S}^{2}$,
such that ${\bf p}\neq{\bf q}$. Thus, we are led to consider $\xi_{S}\left({\bf p},t\right)$
as a zero-mean white process in continuous space. However, it is well
known that such a process is technically problematic in a measure
theoretic framework. Nevertheless, we may observe that it is \textit{not}
actually essential to characterize the covariance structure of $\xi_{S}\left({\bf p},t\right)$
\textit{for all} $\left({\bf p},{\bf q}\right)\in{\cal S}^{2}$, with
${\bf p}\neq{\bf q}$. This is due to the fact that, at each time
slot $t\in\mathbb{N}_{N_{T}}^{+}$, it is \textit{physically impossible}
for any two relays to be arbitrarily close to each other. We may thus
make the following simple assumption on the positions of the relays,
at each time slot $t\in\mathbb{N}_{N_{T}}^{+}$.
\begin{assumption}
\textbf{\textup{(Relay Separation)\label{assu:(Relay-Separation)}}}
There exists an $\varepsilon_{MF}>0$, such that, for all $t\in\mathbb{N}_{N_{T}}^{+}$
and any ensemble of relay positions at time slot $t$, $\left\{ {\bf p}_{i}\left(t\right)\right\} _{i\in\mathbb{N}_{R}^{+}}$,
it is true that
\begin{equation}
\inf_{\substack{\left(i,j\right)\in\mathbb{N}_{R}^{+}\times\mathbb{N}_{R}^{+}\\
\text{with }i\neq j
}
}\left\Vert {\bf p}_{i}\left(t\right)-{\bf p}_{j}\left(t\right)\right\Vert _{2}>\varepsilon_{MF}.\label{eq:=0003B5_constraint-1}
\end{equation}
\end{assumption}
Assumption \ref{assu:(Relay-Separation)} simply states that, at each
$t\in\mathbb{N}_{N_{T}}^{+}$, all relays are at least $\varepsilon_{MF}$
distance units apart from each other. If this constraint is satisfied,
then, without any loss of generality, we may define $\xi_{S}\left({\bf p},t\right)$
as a Gaussian field with zero mean, and with \textit{any} continuous,
isotropic (say) covariance kernel $\boldsymbol{\Sigma}_{\xi}:\mathbb{R}^{2}\rightarrow\mathbb{R}$,
which satisfies
\begin{equation}
\boldsymbol{\Sigma}_{\xi}\left(\boldsymbol{\tau}\right)\triangleq\begin{cases}
\sigma_{\xi}^{2}, & \text{if }\boldsymbol{\tau}\equiv{\bf 0}\\
0, & \text{if }\left\Vert \boldsymbol{\tau}\right\Vert _{2}\ge\varepsilon_{MF}
\end{cases},\label{eq:S_xi}
\end{equation}
and is arbitrarily defined otherwise. A simple example is the spherical,
compactly supported kernel with width $\varepsilon_{MF}$, defined
as \cite{Genton2001_Classes}
\begin{equation}
\dfrac{\boldsymbol{\Sigma}_{o}\left(\boldsymbol{\tau}\right)}{\sigma_{\xi}^{2}}\triangleq\begin{cases}
1-\dfrac{3}{2}\dfrac{\left\Vert \boldsymbol{\tau}\right\Vert _{2}}{\varepsilon_{MF}}+\dfrac{1}{2}\left(\dfrac{\left\Vert \boldsymbol{\tau}\right\Vert _{2}}{\varepsilon_{MF}}\right)^{3}, & \text{if }\left\Vert \boldsymbol{\tau}\right\Vert _{2}<\varepsilon_{MF}\\
0, & \text{if }\left\Vert \boldsymbol{\tau}\right\Vert _{2}\ge\varepsilon_{MF}
\end{cases}.\label{eq:S_o}
\end{equation}
Of course, across (discrete) time slots, $\xi_{S}\left({\bf p},t\right)$
inherits whiteness without any technical issue. 

We should stress that the above assumptions are made for technical
reasons and will be transparent in the subsequent analysis, as long
as the mild constraint (\ref{eq:=0003B5_constraint-1}) is satisfied;
from the perspective of the relays, all evaluations of $\xi_{S}\left({\bf p},t\right)$,
at each time slot, will be independent to each other. And, of course,
$\varepsilon_{MF}$ may be chosen small enough, such that (\ref{eq:=0003B5_constraint-1})
is satisfied virtually always, assuming that the relays are sufficiently
far apart from each other, and/or that, at each time slot $t$, their
new positions are relatively close to their old positions, at time
slot $t-1$.

Based on the explicit statistical description of $\sigma_{S}\left({\bf p},t\right)$
and $\xi_{S}\left({\bf p},t\right)$ presented above, we now additionally
demand that both are spatial fields with \textit{(everywhere) continuous
sample paths}. Equivalently, we demand that, for every $\omega\in\Omega$,
$\sigma_{S}\left(\omega,{\bf p},t\right)\in\mathsf{C}\left({\cal S}\right)$
and $\xi_{S}\left(\omega,{\bf p},t\right)\in\mathsf{C}\left({\cal S}\right)$
, where $\mathsf{C}\left({\cal A}\right)$ denotes the set of continuous
functions on some qualifying set ${\cal A}$. Sample path continuity
of stationary Gaussian fields may be guaranteed under mild conditions
on the respective lag-dependent covariance kernel, as the following
result suggests, however in a, slightly weaker, \textit{almost everywhere}
sense.
\begin{thm}
\textbf{\textup{($a.e.$-Continuity of Gaussian Fields \cite{Adler2010_Geometry,Abrahamsen1997_Review,Adler2009_Random})
\label{thm:Continuity-of-Gaussian}}}Let $X\left(\boldsymbol{s}\right)$,
$\boldsymbol{s}\in\mathbb{R}^{N}$, be a real-valued, zero-mean, stationary
Gaussian random field with a continuous covariance kernel $\boldsymbol{\Sigma}_{X}:\mathbb{R}^{N}\rightarrow\mathbb{R}$.
Suppose that there exist constants $0<c<+\infty$ and $\varepsilon,\zeta>0$,
such that
\begin{equation}
1-\dfrac{\boldsymbol{\Sigma}_{X}\left(\boldsymbol{\tau}\right)}{\boldsymbol{\Sigma}_{X}\left({\bf 0}\right)}\le\dfrac{c}{\left|\log\left(\left\Vert \boldsymbol{\tau}\right\Vert _{2}\right)\right|^{1+\varepsilon}},
\end{equation}
for all $\boldsymbol{\tau}\in\left\{ \left.\boldsymbol{x}\in\mathbb{R}^{N}\right|\left\Vert \boldsymbol{x}\right\Vert _{2}<\zeta\right\} $.
Then, $X\left(\boldsymbol{s}\right)$ is ${\cal P}$-almost everywhere
sample path continuous, or, equivalently, ${\cal P}-a.e.$-continuous,
on every compact subset ${\cal K}\subset\mathbb{R}^{N}$ and, therefore,
on $\mathbb{R}^{N}$ itself. Additionally, $X\left(\boldsymbol{s}\right)$
is bounded, ${\cal P}$-almost everywhere, as well.
\end{thm}
Utilizing Theorem \ref{thm:Continuity-of-Gaussian} and generically
assuming that $\boldsymbol{\Sigma}_{\xi}\triangleq\boldsymbol{\Sigma}_{o}$,
it is possible to show that both fields $\sigma_{S}\left({\bf p},t\right)$
and $\xi_{S}\left({\bf p},t\right)$ satisfy the respective conditions
and thus, that both fields are $a.e.$-continuous on ${\cal S}$.
For $\sigma_{S}\left({\bf p},t\right)$, the reader is referred to
(\cite{Abrahamsen1997_Review}, Example 2.2). Of course, instead of
$\boldsymbol{\Sigma}_{\sigma}$, any other kernel may be considered,
as long as the condition Theorem \ref{thm:Continuity-of-Gaussian}
is satisfied.

As far as $\xi_{S}\left({\bf p},t\right)$ is concerned, let us choose
$\varepsilon\equiv1$ and $\zeta\equiv1$ . We thus need to show that,
for every $\tau\triangleq\left\Vert \boldsymbol{\tau}\right\Vert _{2}\in\left[0,1\right)$,
it holds that
\begin{equation}
1-\dfrac{\boldsymbol{\Sigma}_{o}\left(\boldsymbol{\tau}\right)}{\boldsymbol{\Sigma}_{o}\left({\bf 0}\right)}\le\dfrac{c}{\left(\log\left(\left\Vert \boldsymbol{\tau}\right\Vert _{2}\right)\right)^{2}},
\end{equation}
or, equivalently,
\begin{equation}
1-\left(1-\dfrac{3}{2}\dfrac{\tau}{\varepsilon_{MF}}+\dfrac{1}{2}\left(\dfrac{\tau}{\varepsilon_{MF}}\right)^{3}\right)\mathds{1}_{\left\{ \tau<\varepsilon_{MF}\right\} }\le\dfrac{c}{\left(\log\left(\tau\right)\right)^{2}},\label{eq:SPH_C}
\end{equation}
for some finite, positive constant $c$. We first consider the case
where $1>\tau\ge\varepsilon_{MF}>0$ (whenever $\varepsilon_{MF}<1$,
of course). We then have
\begin{equation}
1\le\dfrac{\left(\log\left(\varepsilon_{MF}\right)\right)^{2}}{\left(\log\left(\tau\right)\right)^{2}}\triangleq\dfrac{c_{1}}{\left(\log\left(\tau\right)\right)^{2}},
\end{equation}
easily verifying the condition required by Theorem \ref{thm:Continuity-of-Gaussian}.
Now, when $0\le\tau<\min\left\{ \varepsilon_{MF},1\right\} $, it
is easy to see that there exists a finite $c_{2}>0$, such that
\begin{equation}
\tau\le\dfrac{c_{2}}{\left(\log\left(\tau\right)\right)^{2}}.\label{eq:tautau}
\end{equation}
If $\tau\equiv0$, then the inequality above holds for any choice
of $c_{2}$. If $\tau>0$, define a function $h:\left(0,1\right)\rightarrow\mathbb{R}_{+}$,
as
\begin{equation}
h\left(\tau\right)\triangleq\tau\left(\log\left(\tau\right)\right)^{2}.
\end{equation}
By a simple first derivative test, it follows that
\begin{flalign}
h\left(\tau\right) & \le\max_{\tau\in\left(0,1\right)}h\left(\tau\right)\nonumber \\
 & \equiv h\left(\exp\left(-2\right)\right)\nonumber \\
 & \equiv4\exp\left(-2\right),\quad\forall\tau\in\left(0,1\right).
\end{flalign}
Consequently, (\ref{eq:tautau}) is (loosely) satisfied for all $\tau\in\left[0,\min\left\{ \varepsilon_{MF},1\right\} \right)\subseteq\left(0,1\right)$,
by choosing $c_{2}\equiv4\exp\left(-2\right)$. Now, observe that
\begin{equation}
\dfrac{3}{2}\dfrac{\tau}{\varepsilon_{MF}}-\dfrac{1}{2}\left(\dfrac{\tau}{\varepsilon_{MF}}\right)^{3}<\dfrac{3}{2}\dfrac{\tau}{\varepsilon_{MF}}\le\dfrac{3c_{2}}{2\varepsilon_{MF}\left(\log\left(\tau\right)\right)^{2}}.
\end{equation}
Finally, simply choose 
\begin{flalign}
c & \equiv\max\left\{ c_{1},\dfrac{3c_{2}}{2\varepsilon_{MF}}\right\} \nonumber \\
 & \equiv\max\left\{ \left(\log\left(\varepsilon_{MF}\right)\right)^{2},\dfrac{6\exp\left(-2\right)}{\varepsilon_{MF}}\right\} <+\infty,
\end{flalign}
which immediately implies (\ref{eq:SPH_C}). Therefore, we have shown
that, if we choose $\boldsymbol{\Sigma}_{\xi}\equiv\boldsymbol{\Sigma}_{o}$
, then, for any fixed, but \textit{arbitrarily small} $\varepsilon_{MF}>0$,
the spatial field $\xi_{S}\left({\bf p},t\right)$ will also be almost
everywhere sample path continuous.

Observe that, via the analysis above, sample path continuity of the
involved fields can be ascertained, but only in the only almost everywhere
sense. Nevertheless, it easy to show that there always exist everywhere
sample path continuous fields $\widetilde{\sigma}_{S}\left({\bf p},t\right)$
and $\widetilde{\xi}_{S}\left({\bf p},t\right)$, which are \textit{indistinguishable}
from $\sigma_{S}\left({\bf p},t\right)$ and $\xi_{S}\left({\bf p},t\right)$,
respectively \cite{Elliott_2004Measure}. Therefore, there is absolutely
no loss of generality if we take both $\sigma_{S}\left({\bf p},t\right)$
and $\xi_{S}\left({\bf p},t\right)$ to be sample path continuous,
\textit{everywhere} in $\Omega$, and we will do so, hereafter.

Sample path continuity of all fields $\sigma_{S\left(D\right)}\left({\bf p},t\right)$
and $\xi_{S\left(D\right)}\left({\bf p},t\right)$ will be essential
in Section \ref{sec:MobRelBeam}, where we rigorously discuss optimality
of the proposed relay motion control framework, with special focus
on the relay beamforming problem. 

We close this section by discussing, in some more detail, the temporal
properties of the \textit{evaluations }of the fields $\sigma_{S}\left({\bf p},t\right)$
and $\sigma_{D}\left({\bf p},t\right)$ at any \textit{deterministic}
set of $N$ (say) positions $\left\{ {\bf p}_{i}\in{\cal S}\right\} _{i\in\mathbb{N}_{N}^{+}}$,
\textit{same} across all $N_{T}$ time slots. This results in the
zero-mean, stationary temporal Gaussian process
\begin{equation}
\boldsymbol{C}\left(t\right)\triangleq\left[\left\{ \sigma_{S}\left({\bf p}_{i},t\right)\right\} _{i\in\mathbb{N}_{N}^{+}}\,\left\{ \sigma_{D}\left({\bf p}_{i},t\right)\right\} _{i\in\mathbb{N}_{N}^{+}}\right]^{\boldsymbol{T}}\in\mathbb{R}^{2N\times1},\quad t\in\mathbb{N}_{N_{T}}^{+},\label{eq:Vector_Process}
\end{equation}
with \textit{matrix covariance kernel} $\boldsymbol{\Sigma}_{\boldsymbol{C}}:\mathbb{Z}\rightarrow\mathbb{S}_{+}^{2N}$,
defined, under the specific spatiotemporal model considered, as
\begin{equation}
\boldsymbol{\Sigma}_{\boldsymbol{C}}\left(\nu\right)\triangleq\exp\left(-\dfrac{\left|\nu\right|}{\gamma}\right)\widetilde{\boldsymbol{\Sigma}}_{\boldsymbol{C}}\in\text{\ensuremath{\mathbb{S}}}_{+}^{2N},
\end{equation}
where $\nu\triangleq t-s,$ for all $\left(t,s\right)\in\mathbb{N}_{N_{T}}^{+}\times\mathbb{N}_{N_{T}}^{+}$,
\begin{flalign}
\widetilde{\boldsymbol{\Sigma}}_{\boldsymbol{C}} & \triangleq\begin{bmatrix}1 & \kappa\\
\kappa & 1
\end{bmatrix}\varolessthan\widehat{\boldsymbol{\Sigma}}_{\boldsymbol{C}}\in\text{\ensuremath{\mathbb{S}}}_{+}^{2N},\\
\kappa & \triangleq\exp\left(-\dfrac{\left\Vert {\bf p}_{S}-{\bf p}_{D}\right\Vert _{2}}{\delta}\right)<1,\\
\widehat{\boldsymbol{\Sigma}}_{\boldsymbol{C}}\left(i,j\right) & \triangleq\boldsymbol{\Sigma}_{\sigma}\left({\bf p}_{i}-{\bf p}_{j}\right),\quad\forall\left(i,j\right)\in\mathbb{N}_{N}^{+}\times\mathbb{N}_{N}^{+},\label{eq:Vector_end}
\end{flalign}
and with ``$\varolessthan$'' denoting the operator of the Kronecker
product. Then, the following result is true. 
\begin{thm}
\textbf{\textup{($\boldsymbol{C}\left(t\right)$ is Markov)}}\label{thm:(Markov)}
For any deterministic, time invariant set of points $\left\{ {\bf p}_{i}\in{\cal S}\right\} _{i\in\mathbb{N}_{N}^{+}}$,
the vector process $\boldsymbol{C}\left(t\right)\in\mathbb{R}^{2N\times1}$,
$t\in\mathbb{N}_{N_{T}}^{+}$, as defined in (\ref{eq:Vector_Process})-(\ref{eq:Vector_end}),
may be represented as a stable order-$1$ vector autoregression, satisfying
the linear stochastic difference equation
\begin{flalign}
\boldsymbol{X}\left(t\right) & \equiv\varphi\boldsymbol{X}\left(t-1\right)+\boldsymbol{W}\left(t\right),\quad t\in\mathbb{N}_{N_{T}}^{+},\label{eq:AUTO_1}
\end{flalign}
where
\begin{flalign}
\varphi & \triangleq\exp\left(-1/\gamma\right)<1,\\
\boldsymbol{X}\left(0\right) & \sim{\cal N}\left({\bf 0},\widetilde{\boldsymbol{\Sigma}}_{\boldsymbol{C}}\right)\quad\text{and}\\
\boldsymbol{W}\left(t\right)\hspace{2pt}\hspace{2.35pt} & \hspace{-2pt}\hspace{-2.35pt}\overset{i.i.d.}{\sim}{\cal N}\left({\bf 0},\left(1-\varphi^{2}\right)\widetilde{\boldsymbol{\Sigma}}_{\boldsymbol{C}}\right),\quad\forall t\in\mathbb{N}_{N_{T}}^{+}.
\end{flalign}
In particular, $\boldsymbol{C}\left(t\right)$ is Markov.
\end{thm}
\begin{proof}[Proof of Theorem \ref{thm:(Markov)}]
The proof is a standard exercise in time series; see Appendix A.
\end{proof}
From a practical point of view, Theorem \ref{thm:(Markov)} is extremely
valuable. Specifically, the Markovian representation of $\boldsymbol{C}\left(t\right)$
may be employed in order to \textit{efficiently simulate} the spatiotemporal
paths of the communication channel on any finite, but arbitrarily
fine grid. This is important, since it allows detailed \textit{numerical
evaluation} of all methods developed in this work. Theorem \ref{thm:(Markov)}
also reveals that the channel model we have considered actually agrees
with experimental results presented in, for instance, \cite{Trappe2_2009,Channel_AR_2010},
which show that autoregressive processes constitute an adequate model
for stochastically describing temporal correlations among wireless
communication links.
\begin{rem}
Unfortunately, to the best of our knowledge, the channel process along
a specific relay trajectory, presented in Section \ref{subsec:Large-Scale-Gaussian},
where the positions of the relays are allowed to vary across time
slots is no longer stationary and may not be shown to satisfy the
Markov Property. Therefore, in our analysis presented hereafter, we
regard the aforementioned process as a general, nonstationary Gaussian
process. All inference results presented below are based on this generic
representation.\hfill{}\ensuremath{\blacksquare}
\end{rem}
\begin{rem}
For simplicity, all motion control problems in this paper are formulated
on the plane (some subset of $\mathbb{R}^{2}$). This means that any
motion of the relays of the network along the third dimension of the
space is indifferent to our channel model. Nevertheless, under appropriate
(based on the requirements discussed above) assumptions concerning
3D wireless channel modeling, all subsequent arguments would hold
in exactly the same fashion when fully unconstrained motion in $\mathbb{R}^{3}$
is assumed to affect the quality of the wireless channel.\hfill{}\ensuremath{\blacksquare}
\end{rem}

\section{\label{sec:MobRelBeam}Spatially Controlled Relay Beamforming}

In this section, we formulate and solve the spatially controlled relay
beamforming problem, advocated in this paper. The beamforming objective
adopted will be \textit{maximization of the Signal-to-Interference+Noise
Ratio (SINR) at the destination (measuring network QoS)}, under a
total power budget at the relays. For the single-source single-destination
setting considered herein, the aforementioned beamforming problem
admits a closed form solution, a fact which will be important in deriving
optimal relay motion control policies, in a tractable fashion. But
first, let us present the general scheduling schema of the proposed
mobile beamforming system, as well as some technical preliminaries
on stochastic programming and optimal control, which will be used
repeatedly in the analysis to follow.

\subsection{Joint Scheduling of Communications \& Controls}

\begin{figure}
\centering\includegraphics[scale=1.4]{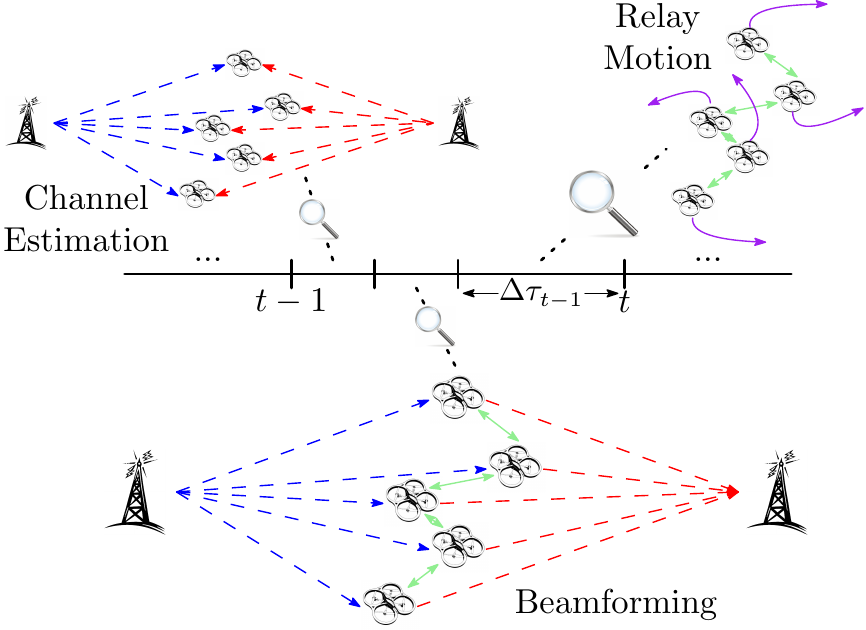}

\caption{\label{fig:Scheduling}Proposed TDMA-like joint scheduling protocol
for communications and controls.}
\end{figure}
At each time slot $t\in\mathbb{N}_{N_{T}}^{+}$ and assuming the same
carrier for all communication tasks, we employ a basic joint communication/decision
making TDMA-like protocol, as follows:
\begin{description}[style =standard, labelindent=0em , labelwidth=0.31cm, leftmargin =!]
\item [{1.}] The source broadcasts a pilot signal to the relays, which
then estimate their respective channels relative to the source.
\item [{2.}] The same procedure is carried out for the channels relative
to the destination.\textbf{ }
\item [{3.}] Then, based on the estimated CSI, the relays beamform in AF
mode (assume perfect CSI estimation).
\item [{4.}] Based on the CSI received \textit{so far}, strategic decision
making is implemented, motion controllers of the relays are determined
and relays are steered to their updated positions.
\end{description}
The above sequence of actions is repeated for all $N_{T}$ time slots,
corresponding to the total operational horizon of the system. This
simple scheduling protocol is graphically depicted in Fig. \ref{fig:Scheduling}.

Concerning relay kinematics, it is assumed that the relays obey the
differential equation
\begin{equation}
\dot{{\bf p}}\left(\tau\right)\equiv{\bf u}\left(\tau\right),\quad\forall\tau\in\left[0,T\right],
\end{equation}
where ${\bf u}\triangleq\left[{\bf u}_{1}\,\ldots\,{\bf u}_{R}\right]^{\boldsymbol{T}}\in{\cal S}^{R}$,
with ${\bf u}_{i}:\left[0,T\right]\rightarrow{\cal S}$ being the
motion controller of relay $i\in\mathbb{N}_{R}^{+}$. Apparently,
relay motion is in continuous time. However, assuming the relays may
move \textit{only after their controls have been determined and up
to the start of the next time slot}, we can write
\begin{equation}
{\bf p}\left(t\right)\equiv{\bf p}\left(t-1\right)+\int_{\Delta\tau_{t-1}}{\bf u}_{t-1}\left(\tau\right)\text{d}\tau,\quad\forall t\in\mathbb{N}_{N_{T}}^{2},\label{eq:motion_model_2}
\end{equation}
with ${\bf p}\left(1\right)\equiv{\bf p}{}_{init}$, and where $\Delta\tau_{t}\subset\mathbb{R}$
and ${\bf u}_{t}:\Delta\tau_{t}\rightarrow{\cal S}^{R}$ denote the
time interval that the relays are allowed to move in and the respective
relay controller, in each time slot $t\in\mathbb{N}_{N_{T}-1}^{+}$.
It holds that ${\bf u}\left(\tau\right)\equiv\sum_{t\in\mathbb{N}_{N_{T}-1}^{+}}{\bf u}_{t}\left(\tau\right)\mathds{1}_{\Delta\tau_{t}}\left(\tau\right)$,
\textit{where $\tau$ belongs in the first} $N_{T}-1$ time slots.Of
course, at each time slot $t$, the length of $\Delta\tau_{t}$, $\left|\Delta\tau_{t}\right|$,
must be sufficiently small such that the temporal correlations of
the CSI at adjacent time slots are sufficiently strong. These correlations
are controlled by the correlation time parameter $\gamma$, which
can be a function of the slot width. Therefore, the velocity of the
relays must be of the order of $\left(\left|\Delta\tau_{t}\right|\right)^{-1}$.
In this work, though, we assume that the relays are not explicitly
resource constrained, in terms of their motion.

Now, regarding the form of the relay motion controllers ${\bf u}_{t-1}\left(\tau\right),\tau\in\Delta\tau_{t-1}$,
\textit{given a goal position vector at time slot} $t$, ${\bf p}^{o}\left(t\right),$
it suffices to fix a path in ${\cal S}^{R}$, such that the points
${\bf p}^{o}\left(t\right)$ and ${\bf p}\left(t-1\right)$ are connected
in at most time $\left|\Delta\tau_{t-1}\right|$. A generic choice
for such a path is the straight line\footnote{Caution is needed here, due to the possibility of physical collisions
among relays themselves, or among relays and other physical obstacles
in the workspace, ${\cal S}$. Nevertheless, for simplicity, we assume
that either such events never occur, or that, if they do, there exists
some transparent collision avoidance mechanism implemented at each
relay, which is out of our direct control.} connecting ${\bf p}_{i}^{o}\left(t\right)$ and ${\bf p}_{i}\left(t-1\right)$,
for all $i\in\mathbb{N}_{R}^{+}$. Therefore, we may choose the relay
controllers at time slot $t-1\in\mathbb{N}_{N_{T}-1}^{+}$ as
\begin{equation}
{\bf u}_{t-1}^{o}\left(\tau\right)\triangleq\dfrac{1}{\Delta\tau_{t-1}}\left({\bf p}^{o}\left(t\right)-{\bf p}\left(t-1\right)\right),\quad\forall\tau\in\Delta\tau_{t-1}.
\end{equation}
As a result, any motion control problem considered hereafter can now
be formulated in terms of specifying the goal relay positions at the
next time slot, given their positions at the current time slot (and
the observed CSI).

In the following, let $\mathscr{C}\left({\cal T}_{t}\right)$ denote
the set of channel gains observed by the relays, \textit{along the
paths of their point trajectories} ${\cal T}_{t}\triangleq\left\{ {\bf p}\left(1\right)\,\ldots\,{\bf p}\left(t\right)\right\} $,
$t\in\mathbb{N}_{N_{T}}^{+}$. Then, ${\cal T}_{t}$ may be recursively
updated as ${\cal T}_{t}\equiv{\cal T}_{t-1}\cup\left\{ {\bf p}\left(t\right)\right\} $,
for all $t\in\mathbb{N}_{N_{T}}^{+}$, with ${\cal T}_{0}\triangleq\varnothing$.
In a technically precise sense, $\left\{ \mathscr{C}\left({\cal T}_{t}\right)\right\} _{t\in\mathbb{N}_{N_{T}}^{+}}$
will also denote the filtration generated by the CSI observed at the
relays, \textit{along} ${\cal T}_{t}$, interchangeably. In other
words, in case the trajectories of the relays are themselves random,
then $\mathscr{C}\left({\cal T}_{t}\right)$ denotes the $\sigma$-algebra
generated by \textit{both} the CSI observed up to and including time
slot $t$ \textit{and} ${\bf p}\left(1\right)\,\ldots\,{\bf p}\left(t\right)$,
for all $t\in\mathbb{N}_{N_{T}}^{+}$. Additionally, we define $\mathscr{C}\left({\cal T}_{0}\right)\equiv\mathscr{C}\left(\left\{ \varnothing\right\} \right)$
as $\mathscr{C}\left({\cal T}_{0}\right)\triangleq\left\{ \varnothing,\Omega\right\} $,
that is, as the trivial $\sigma$-algebra, and we may occasionally
refer to time $t\equiv0$, as a \textit{dummy time slot}, by convention. 

\subsection{\label{subsec:2-Stage}$2$-Stage Stochastic Optimization of Beamforming
Weights and Relay Positions: Base Formulation \& Methodology}

At each time slot $t\in\mathbb{N}_{N_{T}}^{+}$, given the current
CSI encoded in $\mathscr{C}\left({\cal T}_{t}\right)$, we are interested
in determining $\boldsymbol{w}^{o}\left(t\right)\triangleq\left[w_{1}\left(t\right)\,w_{2}\left(t\right)\,\ldots\,w_{R}\left(t\right)\right]^{\boldsymbol{T}}$,
as an optimal solution to a beamforming optimization problem, as a
functional of $\mathscr{C}\left({\cal T}_{t}\right)$. Let the \textit{optimal
value} (say infimum) of this problem be the process $V_{t}\equiv V\left({\bf p}\left(t\right),t\right)$,
a functional of the CSI encoded in $\mathscr{C}\left({\cal T}_{t}\right)$,
depending on the positions of the relays at time slot $t$.

Suppose that, \textit{at time slot} $t-1$, an oracle reveals $\mathscr{C}\left({\cal T}_{t}\equiv{\cal T}_{t-1}\cup\left\{ {\bf p}\left(t\right)\right\} \right)$,
which also determines the channels corresponding to the new positions
of the relays at the next time slot $t$. Then, we could further consider
optimizing $V_{t}$ with respect to ${\bf p}\left(t\right)$, representing
the new position of the relays. But note that, $\mathscr{C}\left({\cal T}_{t}\right)$
is not physically observable and in the absence of the oracle, optimizing
$V_{t}$ with respect to ${\bf p}\left(t\right)$ is impossible, since,
given $\mathscr{C}\left({\cal T}_{t-1}\right)$, the channels at any
position of the relays are nontrivial random variables. However, it
is reasonable to search for the best decision on the positions of
the relays at time slot $t$, as a functional of the available information
encoded in $\mathscr{C}\left({\cal T}_{t-1}\right)$, such that $V_{t}$
is optimized \textit{on average}. This procedure may be formally formulated
as a \textit{$2$-stage stochastic program} \cite{Shapiro2009STOCH_PROG},
\begin{equation}
\begin{array}{rl}
\underset{{\bf p}\left(t\right)}{\mathrm{minimize}} & \mathbb{E}\left\{ V\left({\bf p}\left(t\right),t\right)\right\} \\
\mathrm{subject\,to} & {\bf p}\left(t\right)\equiv{\cal M}\left(\mathscr{C}\left({\cal T}_{t-1}\right)\right)\in{\cal C}\left({\bf p}^{o}\left(t\hspace{-2pt}-\hspace{-2pt}1\right)\right),\\
 & \text{for some }{\cal M}:\mathbb{R}^{4R\left(t-1\right)}\rightarrow\mathbb{R}^{2R}
\end{array},\label{eq:2STAGE_PRE}
\end{equation}
to be solved at each $t-1\in\mathbb{N}_{N_{T}-1}^{+}$, where ${\cal C}:\mathbb{R}^{2R}\rightrightarrows\mathbb{R}^{2R}$
is a multifunction,
\begin{figure}
\centering\includegraphics[scale=1.3]{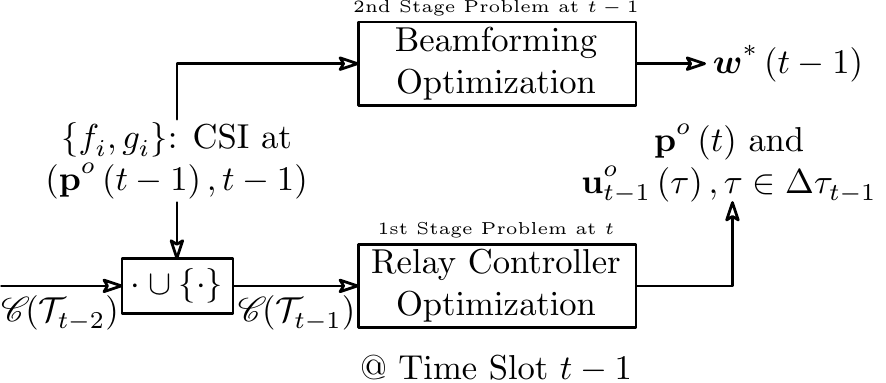}

\caption{\label{fig:Proposed_System}2-Stage optimization of beamforming weights
and spatial relay controllers. The variables $\boldsymbol{w}^{o}\left(t-1\right)$,
${\bf u}_{t-1}^{o}$ and ${\bf p}^{o}\left(t\right)$ denote the optimal
beamforming weights and relay controllers at time slot $t-1$, and
the optimal relay positions at time slot $t$, respectively.}
\end{figure}
 with ${\cal C}\left({\bf p}^{o}\left(t-1\right)\right)\subseteq{\cal S}^{R}$
representing a physically feasible spatial neighborhood around the
point ${\bf p}^{o}\left(t-1\right)\in{\cal S}^{R}$, the decision
vector \textit{selected} at time $t-2\in\mathbb{N}_{N_{T}-2}$ (recall
that $t\equiv0$ denotes a dummy time slot). Note that, in general,
the decision \textit{selected} at $t-2$, ${\bf p}^{o}\left(t-1\right)$,
may not be an optimal decision for the respective problem solved at
$t-2$ and implemented at $t-1$. To distinguish ${\bf p}^{o}\left(t-1\right)$
from an optimal decision at $t-2$, the latter will be denoted as
${\bf p}^{*}\left(t-1\right)$, for all $t\in\mathbb{N}_{N_{T}}^{2}$.
Also note that, in order for (\ref{eq:2STAGE_PRE}) to be well defined,
important technical issues, such as measurability of $V_{t}$ and
existence of its expectation at least for each feasible decision ${\bf p}\left(t\right)$,
should be precisely resolved. Problem (\ref{eq:2STAGE-1}), together
with the respective beamforming problem with optimal value $V_{t}$
(which will focus on shortly) are referred to as the \textit{first-stage
problem} and the \textit{second-stage problem}, respectively \cite{Shapiro2009STOCH_PROG}.
Hereafter, aligned with the literature, any feasible choice for the
decision variable ${\bf p}\left(t\right)$ in (\ref{eq:2STAGE-1}),
will be interchangeably called an \textit{(admissible)} \textit{policy}.
A generic block representation of the proposed $2$-stage stochastic
programming approach is depicted in Fig. \ref{fig:Proposed_System}.

Mainly due to the arbitrary structure of the function ${\cal M}$,
(\ref{eq:2STAGE_PRE}) is too general to consider, within a reasonable
analytical framework. Thus, let us slightly constrain the decision
set of (\ref{eq:2STAGE_PRE}) to include \textit{only measurable decisions},
resulting in the formulation

\begin{equation}
\begin{array}{rl}
\underset{{\bf p}\left(t\right)}{\mathrm{minimize}} & \mathbb{E}\left\{ V\left({\bf p}\left(t\right),t\right)\right\} \\
\mathrm{subject\,to} & {\bf p}\left(t\right)\equiv{\cal M}\left(\mathscr{C}\left({\cal T}_{t-1}\right)\right)\in{\cal C}\left({\bf p}^{o}\left(t\hspace{-2pt}-\hspace{-2pt}1\right)\right),\\
 & {\cal M}^{-1}\left({\cal A}\right)\in\mathscr{B}\left(\mathbb{R}^{4R\left(t-1\right)}\right),\,\forall{\cal A}\in\mathscr{B}\left(\mathbb{R}^{2R}\right)
\end{array},\label{eq:2STAGE-1}
\end{equation}
provided, of course, that the stochastic program (\ref{eq:2STAGE-1})
is well defined. The second constraint in (\ref{eq:2STAGE-1}) is
equivalent to ${\cal M}$ being Borel measurable, instead of being
any arbitrary function, as in (\ref{eq:2STAGE_PRE}).

Provided its well definiteness, the stochastic program (\ref{eq:2STAGE-1})
is difficult to solve, most importantly because of its \textit{variational
character}; the decision variable ${\bf p}\left(t\right)$ is constrained
to be a functional of the CSI observed up to and including time $t-1$.
A very powerful tool, which will enable us to both make (\ref{eq:2STAGE-1})
meaningful and overcome the aforementioned difficulty, is the \textit{Fundamental
Lemma of Stochastic Control} \cite{Speyer2008STOCHASTIC,Astrom1970CONTROL,Rockafellar2009VarAn,Shapiro2009STOCH_PROG,Bertsekas_Vol_2,Bertsekas1978Stochastic},
which in fact refers to a family of technical results related to the
\textit{interchangeability} of integration (expectation) and minimization
in general stochastic programming. Under the framework of the Fundamental
Lemma, in Appendix B, we present a detailed discussion, best suited
for the purposes of this paper, which is related to the important
technical issues, arising when one wishes to meaningfully define and
tractably simplify ``hard'', variational problems of the form of
(\ref{eq:2STAGE-1}).

In particular, Lemma \ref{lem:FUND_Lemma_FINAL}, presented in Section
\ref{subsec:FINAL} (Appendix B), identifies six sufficient technical
conditions (conditions \textbf{C1-C6}, see statement of Lemma \ref{lem:FUND_Lemma_FINAL}),
under which the variational problem (\ref{eq:2STAGE-1}) is \textit{exchangeable}
by the structurally simpler, \textit{pointwise} optimization problem
\begin{equation}
\begin{array}{rl}
\underset{{\bf p}\left(t\right)}{\mathrm{minimize}} & \mathbb{E}\left\{ V\left({\bf p}\left(t\right),t\right)\left|\mathscr{C}\left({\cal T}_{t-1}\right)\right.\right\} \\
\mathrm{subject\,to} & {\bf p}\left(t\right)\in{\cal C}\left({\bf p}^{o}\left(t\hspace{-2pt}-\hspace{-2pt}1\right)\right)
\end{array},\label{eq:2STAGE-2}
\end{equation}
to be solved at each $t-1\in\mathbb{N}_{N_{T}-1}^{+}$. Observe that,
in (\ref{eq:2STAGE-2}), the decision variable ${\bf p}\left(t\right)$
is constant, as opposed to (\ref{eq:2STAGE-1}), where the decision
variable ${\bf p}\left(t\right)$ is itself a functional of the observed
information at time slot $t-1$, that is, a policy. Provided that
CSI $\mathscr{C}\left({\cal T}_{t-1}\right)$ and ${\bf p}^{o}\left(t\hspace{-2pt}-\hspace{-2pt}1\right)$
are known and that the involved conditional expectation can be somehow
evaluated, (\ref{eq:2STAGE-2}) constitutes an ordinary, nonlinear
optimization problem.

If Lemma \ref{lem:FUND_Lemma_FINAL} is in power, \textit{exchangeability}
of (\ref{eq:2STAGE-1}) by (\ref{eq:2STAGE-2}) is understood in the
sense that the optimal value of (\ref{eq:2STAGE-1}), which is a number,
coincides with the \textit{expectation }of optimal value of (\ref{eq:2STAGE-2}),
which turns out to be a measurable function of $\mathscr{C}\left({\cal T}_{t-1}\right)$.
In other words, minimization is \textit{interchangeable} with integration,
in the sense that 
\begin{equation}
\inf_{{\bf p}\left(t\right)\in{\cal D}_{t}}\mathbb{E}\left\{ V\left({\bf p}\left(t\right),t\right)\right\} \equiv\mathbb{E}\left\{ \inf_{{\bf p}\left(t\right)\in{\cal C}\left({\bf p}^{o}\left(t-1\right)\right)}\mathbb{E}\left\{ V\left({\bf p}\left(t\right),t\right)\left|\mathscr{C}\left({\cal T}_{t-1}\right)\right.\right\} \right\} ,
\end{equation}
for all $t\in\mathbb{N}_{N_{T}}^{2}$, where ${\cal D}_{t}$ denotes
the set of feasible decisions for (\ref{eq:2STAGE-1}). What is more,
under the aforementioned technical conditions of Lemma \ref{lem:FUND_Lemma_FINAL},
exchangeability implies that, if there exists an admissible policy
of (\ref{eq:2STAGE-1}), say ${\bf p}^{*}\left(t\right)$, which solves
(\ref{eq:2STAGE-2}), then ${\bf p}^{*}\left(t\right)$ is also optimal
for (\ref{eq:2STAGE-1}). Additionally, Lemma \ref{lem:FUND_Lemma_FINAL}
implies existence of at least one optimal solution to (\ref{eq:2STAGE-2}),
which is simultaneously feasible and, thus, optimal, for the original
stochastic program (\ref{eq:2STAGE-1}). If, further, (\ref{eq:2STAGE-2})
features a unique optimal solution, say ${\bf p}^{*}\left(t\right)$,
then ${\bf p}^{*}\left(t\right)$ must be an optimal solution to (\ref{eq:2STAGE-1}).

In the next subsection, we will specify the optimal value of the second-stage
subproblem, $V_{t}$, for each time $t\in\mathbb{N}_{N_{T}}^{+}$.
That is, we will consider a fixed criterion for implementing relay
beamforming (recourse actions) at each $t$,\textit{ after} the predictive
decisions on the positions of the relays have been made (\textit{at
time $t-1$}) and the relays have moved to their new positions, implying
that the CSI at time at time $t$ has been revealed. Of course, one
of the involved challenges will be to explicitly show that Conditions
\textbf{C1-C6 }are satisfied for each case considered, so that we
can focus on solving the ordinary nonlinear optimization problem (\ref{eq:2STAGE-2}),
instead of the much more difficult variational problem (\ref{eq:2STAGE-1}).
The other challenge we will face is actually solving (\ref{eq:2STAGE-2}).
\begin{rem}
It would be important to note that the pointwise problem (\ref{eq:2STAGE-1})
admits a reasonable and intuitive interpretation: At each time slot
$t-1$, instead of (deterministically) optimizing $V_{t}$ with respect
to ${\bf p}\left(t\right)$ in ${\cal C}\left({\bf p}^{o}\left(t\hspace{-2pt}-\hspace{-2pt}1\right)\right)$,
which is, of course, impossible, one considers optimizing \textit{a
projection of $V\left({\bf p},t\right)$, ${\bf p}\in{\cal S}^{R}$
onto the space of all measurable functionals of }$\mathscr{C}\left({\cal T}_{t-1}\right)$,
which corresponds to the information observed by the relays, up to
$t-1$. Provided that, for every ${\bf p}\in{\cal S}^{R}$, $V\left({\bf p},t\right)$
is in the Hilbert space of square-integrable, real-valued functions
relative to ${\cal P}$, ${\cal L}_{2}\left(\Omega,\mathscr{F},{\cal P};\mathbb{R}\right)$,
it is then reasonable to consider orthogonal projections, that is,
the Minimum Mean Square Error (MMSE) estimate, or, more accurately,
prediction of $V\left({\bf p},t\right)$ given $\mathscr{C}\left({\cal T}_{t-1}\right)$.
This, of course, coincides with the conditional expectation $\mathbb{E}\left\{ V\left({\bf p},t\right)\left|\mathscr{C}\left({\cal T}_{t-1}\right)\right.\right\} $.
One then optimizes the \textit{random utility} $\mathbb{E}\left\{ V\left({\bf p},t\right)\left|\mathscr{C}\left({\cal T}_{t-1}\right)\right.\right\} $,
with respect to ${\bf p}$ in the random set ${\cal C}\left({\bf p}^{o}\left(t\hspace{-2pt}-\hspace{-2pt}1\right)\right)$,
as in (\ref{eq:2STAGE-2}).

Although there is nothing technically wrong with actually starting
with (\ref{eq:2STAGE-2}) as our \textit{initial} problem formulation,
and essentially bypassing the technical difficulties of (\ref{eq:2STAGE-1}),
the fact that the objective of (\ref{eq:2STAGE-2}) depends on $\mathscr{C}\left({\cal T}_{t-1}\right)$
does not render it a useful optimality criterion. This is because
the objective of (\ref{eq:2STAGE-2}) quantifies the performance of
a \textit{single} decision, \textit{only conditioned} on $\mathscr{C}\left({\cal T}_{t-1}\right)$,
\textit{despite} the fact that an optimal solution to (\ref{eq:2STAGE-2})
(provided it exists) constitutes itself a functional of $\mathscr{C}\left({\cal T}_{t-1}\right)$.
In other words, the objective of (\ref{eq:2STAGE-2}) \textit{does
not} quantify the performance of \textit{a} \textit{policy (a decision
rule)}; in order to do that, any reasonable performance criterion
should assign \textit{a number} to each policy, ranking its quality,
and \textit{not} a function depending on $\mathscr{C}\left({\cal T}_{t-1}\right)$.
The expected utility $\mathbb{E}\left\{ V_{t}\right\} $ of the variational
problem (\ref{eq:2STAGE-1}) constitutes a suitable such criterion.
And by the Fundamental Lemma, (\ref{eq:2STAGE-1}) may be indeed reduced
to (\ref{eq:2STAGE-2}), which can thus be regarded as a proxy for
solving the former.

There are two main reasons justifying our interest in policies, rather
than individual decisions. First, one should be interested in the
\textit{long-term behavior} of the beamforming (in our case) system,
in the sense that it should be possible to assess system performance
if the system is used repeatedly over time, e.g., periodically (every
hour, day) or on demand. For example, consider a beamforming system
(the ``experiment''), which operates for $N_{T}$ time slots and
dependently \textit{restarts} its operation at time slots $kN_{T}+1$,
for $k$ in some subset of $\mathbb{N}^{+}$. This might be practically
essential for maintaining system stability over time, saving on resources,
etc. It is then clear that merely quantifying the performance of individual
decisions is meaningless, from an operational point of view; simply,
the random utility approach quantifies performance \textit{only along
a specific path of the observed information}, $\mathscr{C}\left({\cal T}_{t-1}\right)$,
for $t\in\mathbb{N}_{N_{T}}^{+}$. This issue is more profound when
channel observations taking specific values correspond to events of
zero measure (this is actually the case with the Gaussian channel
model introduced in Section \ref{sec:Spatiotemporal-Wireless-Channel}).
On the contrary, it is of interest to jointly quantify system performance
when decisions are made for different outcomes of the sample space
$\Omega$. This immediately results in the need for quantifying the
performance of different policies (decision rules), and this is only
possible by considering variational optimization problems, such as
(\ref{eq:2STAGE-1}).

Additionally, because decisions are made \textit{in stages}, it is
of great interest to consider how the system performs \textit{across
time slots}, or, in other words, to discover \textit{temporal trends}
in performance, if such trends exist. In particular, for the beamforming
problem considered in this paper, we will be able to theoretically
characterize system behavior under both suboptimal and optimal decision
making, in the average (expected) sense (see Section \ref{subsec:Theoretical-Guarantees}),
\textit{across all time slots}; this is impossible to do for each
possible outcome of the sample space, individually, when the random
utility approach is considered.

The second main reason for considering the variational program (\ref{eq:2STAGE-1})
as our main objective, instead of (\ref{eq:2STAGE-2}), is practical,
and extremely important from an engineering point of view. The expected
utility approach assigns, at each time slot, a number to each policy,
quantifying its quality. Simulating repeatedly the system and invoking
the Law of Large Numbers, one may obtain excellent estimates of the
expected performance of the system, quantified by the chosen utility.
Therefore, the systematic experimental assessment of a particular
sequence of policies (one for each time slot) is readily possible.
Apparently, such experimental validation approach is impossible to
perform by adopting the random (conditional) utility approach, since
the performance of the system will be quantified via a real valued
(in general) random quantity.\hfill{}\ensuremath{\blacksquare}
\end{rem}
\begin{rem}
The stochastic programming methodology presented in this subsection
is very general and can support lots of choices in regard to the structure
of the second-stage subproblem, $V_{t}$. As shown in the discussion
developed in Appendix B, the key to showing the validity of the Fundamental
Lemma is the set of conditions \textbf{C1-C6}. If these are satisfied,
it is then possible to convert the original, variational problem into
a pointwise one, while strictly preserving optimality.\hfill{}\ensuremath{\blacksquare}
\end{rem}

\subsection{\label{subsec:SINR-Maximization}SINR Maximization at the Destination}

The basic and fundamentally important beamforming criterion considered
in this paper is that of enhancing network QoS, or, in other words,
maximizing the respective SINR at the destination, subject to a total
power budget at the relays. At each time slot $t\in\mathbb{N}_{N_{T}}^{+}$,
\textit{given} CSI encoded in $\mathscr{C}\left({\cal T}_{t}\right)$
and with $\boldsymbol{w}\left(t\right)\triangleq\left[w_{1}\left(t\right)\,\ldots\,w_{R}\left(t\right)\right]^{\boldsymbol{T}}$,
this may be achieved by formulating the constrained optimization problem
\cite{Havary_BEAM_2008,Beamforming_2_2009}\renewcommand{\arraystretch}{1.8}
\begin{equation}
\begin{array}{rl}
\underset{\boldsymbol{w}\left(t\right)}{\mathrm{maximize}} & \dfrac{\mathbb{E}\left\{ \left.P_{S}\left(t\right)\right|\mathscr{C}\left({\cal T}_{t}\right)\right\} }{\mathbb{E}\left\{ \left.P_{I+N}\left(t\right)\right|\mathscr{C}\left({\cal T}_{t}\right)\right\} }\\
\mathrm{subject\,to} & \mathbb{E}\left\{ \left.P_{R}\left(t\right)\right|\mathscr{C}\left({\cal T}_{t}\right)\right\} \le P_{c}
\end{array},\label{eq:Beamforming}
\end{equation}
\renewcommand{\arraystretch}{1}where $P_{R}\left(t\right)$, $P_{S}\left(t\right)$
and $P_{I+N}\left(t\right)$ denote the random instantaneous power
at the relays, that of the signal component and that of the interference
plus noise component at the destination (see (\ref{eq:MODEL})), respectively
and where $P_{c}>0$ denotes the total available relay transmission
power. Using the mutual independence assumptions regarding CSI related
to the source and destination, respectively, (\ref{eq:Beamforming})
can be reexpressed analytically as \cite{Havary_BEAM_2008} \renewcommand{\arraystretch}{1.8}
\begin{equation}
\begin{array}{rl}
\underset{\boldsymbol{w}\left(t\right)}{\mathrm{maximize}} & \dfrac{\boldsymbol{w}^{\boldsymbol{H}}\left(t\right){\bf R}\left({\bf p}\left(t\right),t\right)\boldsymbol{w}\left(t\right)}{\sigma_{D}^{2}+\boldsymbol{w}^{\boldsymbol{H}}\left(t\right){\bf Q}\left({\bf p}\left(t\right),t\right)\boldsymbol{w}\left(t\right)}\\
\mathrm{subject\,to} & \boldsymbol{w}^{\boldsymbol{H}}\left(t\right){\bf D}\left({\bf p}\left(t\right),t\right)\boldsymbol{w}\left(t\right)\le P_{c}
\end{array},\label{eq:Beamforming_2}
\end{equation}
\renewcommand{\arraystretch}{1}where, dropping the dependence on
$\left({\bf p}\left(t\right),t\right)$ or $t$ for brevity, 
\begin{flalign}
{\bf D} & \triangleq P_{0}\text{diag}\left(\left[\left|f_{1}\right|^{2}\,\left|f_{2}\right|^{2}\,\ldots\,\left|f_{R}\right|^{2}\right]^{\boldsymbol{T}}\right)+\sigma^{2}{\bf I}_{R}\in\mathbb{S}_{++}^{R},\label{eq:MAT_1}\\
{\bf R} & \triangleq P_{0}{\bf h}{\bf h}^{\boldsymbol{H}}\in\mathbb{S}_{+}^{R},\text{ with }{\bf h}\triangleq\left[f_{1}g_{1}\,f_{2}g_{2}\,\ldots\,f_{R}g_{R}\right]^{\boldsymbol{T}}\text{ and}\label{eq:MAT_2}\\
{\bf Q} & \triangleq\sigma^{2}\text{diag}\left(\left[\left|g_{1}\right|^{2}\,\left|g_{2}\right|^{2}\,\ldots\,\left|g_{R}\right|^{2}\right]^{\boldsymbol{T}}\right)\in\mathbb{S}_{++}^{R}.\label{eq:MAT_3}
\end{flalign}
Note that the program (\ref{eq:Beamforming_2}) is \textit{always
feasible, as long as $P_{c}$ is nonnegative}. It is well known that
the optimal value of (\ref{eq:Beamforming_2}) can be expressed in
closed form as \cite{Havary_BEAM_2008}
\begin{equation}
V_{t}\equiv V\left({\bf p}\left(t\right),t\right)\triangleq P_{c}\lambda_{max}\left(\left(\sigma_{D}^{2}{\bf I}_{R}+P_{c}{\bf D}^{-1/2}{\bf Q}{\bf D}^{-1/2}\right)^{-1}{\bf D}^{-1/2}{\bf R}{\bf D}^{-1/2}\right),
\end{equation}
for all $t\in\mathbb{N}_{N_{T}}^{+}$. Exploitting the structure of
the matrices involved, $V_{t}$ may also be expressed \textit{analytically}
as \cite{Beamforming_2_2009}
\begin{flalign}
V_{t} & \equiv\sum_{i\in\mathbb{N}_{R}^{+}}\dfrac{P_{c}P_{0}\left|f\left({\bf p}_{i}\left(t\right),t\right)\right|^{2}\left|g\left({\bf p}_{i}\left(t\right),t\right)\right|^{2}}{P_{0}\sigma_{D}^{2}\left|f\left({\bf p}_{i}\left(t\right),t\right)\right|^{2}+P_{c}\sigma^{2}\left|g\left({\bf p}_{i}\left(t\right),t\right)\right|^{2}+\sigma^{2}\sigma_{D}^{2}}\nonumber \\
 & \triangleq\sum_{i\in\mathbb{N}_{R}^{+}}V_{I}\left({\bf p}_{i}\left(t\right),t\right),\quad\forall t\in\mathbb{N}_{N_{T}}^{+}.
\end{flalign}
Adopting the $2$-stage stochastic optimization framework presented
and discussed in Section \ref{subsec:2-Stage}, we are now interested,
at each time slot $t-1\in\mathbb{N}_{N_{T}-1}^{+}$, in the program
\begin{equation}
\begin{array}{rl}
\underset{{\bf p}\left(t\right)}{\mathrm{maximize}} & {\displaystyle \mathbb{E}\left\{ \sum_{i\in\mathbb{N}_{R}^{+}}V_{I}\left({\bf p}_{i}\left(t\right),t\right)\right\} }\\
\mathrm{subject\,to} & {\bf p}\left(t\right)\equiv{\cal M}\left(\mathscr{C}\left({\cal T}_{t-1}\right)\right)\in{\cal C}\left({\bf p}^{o}\left(t\hspace{-2pt}-\hspace{-2pt}1\right)\right),\\
 & {\cal M}^{-1}\left({\cal A}\right)\in\mathscr{B}\left(\mathbb{R}^{4R\left(t-1\right)}\right),\,\forall{\cal A}\in\mathscr{B}\left(\mathbb{R}^{2R}\right)
\end{array},\label{eq:2STAGE-SINR-2}
\end{equation}
where ${\bf p}^{o}\left(1\right)\in{\cal S}^{R}$ is a known constant,
representing the initial positions of the relays. But in order to
be able to formulate (\ref{eq:2STAGE-SINR-2}) in a well defined manner
fully and and simplify it by exploitting the Fundamental Lemma, we
have to explicitly verify Conditions \textbf{C1-C6 }of Lemma \ref{lem:FUND_Lemma_FINAL}
in Section \ref{subsec:FINAL} of Appendix B. To this end, let us
present a definition.
\begin{defn}
\textbf{(Translated Multifunctions)}\label{def:(Translated-Multifunctions)}
Given nonempty sets ${\cal H}\subset\mathbb{R}^{N}$, ${\cal A}\subseteq\mathbb{R}^{N}$
and any fixed $\boldsymbol{h}\in{\cal H}$, ${\cal D}:\mathbb{R}^{N}\rightrightarrows\mathbb{R}^{N}$
is called the $\left({\cal H},\boldsymbol{h}\right)$-translated multifunction
in ${\cal A}$, if and only if ${\cal D}\left(\boldsymbol{y}\right)\triangleq\left\{ \left.\boldsymbol{x}\in{\cal A}\right|\boldsymbol{x}-\boldsymbol{y}\in{\cal H}\right\} $,
for all $\boldsymbol{y}\in{\cal A}-\boldsymbol{h}\triangleq\left\{ \left.\boldsymbol{x}\in\mathbb{R}^{N}\right|\boldsymbol{x}+\boldsymbol{h}\in{\cal A}\right\} $.
\end{defn}

Note that translated multifunctions, in the sense of Definition \ref{def:(Translated-Multifunctions)},
are always unique and nonempty, whenever $\boldsymbol{y}\in{\cal A}-\boldsymbol{h}$.
We also observe that, if $\boldsymbol{y}\notin{\cal A}-\boldsymbol{h}$,
${\cal D}\left(\boldsymbol{y}\right)$ is undefined; in fact, outside
${\cal A}-\boldsymbol{h}$, ${\cal D}$ may be defined arbitrarily,
and this will be irrelevant in our analysis. The following assumption
on the structure of the compact-valued multifunction ${\cal C}:\mathbb{R}^{2R}\rightrightarrows\mathbb{R}^{2R}$
is adopted hereafter, and for the rest of this paper.
\begin{assumption}
\textbf{\textup{(${\cal C}$ is Translated)}}\label{assu:AS_(TransMultifunctions)}
Given any arbitrary compact set ${\bf 0}\in{\cal G}$,
${\cal C}$ constitutes the corresponding $\left({\cal G},{\bf 0}\right)$-translated,
compact-valued multifunction in ${\cal S}^{R}$.
\end{assumption}

Then, the following important result is true.
\begin{thm}
\textbf{\textup{(Verification Theorem / SINR Maximization)\label{lem:C1C4_SAT}}}
Suppose that, at time slot $t-1\in\mathbb{N}_{N_{T}-1}^{+}$, the
\uline{selected} decision at $t-2$, ${\bf p}^{o}\left(t-1\right)\equiv{\bf p}^{o}\left(\omega,t-1\right)$,
is measurable relative to $\mathscr{C}\left({\cal T}_{t-2}\right)$.
Then, the stochastic program (\ref{eq:2STAGE-SINR-2}) satisfies conditions
\textbf{\textup{C1-C6 }}and the Fundamental Lemma applies (see Appendix
B, Section \ref{subsec:FINAL}, Lemma \ref{lem:FUND_Lemma_FINAL}).
Additionally, as long as the pointwise program\renewcommand{\arraystretch}{1.3}
\begin{equation}
\begin{array}{rl}
\underset{{\bf p}}{\mathrm{maximize}} & {\displaystyle \sum_{i\in\mathbb{N}_{R}^{+}}\mathbb{E}\left\{ \left.V_{I}\left({\bf p}_{i},t\right)\right|\mathscr{C}\left({\cal T}_{t-1}\right)\right\} }\\
\mathrm{subject\,to} & {\bf p}\in{\cal C}\left({\bf p}^{o}\left(t\hspace{-2pt}-\hspace{-2pt}1\right)\right)
\end{array}\label{eq:SINR_SIMPLIFIED}
\end{equation}
\renewcommand{\arraystretch}{1}has a unique maximizer ${\bf p}^{*}\left(t\right)$,
and ${\bf p}^{o}\left(t\right)\equiv{\bf p}^{*}\left(t\right)$, then
${\bf p}^{o}\left(t\right)$ is $\mathscr{C}\left({\cal T}_{t-1}\right)$-measurable
and the condition of the theorem is automatically satisfied at time
slot $t$.
\end{thm}
\begin{proof}[Proof of Theorem \ref{lem:C1C4_SAT}]
See Appendix C.
\end{proof}

As Theorem \ref{lem:C1C4_SAT} suggests, in order for conditions \textbf{C1-C6
}to be simultaneously satisfied for all $t\in\mathbb{N}_{N_{T}}^{2}$,
it is sufficient that the program (\ref{eq:SINR_SIMPLIFIED}) has
a \textit{unique} optimal solution, for each $t$. Although, in general,
such requirement might not be particularly appealing, for the problems
of interest in this paper, the event where (\ref{eq:SINR_SIMPLIFIED})
does not have a unique optimizer is extremely rare, almost never occurring
in practice. Nevertheless, uniqueness of the optimal solution to (\ref{eq:SINR_SIMPLIFIED})
does not constitute a necessary condition for $\mathscr{C}\left({\cal T}_{t-1}\right)$-measurability
of the optimal decision at time slot $t-1$. For instance, \textbf{${\bf p}^{*}\left(t\right)$
}will always be $\mathscr{C}\left({\cal T}_{t-1}\right)$-measurable
when the compact-valued, closed multifunction ${\cal C}:\mathbb{R}^{2R}\rightrightarrows\mathbb{R}^{2R}$
is additionally \textit{finite-valued}, and ${\bf p}^{o}\left(t\right)\equiv{\bf p}^{*}\left(t\right)$.
This choice for ${\cal C}$ is particularly useful for practical implementations.
In any case, as long as conditions \textbf{C1-C6 }are guaranteed to
be satisfied, we may focus exclusively on the pointwise program (\ref{eq:SINR_SIMPLIFIED}),
whose expected optimal value, via the Fundamental Lemma, coincides
with the optimal value of the original problem (\ref{eq:2STAGE-SINR-2}).

By definition, we readily observe that the problem (\ref{eq:SINR_SIMPLIFIED})
is separable. In fact, given that, \textit{for each $t\in\mathbb{N}_{N_{T}-1}^{+}$,
decisions taken and CSI collected so far are available to all relays},
(\ref{eq:SINR_SIMPLIFIED}) can be solved in a \textit{completely
distributed fashion at the relays}, with the $i$-th relay being responsible
for solving the program\renewcommand{\arraystretch}{1.3}
\begin{equation}
\begin{array}{rl}
\underset{{\bf p}}{\mathrm{maximize}} & {\displaystyle \mathbb{E}\left\{ \left.V_{I}\left({\bf p},t\right)\right|\mathscr{C}\left({\cal T}_{t-1}\right)\right\} }\\
\mathrm{subject\,to} & {\bf p}\in{\cal C}_{i}\left({\bf p}^{o}\left(t\hspace{-2pt}-\hspace{-2pt}1\right)\right)
\end{array},\label{eq:SINR_DIST}
\end{equation}
\renewcommand{\arraystretch}{1}at each $t-1\in\mathbb{N}_{N_{T}-1}^{+}$,
where ${\cal C}_{i}:\mathbb{R}^{2}\rightrightarrows\mathbb{R}^{2}$
denotes the corresponding part of ${\cal C}$, for each \textit{$i\in\mathbb{N}_{R}^{+}$.
}Note that no local exchange of intermediate results is required among
relays; given the available information, each relay independently
solves its own subproblem. It is also evident that apart from the
obvious difference in the feasible set, the optimization problems
at each of the relays are identical. The problem, however, with (\ref{eq:SINR_DIST})
is that its objective involves the evaluation of a conditional expectation
of a well defined ratio of almost surely positive random variables,
which is \textit{impossible to perform analytically}. For this reason,
it is imperative to resort to the development of well behaved approximations
to (\ref{eq:SINR_DIST}), which, at the same time, would facilitate
implementation. In the following, we present two such heuristic approaches.

\subsubsection{\label{subsec:SINR_Taylor}Approximation by the Method of Statistical
Differentials}

The first idea we are going to explore is that of approximating the
objective of (\ref{eq:SINR_DIST}) by truncated Taylor expansions.
Observe that $V_{I}$ can be equivalently expressed as 
\begin{align}
V_{I}\left({\bf p},t\right) & \equiv\dfrac{1}{\dfrac{\sigma_{D}^{2}}{P_{c}}\left|g\left({\bf p},t\right)\right|^{-2}+\dfrac{\sigma^{2}}{P_{0}}\left|f\left({\bf p},t\right)\right|^{-2}+\dfrac{\sigma^{2}\sigma_{D}^{2}}{P_{c}P_{0}}\left|f\left({\bf p},t\right)\right|^{-2}\left|g\left({\bf p},t\right)\right|^{-2}}\triangleq\dfrac{1}{V_{II}\left({\bf p},t\right)},
\end{align}
for all $\left({\bf p},t\right)\in{\cal S}\times\mathbb{N}_{N_{T}}^{+}$.
Then, for $t\in\mathbb{N}_{N_{T}}^{2}$, we may locally approximate
${\displaystyle \mathbb{E}\left\{ \left.V_{I}\left({\bf p},t\right)\right|\mathscr{C}\left({\cal T}_{t-1}\right)\right\} }$
around the point ${\displaystyle \mathbb{E}\left\{ \left.V_{II}\left({\bf p},t\right)\right|\mathscr{C}\left({\cal T}_{t-1}\right)\right\} }$
(see Section 3.14.2 in \cite{Survival_1990Elandt}; also known as
the \textit{Method of Statistical Differentials}) via a first order
Taylor expansion as 
\begin{equation}
{\displaystyle \mathbb{E}\left\{ \left.V_{I}\left({\bf p},t\right)\right|\mathscr{C}\left({\cal T}_{t-1}\right)\right\} }\approx\dfrac{1}{\mathbb{E}\left\{ \left.V_{II}\left({\bf p},t\right)\right|\mathscr{C}\left({\cal T}_{t-1}\right)\right\} },\label{eq:SINR_APPROX_1}
\end{equation}
or via a second order Taylor expansion as
\begin{align}
{\displaystyle \mathbb{E}\left\{ \left.V_{I}\left({\bf p},t\right)\right|\mathscr{C}\left({\cal T}_{t-1}\right)\right\} } & \approx\dfrac{\mathbb{E}\left\{ \left.\left(V_{II}\left({\bf p},t\right)\right)^{2}\right|\mathscr{C}\left({\cal T}_{t-1}\right)\right\} }{\left(\mathbb{E}\left\{ \left.V_{II}\left({\bf p},t\right)\right|\mathscr{C}\left({\cal T}_{t-1}\right)\right\} \right)^{3}},\label{eq:SINR_APPROX_2}
\end{align}
where it is straightforward to show that the square on the numerator
can be expanded as
\begin{flalign}
\left(V_{II}\left({\bf p},t\right)\right)^{2} & \equiv\left(\dfrac{\sigma^{2}\sigma_{D}^{2}}{P_{c}P_{0}}\right)^{2}\left|f\left({\bf p},t\right)\right|^{-4}\left|g\left({\bf p},t\right)\right|^{-4}+2\dfrac{\sigma^{2}\sigma_{D}^{2}}{P_{c}P_{0}}\left|f\left({\bf p},t\right)\right|^{-2}\left|g\left({\bf p},t\right)\right|^{-2}\nonumber \\
 & \quad+2\left(\dfrac{\sigma^{2}}{P_{0}}\right)^{2}\dfrac{\sigma_{D}^{2}}{P_{c}}\left|f\left({\bf p},t\right)\right|^{-4}\left|g\left({\bf p},t\right)\right|^{-2}+2\dfrac{\sigma^{2}}{P_{0}}\left(\dfrac{\sigma_{D}^{2}}{P_{c}}\right)^{2}\left|f\left({\bf p},t\right)\right|^{-2}\left|g\left({\bf p},t\right)\right|^{-4}\nonumber \\
 & \quad+\left(\dfrac{\sigma^{2}}{P_{0}}\right)^{2}\left|f\left({\bf p},t\right)\right|^{-4}+\left(\dfrac{\sigma_{D}^{2}}{P_{c}}\right)^{2}\left|g\left({\bf p},t\right)\right|^{-4}.\label{eq:SINR_EXPANSE}
\end{flalign}
The approximate formula (\ref{eq:SINR_APPROX_2}) may be in fact computed
in closed form at any point ${\bf p}\in{\cal S}$, thanks to the following
technical, but simple, result.
\begin{lem}
\textbf{\textup{(Big Expectations)\label{thm:(Big-Expectations)}}}
Under the wireless channel model introduced in Section \ref{sec:Spatiotemporal-Wireless-Channel},
it is true that, at any ${\bf p}\in{\cal S}$,
\begin{equation}
\left.\left[F\left({\bf p},t\right)\,G\left({\bf p},t\right)\right]^{\boldsymbol{T}}\right|\mathscr{C}\left({\cal T}_{t-1}\right)\sim{\cal N}\left(\boldsymbol{\mu}_{\left.t\right|t-1}^{F,G}\hspace{-2pt}\left({\bf p}\right),\boldsymbol{\Sigma}_{\left.t\right|t-1}^{F,G}\hspace{-2pt}\left({\bf p}\right)\right),\label{eq:JOINT_COND}
\end{equation}
for all $t\in\mathbb{N}_{N_{T}}^{2}$, and where we define
\begin{flalign}
\boldsymbol{\mu}_{\left.t\right|t-1}^{F,G}\hspace{-2pt}\left({\bf p}\right) & \hspace{-2pt}\triangleq\hspace{-2pt}\left[\mu_{\left.t\right|t-1}^{F}\left({\bf p}\right)\,\mu_{\left.t\right|t-1}^{G}\left({\bf p}\right)\right]^{\boldsymbol{T}},\\
\mu_{\left.t\right|t-1}^{F}\left({\bf p}\right) & \hspace{-2pt}\triangleq\hspace{-2pt}\alpha_{S}\left({\bf p}\right)\ell+\boldsymbol{c}_{1:t-1}^{F}\left({\bf p}\right)\boldsymbol{\Sigma}_{1:t-1}^{-1}\hspace{-2pt}\left(\boldsymbol{m}_{1:t-1}\hspace{-2pt}-\hspace{-2pt}\boldsymbol{\mu}_{1:t-1}\right)\in\mathbb{R},\\
\mu_{\left.t\right|t-1}^{G}\left({\bf p}\right) & \hspace{-2pt}\triangleq\hspace{-2pt}\alpha_{D}\left({\bf p}\right)\ell+\boldsymbol{c}_{1:t-1}^{G}\left({\bf p}\right)\boldsymbol{\Sigma}_{1:t-1}^{-1}\hspace{-2pt}\left(\boldsymbol{m}_{1:t-1}\hspace{-2pt}-\hspace{-2pt}\boldsymbol{\mu}_{1:t-1}\right)\in\mathbb{R}\quad\text{and}\\
\boldsymbol{\Sigma}_{\left.t\right|t-1}^{F,G}\left({\bf p}\right) & \hspace{-2pt}\triangleq\hspace{-2pt}\begin{bmatrix}\eta^{2}+\sigma_{\xi}^{2} & \eta^{2}e^{-\frac{\left\Vert {\bf p}_{S}-{\bf p}_{D}\right\Vert _{2}}{\delta}}\\
\eta^{2}e^{-\frac{\left\Vert {\bf p}_{S}-{\bf p}_{D}\right\Vert _{2}}{\delta}} & \eta^{2}+\sigma_{\xi}^{2}
\end{bmatrix}-\begin{bmatrix}\boldsymbol{c}_{1:t-1}^{F}\left({\bf p}\right)\\
\boldsymbol{c}_{1:t-1}^{G}\left({\bf p}\right)
\end{bmatrix}\boldsymbol{\Sigma}_{1:t-1}^{-1}\begin{bmatrix}\boldsymbol{c}_{1:t-1}^{F}\left({\bf p}\right)\\
\boldsymbol{c}_{1:t-1}^{G}\left({\bf p}\right)
\end{bmatrix}^{\boldsymbol{T}}\in\mathbb{S}_{++}^{2},
\end{flalign}
with 
\begin{flalign}
\boldsymbol{m}_{1:t-1} & \hspace{-2pt}\triangleq\hspace{-2pt}\left[\boldsymbol{F}^{\boldsymbol{T}}\left(1\right)\,\boldsymbol{G}^{\boldsymbol{T}}\left(1\right)\,\ldots\,\boldsymbol{F}^{\boldsymbol{T}}\left(t-1\right)\,\boldsymbol{G}^{\boldsymbol{T}}\left(t-1\right)\right]^{\boldsymbol{T}}\in\mathbb{R}^{2R\left(t-1\right)\times1},\\
\boldsymbol{\mu}_{1:t-1} & \hspace{-2pt}\triangleq\hspace{-2pt}\left[\boldsymbol{\alpha}_{S}\left({\bf p}\left(1\right)\right)\,\boldsymbol{\alpha}_{D}\left({\bf p}\left(1\right)\right)\,\ldots\,\boldsymbol{\alpha}_{S}\left({\bf p}\left(t-1\right)\right)\,\boldsymbol{\alpha}_{D}\left({\bf p}\left(t-1\right)\right)\right]^{\boldsymbol{T}}\ell\in\mathbb{R}^{2R\left(t-1\right)\times1},\\
\boldsymbol{c}_{1:t-1}^{F}\left({\bf p}\right) & \hspace{-2pt}\triangleq\hspace{-2pt}\left[\boldsymbol{c}_{1}^{F}\left({\bf p}\right)\,\ldots\,\boldsymbol{c}_{t-1}^{F}\left({\bf p}\right)\right]\in\mathbb{R}^{1\times2R\left(t-1\right)},\\
\boldsymbol{c}_{1:t-1}^{G}\left({\bf p}\right) & \hspace{-2pt}\triangleq\hspace{-2pt}\left[\boldsymbol{c}_{1}^{G}\left({\bf p}\right)\,\ldots\,\boldsymbol{c}_{t-1}^{G}\left({\bf p}\right)\right]\in\mathbb{R}^{1\times2R\left(t-1\right)},\\
\boldsymbol{c}_{k}^{F}\left({\bf p}\right) & \hspace{-2pt}\triangleq\hspace{-2pt}\left[\left\{ \mathbb{E}\left\{ \sigma_{S}\left({\bf p},t\right)\sigma_{S}^{j}\left(k\right)\right\} \hspace{-2pt}\right\} _{j\in\mathbb{N}_{R}^{+}}\,\left\{ \mathbb{E}\left\{ \sigma_{S}\left({\bf p},t\right)\sigma_{D}^{j}\left(k\right)\right\} \hspace{-2pt}\right\} _{j\in\mathbb{N}_{R}^{+}}\right]\hspace{-2pt},\;\forall k\in\mathbb{N}_{t-1}^{+}\\
\boldsymbol{c}_{k}^{G}\left({\bf p}\right) & \hspace{-2pt}\triangleq\hspace{-2pt}\left[\left\{ \mathbb{E}\left\{ \sigma_{D}\left({\bf p},t\right)\sigma_{S}^{j}\left(k\right)\right\} \hspace{-2pt}\right\} _{j\in\mathbb{N}_{R}^{+}}\,\left\{ \mathbb{E}\left\{ \sigma_{D}\left({\bf p},t\right)\sigma_{D}^{j}\left(k\right)\right\} \hspace{-2pt}\right\} _{j\in\mathbb{N}_{R}^{+}}\right]\hspace{-2pt},\;\forall k\in\mathbb{N}_{t-1}^{+}\;\text{and}\\
\boldsymbol{\Sigma}_{1:t-1} & \hspace{-2pt}\triangleq\hspace{-2pt}\begin{bmatrix}\boldsymbol{\Sigma}\left(1,1\right) & \cdots & \boldsymbol{\Sigma}\left(1,t-1\right)\\
\vdots & \ddots & \vdots\\
\boldsymbol{\Sigma}\left(t-1,1\right) & \cdots & \boldsymbol{\Sigma}\left(t-1,t-1\right)
\end{bmatrix}\in\mathbb{S}_{++}^{2R\left(t-1\right)},
\end{flalign}
for all $\left({\bf p},t\right)\in{\cal S}\times\mathbb{N}_{N_{T}}^{2}$.
Further, for any choice of $\left(m,n\right)\in\mathbb{Z}\times\mathbb{Z}$,
the conditional correlation of the fields $\left|f\left({\bf p},t\right)\right|^{m}$
and $\left|g\left({\bf p},t\right)\right|^{n}$ relative to $\mathscr{C}\left({\cal T}_{t-1}\right)$
may be expressed in closed form as 
\begin{multline}
\mathbb{E}\left\{ \left.\left|f\left({\bf p},t\right)\right|^{m}\left|g\left({\bf p},t\right)\right|^{n}\right|\mathscr{C}\left({\cal T}_{t-1}\right)\right\} \\
\equiv\hspace{-2pt}10^{\left(m+n\right)\rho/20}\exp\hspace{-2pt}\left(\hspace{-2pt}\dfrac{\log\left(10\right)}{20}\hspace{-2pt}\begin{bmatrix}m\\
n
\end{bmatrix}^{\boldsymbol{T}}\hspace{-2pt}\boldsymbol{\mu}_{\left.t\right|t-1}^{F,G}\hspace{-2pt}\left({\bf p}\right)\hspace{-2pt}+\hspace{-2pt}\left(\dfrac{\log\left(10\right)}{20}\right)^{2}\hspace{-2pt}\begin{bmatrix}m\\
n
\end{bmatrix}^{\boldsymbol{T}}\hspace{-2pt}\boldsymbol{\Sigma}_{\left.t\right|t-1}^{F,G}\hspace{-2pt}\left({\bf p}\right)\hspace{-2pt}\begin{bmatrix}m\\
n
\end{bmatrix}\hspace{-2pt}\right)\hspace{-2pt},\label{eq:BOX_1}
\end{multline}
at any ${\bf p}\in{\cal S}$ and for all $t\in\mathbb{N}_{N_{T}}^{2}$.
\end{lem}
\begin{proof}[Proof of Lemma \ref{thm:(Big-Expectations)}]
See Appendix C.
\end{proof}

Since, by exploitting Lemma \ref{thm:(Big-Expectations)} and (\ref{eq:SINR_EXPANSE}),
formula (\ref{eq:SINR_APPROX_2}) can be evaluated without any particular
difficulty, we now propose the replacement of the original pointwise
problem of interest, (\ref{eq:SINR_DIST}), with either of the heuristics\renewcommand{\arraystretch}{1.4}
\begin{equation}
\hspace{-2pt}\hspace{-2pt}\hspace{-2pt}\hspace{-2pt}\hspace{-2pt}\hspace{-2pt}\hspace{-2pt}\hspace{-2pt}\hspace{-2pt}\hspace{-2.6pt}\begin{array}{rl}
\underset{{\bf p}}{\mathrm{maximize}} & \dfrac{1}{\mathbb{E}\left\{ \left.V_{II}\left({\bf p},t\right)\right|\mathscr{C}\left({\cal T}_{t-1}\right)\right\} }\\
\mathrm{subject\,to} & {\bf p}\in{\cal C}_{i}\left({\bf p}^{o}\left(t\hspace{-2pt}-\hspace{-2pt}1\right)\right)
\end{array}\label{eq:SINR_APPROX_PROG_1}
\end{equation}
\renewcommand{\arraystretch}{1}and\renewcommand{\arraystretch}{1.4}
\begin{equation}
\begin{array}{rl}
\underset{{\bf p}}{\mathrm{maximize}} & \dfrac{\mathbb{E}\left\{ \left.\left(V_{II}\left({\bf p},t\right)\right)^{2}\right|\mathscr{C}\left({\cal T}_{t-1}\right)\right\} }{\left(\mathbb{E}\left\{ \left.V_{II}\left({\bf p},t\right)\right|\mathscr{C}\left({\cal T}_{t-1}\right)\right\} \right)^{3}}\\
\mathrm{subject\,to} & {\bf p}\in{\cal C}_{i}\left({\bf p}^{o}\left(t\hspace{-2pt}-\hspace{-2pt}1\right)\right)
\end{array},\label{eq:SINR_APPROX_PROG_2}
\end{equation}
\renewcommand{\arraystretch}{1}to be solved at relay $i\in\mathbb{N}_{R}^{+}$,
at each time $t-1\in\mathbb{N}_{N_{T}-1}^{+}$, depending on the order
of approximation employed, respectively. Observe that Jensen's Inequality
directly implies that the objective of (\ref{eq:SINR_APPROX_PROG_1})
is always \textit{lower than or equal} than that of (\ref{eq:SINR_APPROX_PROG_2})
and that of the original program (\ref{eq:SINR_DIST}), conditioned,
of course, on identical information. As a result, (\ref{eq:SINR_APPROX_PROG_1})
is also a \textit{lower bound relaxation} to (\ref{eq:SINR_DIST}).
On the other hand, the objective of (\ref{eq:SINR_APPROX_PROG_1})
might be desirable in practice, since it is easier to compute. Both
approximations are technically well behaved, though, as made precise
by the next theorem.

\begin{thm}
\textbf{\textup{(Behavior of Approximation Chains I / SINR Maximization)\label{lem:WELL_Behaved}}}
Both heuristics (\ref{eq:SINR_APPROX_PROG_1}) and (\ref{eq:SINR_APPROX_PROG_2})
each feature at least one measurable maximizer. Therefore, provided
that any of the two heuristics is solved at each time slot $t-1\in\mathbb{N}_{N_{T}-1}^{+}$,
that the selected one features a unique maximizer, $\widetilde{{\bf p}}^{*}\left(t\right)$,
and that $\widetilde{{\bf p}}^{*}\left(t\right)\equiv{\bf p}^{o}\left(t\right)$,
for all $t\in\mathbb{N}_{N_{T}}^{2}$, the produced decision chain
is measurable and condition \textbf{\textup{C2}} is satisfied at all
times.
\end{thm}

\begin{proof}[Proof of Theorem \ref{lem:WELL_Behaved}]
See Appendix C.
\end{proof}

Theorem \ref{lem:WELL_Behaved} implies that, at each time slot $t\in\mathbb{N}_{N_{T}-1}^{+}$
and under the respective conditions, the chosen heuristic constitutes
a well defined approximation to the original problem, (\ref{eq:SINR_DIST})
and, in turn, to (\ref{eq:2STAGE-SINR-2}), in the sense that all
conditions \textbf{C1-C4} are satisfied.

At this point, it will be important to note that, for \textit{each}
${\bf p}\in{\cal S}$, computation of the conditional mean and covariance
in (\ref{eq:JOINT_COND}) of Lemma \ref{thm:(Big-Expectations)} require
execution of matrix operations, which are of \textit{expanding dimension
in} $t\in\mathbb{N}_{N_{T}}^{2}$; observe that, for instance, the
covariance matrix $\boldsymbol{\Sigma}_{1:t-1}$ is of size $2R\left(t-1\right)$,
which is increasing in $t\in\mathbb{N}_{N_{T}}^{2}$. Fortunately,
however, the increase is linear in $t$. Additionally, the reader
may readily observe that the inversion of the covariance matrix $\boldsymbol{\Sigma}_{1:t-1}$
constitutes the computationally dominant operation in the long formulas
of Lemma \ref{thm:(Big-Expectations)}. The computational complexity
of this matrix inversion, which takes place at each time slot $t-1\in\mathbb{N}_{N_{T}-1}^{+}$,
is, in general, of the order of ${\cal O}\left(R^{3}t^{3}\right)$
elementary operations. Fortunately though, we may exploit the Matrix
Inversion Lemma, in order to reduce the computational complexity of
the aforementioned matrix inversion to the order of ${\cal O}\left(R^{3}t^{2}\right)$.
Indeed, by construction, $\boldsymbol{\Sigma}_{1:t-1}$ may be expressed
as
\begin{equation}
\boldsymbol{\Sigma}_{1:t-1}\equiv\begin{bmatrix}\boldsymbol{\Sigma}_{1:t-2} & \boldsymbol{\Sigma}_{1:t-2}^{c}\\
\left(\boldsymbol{\Sigma}_{1:t-2}^{c}\right)^{\boldsymbol{T}} & \boldsymbol{\Sigma}\left(t-1,t-1\right)
\end{bmatrix},
\end{equation}
where
\begin{equation}
\boldsymbol{\Sigma}_{1:t-2}^{c}\triangleq\left[\boldsymbol{\Sigma}\left(1,t-1\right)\,\ldots\,\boldsymbol{\Sigma}\left(t-2,t-1\right)\right]^{\boldsymbol{T}}\in\mathbb{R}^{2R\left(t-2\right)\times2R}.
\end{equation}
Invoking the Matrix Inversion Lemma, we obtain the \textit{recursive}
expression\renewcommand{\arraystretch}{1.4}
\begin{flalign}
\boldsymbol{\Sigma}_{1:t-1}^{-1} & =\begin{bmatrix}\boldsymbol{\Sigma}_{1:t-2}^{-1}+\boldsymbol{\Sigma}_{1:t-2}^{-1}\boldsymbol{\Sigma}_{1:t-2}^{c}\boldsymbol{\mathsf{S}}_{t-1}^{-1}\left(\boldsymbol{\Sigma}_{1:t-2}^{c}\right)^{\boldsymbol{T}}\boldsymbol{\Sigma}_{1:t-2}^{-1} & -\boldsymbol{\Sigma}_{1:t-2}^{-1}\boldsymbol{\Sigma}_{1:t-2}^{c}\boldsymbol{\mathsf{S}}_{t-1}^{-1}\\
-\boldsymbol{\mathsf{S}}_{t-1}^{-1}\left(\boldsymbol{\Sigma}_{1:t-2}^{c}\right)^{\boldsymbol{T}}\boldsymbol{\Sigma}_{1:t-2}^{-1} & \boldsymbol{\mathsf{S}}_{t-1}^{-1}
\end{bmatrix},\quad\text{with}\label{eq:MIL_1}\\
\boldsymbol{\mathsf{S}}_{t-1} & \triangleq\boldsymbol{\Sigma}\left(t-1,t-1\right)-\left(\boldsymbol{\Sigma}_{1:t-2}^{c}\right)^{\boldsymbol{T}}\boldsymbol{\Sigma}_{1:t-2}^{-1}\boldsymbol{\Sigma}_{1:t-2}^{c}\in\mathbb{S}_{++}^{2R},\label{eq:MIL_2}
\end{flalign}
\renewcommand{\arraystretch}{1}where $\boldsymbol{\mathsf{S}}_{t-1}$
is the respective Schur complement. From (\ref{eq:MIL_1}) and (\ref{eq:MIL_2}),
it can be easily verified that the most computationally demanding
operation involved is $\boldsymbol{\Sigma}_{1:t-2}^{-1}\boldsymbol{\Sigma}_{1:t-2}^{c}$,
of order ${\cal O}\left(R^{3}t^{2}\right)$. Since the inversion of
$\boldsymbol{\mathsf{S}}_{t-1}$ is of the order of ${\cal O}\left(R^{3}\right)$,
we arrive at a total complexity of ${\cal O}\left(R^{3}t^{2}\right)$
elementary operations of the recursive scheme presented above, and
implemented at each time slot $t-1$. The achieved reduction in complexity
is important. In most scenarios, $R$, the number of relays, will
be relatively small and fixed for the whole operation of the system,
whereas $t$, the time slot index, might generally take large values,
since it is common for the operational horizon of the system, $N_{T}$,
to be large. Additionally, the reader may readily observe that the
aforementioned covariance matrix is independent of the position at
which the channel is predicted, ${\bf p}$. As a result, its inversion
may be performed just once in each time slot, for all evaluations
of the mean and covariance of the Gaussian density in (\ref{eq:JOINT_COND}),
for all different choices of ${\bf p}$ on a fixed grid (say). Consequently,
if the total number of such evaluations is $P\in\mathbb{N}^{+}$,
and recalling that the complexity for a matrix-vector multiplication
is quadratic in the dimension of the quantities involved, then, \textit{at
worst}, the total computational complexity for channel prediction
is of the order of ${\cal O}\left(PR^{2}t^{2}+R^{3}t^{2}\right)$,
\textit{at each $t-1\in\mathbb{N}_{N_{T}-1}^{+}$. }This means that
a potential actual computational system would have to be able to execute
matrix operations with complexity at most of the order of ${\cal O}\left(PR^{2}N_{T}^{2}+R^{3}N_{T}^{2}\right)$,
which constitutes the worst case complexity, \textit{over all $N_{T}$
time slots}. The analysis above characterizes the complexity for solving
either of the heuristics (\ref{eq:SINR_APPROX_PROG_1}) and (\ref{eq:SINR_APPROX_PROG_2}),
if the feasible set ${\cal C}_{i}$ is assumed to be finite, for all
$i\in\mathbb{N}_{R}^{+}$. Of course, if the quantity $RN_{T}$ is
considered a fixed constant, implying that computation of the mean
and covariance in (\ref{eq:JOINT_COND}) is considered the result
of a black box with fixed (worst) execution time and with input ${\bf p}$,
then, at each $t-1\in\mathbb{N}_{N_{T}-1}^{+}$, the total computational
complexity for channel prediction is of the order of ${\cal O}\left(P\right)$
function evaluations, that is, linear in $P$.

\subsubsection{\label{subsec:Brute-Force}Brute Force}

The second approach to the solution of (\ref{eq:SINR_DIST}), considered
in this section, is based on the fact that the objective of the aforementioned
program can be evaluated rather efficiently, relying on the \textit{multidimensional
Gauss-Hermite Quadrature Rule} \cite{NumericalRecipes1996}, which
constitutes a readily available routine for numerical integration.
It is particularly effective for computing expectations of complicated
functions of Gaussian random variables \cite{Elliott_GHQ_2007}. This
is indeed the case here, as shown below.

Leveraging Lemma \ref{thm:(Big-Expectations)} and as it can also
be seen in the proof of Theorem \ref{lem:C1C4_SAT} (condition \textbf{C6}),
the objective of (\ref{eq:SINR_DIST}) can be equivalently represented,
for all $t\in\mathbb{N}_{N_{T}}^{2}$, via a Lebesgue integral as
\begin{equation}
{\displaystyle \mathbb{E}\left\{ \left.V_{I}\left({\bf p},t\right)\right|\mathscr{C}\left({\cal T}_{t-1}\right)\right\} }=\int_{\mathbb{R}^{2}}r\left(\boldsymbol{x}\right){\cal N}\left(\boldsymbol{x};\boldsymbol{\mu}_{\left.t\right|t-1}^{F,G}\hspace{-2pt}\left({\bf p}\right),\boldsymbol{\Sigma}_{\left.t\right|t-1}^{F,G}\hspace{-2pt}\left({\bf p}\right)\right)\text{d}\boldsymbol{x},\label{eq:Cond_Density_1}
\end{equation}
for any choice of ${\bf p}\in{\cal S}$, where ${\cal N}\left(\cdot;\boldsymbol{\mu},\boldsymbol{\Sigma}\right):\mathbb{R}^{2}\rightarrow\mathbb{R}_{++}$
denotes the bivariate Gaussian density, with mean $\boldsymbol{\mu}\in\mathbb{R}^{2\times1}$
and covariance $\boldsymbol{\Sigma}\in\mathbb{S}_{+}^{2\times2}$,
and the function $r:\mathbb{R}^{2}\rightarrow\mathbb{R}_{++}$ is
defined exploitting the trick (\ref{eq:CONVERTER}) as 
\begin{equation}
r\left(\boldsymbol{x}\right)\equiv r\left(x_{1},x_{2}\right)\triangleq\dfrac{P_{c}P_{0}10^{2\rho/10}\left[\exp\left(x_{1}+x_{2}\right)\right]^{\frac{\log\left(10\right)}{10}}}{P_{0}\sigma_{D}^{2}\left[\exp\left(x_{1}\right)\right]^{\frac{\log\left(10\right)}{10}}+P_{c}\sigma^{2}\left[\exp\left(x_{2}\right)\right]^{\frac{\log\left(10\right)}{10}}+10^{-\rho/10}\sigma^{2}\sigma_{D}^{2}},
\end{equation}
for all $\boldsymbol{x}\equiv\left(x_{1},x_{2}\right)\in\mathbb{R}^{2}$.
Exploitting the Lebesgue integral representation (\ref{eq:Cond_Density_1}),
it can be easily shown that the conditional expectation may be closely
approximated by the double summation formula (see Section IV in \cite{Elliott_GHQ_2007})
\begin{equation}
{\displaystyle \mathbb{E}\left\{ \left.V_{I}\left({\bf p},t\right)\right|\mathscr{C}\left({\cal T}_{t-1}\right)\right\} }\approx\sum_{l_{1}\in\mathbb{N}_{M}^{+}}\varpi_{l_{1}}\sum_{l_{2}\in\mathbb{N}_{M}^{+}}\varpi_{l_{2}}r\left(\sqrt{\boldsymbol{\Sigma}_{\left.t\right|t-1}^{F,G}\hspace{-2pt}\left({\bf p}\right)}\boldsymbol{q}_{\left(l_{1},l_{2}\right)}+\boldsymbol{\mu}_{\left.t\right|t-1}^{F,G}\hspace{-2pt}\left({\bf p}\right)\right),\label{eq:Quadrature_1}
\end{equation}
where $M\in\mathbb{N}^{+}$ denotes the quadrature resolution, $\boldsymbol{q}_{\left(l_{1},l_{2}\right)}\triangleq\left[q_{l_{1}}\,q_{l_{2}}\right]^{\boldsymbol{T}}\in\mathbb{R}^{2\times1}$
denotes the $\left(l_{1},l_{2}\right)$-th quadrature point and $\boldsymbol{\varpi}_{\left(l_{1},l_{2}\right)}\triangleq\left[\varpi_{l_{1}}\,\varpi_{l_{2}}\right]^{\boldsymbol{T}}\in\mathbb{R}^{2\times1}$
denotes respective weighting coefficient, for all $\left(l_{1},l_{2}\right)\in\mathbb{N}_{M}^{+}\times\mathbb{N}_{M}^{+}$.
Both sets of quadrature points and weighting coefficients are automatically
selected \textit{apriori and independently in each dimension}, via
the following simple procedure \cite{Golub1969Calculation,Elliott_GHQ_2007}.
Let us define a matrix $\boldsymbol{J}\in\mathbb{R}^{M\times M}$,
such that
\begin{equation}
\boldsymbol{J}\left(i,j\right)\triangleq\begin{cases}
\sqrt{\dfrac{\min\left\{ i,j\right\} }{2}}, & \left|j-i\right|\equiv1\\
0, & \text{otherwise}
\end{cases},\quad\forall\left(i,j\right)\in\mathbb{N}_{M}^{+}\times\mathbb{N}_{M}^{+}.
\end{equation}
That is, $\boldsymbol{J}$ constitutes a hollow, tridiagonal, symmetric
matrix. Let the sets $\left\{ \lambda_{i}\left(\boldsymbol{J}\right)\in\mathbb{R}\right\} _{i\in\mathbb{N}_{M}^{+}}$
and $\left\{ \boldsymbol{v}_{i}\left(\boldsymbol{J}\right)\in\mathbb{R}^{M\times1}\right\} _{i\in\mathbb{N}_{M}^{+}}$
contain the eigenvalues and \textit{normalized} eigenvectors of $\boldsymbol{J}$,
respectively. Then, simply, quadrature points and the respective weighting
coefficients are selected independently in each dimension $j\in\left\{ 1,2\right\} $
as
\begin{flalign}
q_{l_{j}} & \equiv\sqrt{2}\lambda_{l_{j}}\left(\boldsymbol{J}\right)\quad\text{and}\label{eq:TRIG_1}\\
\varpi_{l_{j}} & \equiv\left(\boldsymbol{v}_{l_{j}}\left(\boldsymbol{J}\right)\left(1\right)\right)^{2},\quad\forall l_{j}\in\mathbb{N}_{M}^{+}.\label{eq:TRIG_2}
\end{flalign}
In (\ref{eq:TRIG_2}), $\boldsymbol{v}_{l_{j}}\left(\boldsymbol{J}\right)\left(1\right)$
denotes the first entry of the involved vector.

Under the above considerations, in this subsection, we propose, for
a sufficiently large number of quadrature points $M$, the replacement
of the original pointwise problem (\ref{eq:SINR_DIST}) with the heuristic\renewcommand{\arraystretch}{1.4}
\begin{equation}
\begin{array}{rl}
\underset{{\bf p}}{\mathrm{maximize}} & {\displaystyle \sum_{\left(l_{1},l_{2}\right)\in\mathbb{N}_{M}^{+}\times\mathbb{N}_{M}^{+}}\varpi_{l_{1}}\varpi_{l_{2}}r\left(\sqrt{\boldsymbol{\Sigma}_{\left.t\right|t-1}^{F,G}\hspace{-2pt}\left({\bf p}\right)}\boldsymbol{q}_{\left(l_{1},l_{2}\right)}+\boldsymbol{\mu}_{\left.t\right|t-1}^{F,G}\hspace{-2pt}\left({\bf p}\right)\right)}\\
\mathrm{subject\,to} & {\bf p}\in{\cal C}_{i}\left({\bf p}^{o}\left(t\hspace{-2pt}-\hspace{-2pt}1\right)\right)
\end{array},\label{eq:SINR_APPROX_QUADR}
\end{equation}
\renewcommand{\arraystretch}{1}to be solved at relay $i\in\mathbb{N}_{R}^{+}$,
at each time $t-1\in\mathbb{N}_{N_{T}-1}^{+}$. As in Section \ref{subsec:SINR_Taylor}
above, the following result is in power, concerning the technical
consistency of the decision chain produced by considering the approximate
program (\ref{eq:SINR_APPROX_QUADR}), for all $t\in\mathbb{N}_{N_{T}}^{2}$.
Proof is omitted, as it is essentially identical to that of Theorem
\ref{lem:WELL_Behaved}.

\begin{thm}
\textbf{\textup{(Behavior of Approximation Chains II / SINR Maximization)\label{lem:WELL_Behaved-1}}}
Consider the the heuristic (\ref{eq:SINR_APPROX_QUADR}). Then, under
the same circumstances, all conclusions of Theorem \ref{lem:WELL_Behaved}
hold true.
\end{thm}
Since the computations in (\ref{eq:TRIG_1}) and (\ref{eq:TRIG_2})
do not depend on ${\bf p}$ or the information collected so far, encoded
in $\mathscr{C}\left({\cal T}_{t-1}\right)$, for $t\in\mathbb{N}_{N_{T}}^{2}$,
quadrature points and the respective weights can be determined offline
and stored in memory. Therefore, the computational burden of (\ref{eq:Quadrature_1})
concentrates solely on the computation of an inner product, whose
computational complexity is of the order of ${\cal O}\left(M^{2}\right)$,
as well as a total of $M^{2}$ evaluations of $r\hspace{-2pt}\left(\hspace{-2pt}\sqrt{\boldsymbol{\Sigma}_{\left.t\right|t-1}^{F,G}\hspace{-2pt}\left({\bf p}\right)}\boldsymbol{q}_{\left(l_{1},l_{2}\right)}\hspace{-2pt}+\hspace{-2pt}\boldsymbol{\mu}_{\left.t\right|t-1}^{F,G}\hspace{-2pt}\left({\bf p}\right)\hspace{-2pt}\right)$,
for \textit{each} value of ${\bf p}$. Excluding temporarily the computational
burden of $\boldsymbol{\mu}_{\left.t\right|t-1}^{F,G}\hspace{-2pt}\left({\bf p}\right)$
and $\boldsymbol{\Sigma}_{\left.t\right|t-1}^{F,G}\hspace{-2pt}\left({\bf p}\right)$,
each of the latter evaluations is of fixed complexity, since each
involves elementary operations among matrices and vectors in $\mathbb{R}^{2\times2}$
and $\mathbb{R}^{2\times1}$, respectively and, additionally, the
involved matrix square root can be evaluated in closed form, via the
formula \cite{SQRT_1980Levinger} 
\begin{equation}
\sqrt{\boldsymbol{\Sigma}_{\left.t\right|t-1}^{F,G}\hspace{-2pt}\left({\bf p}\right)}\equiv\dfrac{\boldsymbol{\Sigma}_{\left.t\right|t-1}^{F,G}\hspace{-2pt}\left({\bf p}\right)+\sqrt{\det\left(\boldsymbol{\Sigma}_{\left.t\right|t-1}^{F,G}\hspace{-2pt}\left({\bf p}\right)\right)}{\bf I}_{2}}{\sqrt{\text{tr}\left(\boldsymbol{\Sigma}_{\left.t\right|t-1}^{F,G}\hspace{-2pt}\left({\bf p}\right)\right)+2\sqrt{\det\left(\boldsymbol{\Sigma}_{\left.t\right|t-1}^{F,G}\hspace{-2pt}\left({\bf p}\right)\right)}}}\in\mathbb{S}_{+}^{2\times2},
\end{equation}
where we have taken into account that $\boldsymbol{\Sigma}_{\left.t\right|t-1}^{F,G}\hspace{-2pt}\left({\bf p}\right)$
is always a (conditional) covariance matrix and, thus, (conditionally)
positive semidefinite. As a result and considering the last paragraph
of Section \ref{subsec:SINR_Taylor}, if (\ref{eq:Quadrature_1})
is evaluated on a finite grid of possible locations, say $P\in\mathbb{N}^{+}$,
then, at each $t-1\in\mathbb{N}_{N_{T}-1}^{+}$, the total computational
complexity of the Gauss-Hermite Quadrature Rule outlined above is
of the order of ${\cal O}\left(PM^{2}+PR^{2}t^{2}+R^{3}t^{2}\right)$
elementary operations / function evaluations. This will be the total,
worst case computational complexity for solving (\ref{eq:SINR_APPROX_QUADR}),
if the feasible set ${\cal C}_{i}$ is assumed to be finite, for all
$i\in\mathbb{N}_{R}^{+}$. As noted above, a finite feasible set greatly
simplifies implementation, since a trial-and-error approach may be
employed for solving the respective optimization problem. If $M$
is considered a fixed constant (e.g., $M\equiv10^{3}$), and the same
holds for $Rt\le RN_{T}$, then, in each time slot, the total complexity
of the Gauss-Hermite Quadrature Rule is of the order of ${\cal O}\left(P\right)$
evaluations of (\ref{eq:Quadrature_1}), that is, linear in $P$.
In that case, the whole numerical integration routine is considered
a black box of fixed computational load, which, in each time slot,
takes ${\bf p}$ as its input. Observe that, whenever $M\approx RN_{T}$,
the worst case complexity of the brute force method, described in
this subsection, \textit{over all $N_{T}$ time slots}, is essentially
the same as that of the Taylor approximation method, presented earlier
in Section \ref{subsec:SINR_Taylor}.

\subsection{\label{subsec:Theoretical-Guarantees}Theoretical Guarantees: Network
QoS Increases Across Time Slots}

The proposed relay position selection approach presented in Section
\ref{subsec:SINR-Maximization} enjoys a very important and useful
feature, initially observed via numerical simulations: Although a
$2$-stage stochastic programming procedure is utilized \textit{independently}
at each time slot for determining optimal relay positioning and beamforming
weights at the next time slot, \textit{the average network QoS (that
is, the achieved SINR) actually increases, as a function of time (the
time slot)}. Then, it was somewhat surprising to discover that, additionally,
this behavior of the achieved SINR can be predicted \textit{theoretically},
in an indeed elegant manner and, as it will be clear below, under
mild and reasonable assumptions on the structure of the spatially
controlled beamforming problem under consideration. But first, it
would be necessary to introduce the following definition.
\begin{defn}
\textbf{($\mathbf{L.MD.G}$ Fields)}\label{def:MD} On $\left(\Omega,\mathscr{F},{\cal P}\right)$,
an integrable stochastic field $\Xi:\Omega\times\mathbb{R}^{N}\times\mathbb{N}\rightarrow\mathbb{R}$
is said to be a \textit{Linear Martingale Difference (MD) Generator,
relative to a filtration }$\left\{ \mathscr{H}_{t}\subseteq\mathscr{F}\right\} _{t\in\mathbb{N}}$,
and \textit{with scaling factor} $\mu\in\mathbb{R}$, or, equivalently,
$\mathbf{L.MD.G}\diamondsuit\left(\mathscr{H}_{t},\mu\right)$, if
and only if, for each $t\in\mathbb{N}^{+}$, there exists a measurable
set $\Omega_{t}\subseteq\Omega$, with ${\cal P}\left(\Omega_{t}\right)\equiv1$,
such that, for every $\boldsymbol{x}\in\mathbb{R}^{N}$, it is true
that
\begin{equation}
\mathbb{E}\left\{ \left.\Xi\left(\boldsymbol{x},t\right)\right|\mathscr{\mathscr{H}}_{t-1}\right\} \left(\omega\right)\equiv\mu\mathbb{E}\left\{ \left.\Xi\left(\boldsymbol{x},t-1\right)\right|\mathscr{\mathscr{H}}_{t-1}\right\} \left(\omega\right),\label{eq:MD_Property}
\end{equation}
for all $\omega\in\Omega_{t}$.
\end{defn}
\begin{rem}
A fine detail in the definition of a $\mathbf{L.MD.G}\diamondsuit\text{\ensuremath{\left(\mathscr{\mathscr{H}}_{t},\mu\right)}}$
field is that, for each $t\in\mathbb{N}$, the event $\Omega_{t}$
\textit{does not depend} on the choice of point $\boldsymbol{x}\in\mathbb{R}^{N}$.
Nevertheless, even if the event where (\ref{eq:MD_Property}) is satisfied
is indeed dependent on the particular $\boldsymbol{x}\in\mathbb{R}^{N}$,
let us denote it as $\Omega_{\boldsymbol{x},t}$, we may leverage
the fact that conditional expectations are unique almost everywhere,
and \textit{arbitrarily define
\begin{equation}
\mathbb{E}\left\{ \left.\Xi\left(\boldsymbol{x},t\right)\right|\mathscr{\mathscr{H}}_{t-1}\right\} \left(\omega\right)\triangleq\mu\mathbb{E}\left\{ \left.\Xi\left(\boldsymbol{x},t-1\right)\right|\mathscr{\mathscr{H}}_{t-1}\right\} \left(\omega\right),
\end{equation}
}for all $\omega\in\Omega_{\boldsymbol{x},t}^{c}$, where ${\cal P}\left(\Omega_{\boldsymbol{x},t}^{c}\right)\equiv0$.
That is, we modify \textit{both, or either of }the random elements
$\mathbb{E}\left\{ \left.\Xi\left(\boldsymbol{x},t-1\right)\right|\mathscr{\mathscr{H}}_{t-1}\right\} $
and $\mathbb{E}\left\{ \left.\Xi\left(\boldsymbol{x},t\right)\right|\mathscr{\mathscr{H}}_{t-1}\right\} $,
on the null set $\Omega_{\boldsymbol{x},t}^{c}$, such that (\ref{eq:MD_Property})
is satisfied. Then, it may be easily verified that both such modifications
result in valid versions of the conditional expectations of $\Xi\left(\boldsymbol{x},t-1\right)$
and $\Xi\left(\boldsymbol{x},t\right)$ relative to $\mathscr{\mathscr{H}}_{t-1}$,
respectively and satisfy property (\ref{eq:MD_Property}), \textit{everywhere
with respect to} $\omega\in\Omega$.

In Definition \ref{def:MD}, invariance of $\Omega_{t}$ with respect
to $\boldsymbol{x}\in\mathbb{R}^{N}$, in conjunction with the power
of the substitution rule for conditional expectations (Section \ref{subsec:SP}),
will allow the development of strong conditional arguments, \textit{when
$\boldsymbol{x}$ is replaced by a random element}, \textit{measurable
relative to} $\mathscr{H}_{t-1}$.\hfill{}\ensuremath{\blacksquare}
\end{rem}
\begin{rem}
\label{rem:MD_Examples}There are lots of examples of $\mathbf{L.MD.G}$
stochastic fields, satisfying the technical properties of Definition
\ref{def:MD}. For completeness, let us present two such examples.
Employing generic notation, consider an integrable real-valued stochastic
field $Y\left(\boldsymbol{x},t\right)$, $\boldsymbol{x}\in\mathbb{R}^{N}$,
$t\in\mathbb{N}$. Let the natural filtration associated with $Y\left(\boldsymbol{x},t\right)$
be $\left\{ \mathscr{Y}_{t}\right\} _{t\in\mathbb{N}}$, with $\mathscr{Y}_{t}\triangleq\sigma\left\{ Y\left(\boldsymbol{x},t\right),\boldsymbol{x}\in\mathbb{R}^{N}\right\} $,
for all $t\in\mathbb{N}$. Also, consider another, for simplicity
temporal, integrable real-valued process $W\left(t\right),t\in\mathbb{N}$.
Suppose, further, that $Y\left(\boldsymbol{x},t\right)$ is a martingale
with respect to $t\in\mathbb{N}$ (relative to $\left\{ \mathscr{Y}_{t}\right\} _{t\in\mathbb{N}}$),
and that $W\left(t\right)$ is a zero mean process, independent of
$Y\left(\boldsymbol{x},t\right)$. In particular, we assume that,
for every $t\in\mathbb{N}^{+}$, there exist events $\Omega_{t}^{Y}\subseteq\Omega$
and $\Omega_{t}^{W}\subseteq\Omega$, satisfying ${\cal P}\left(\Omega_{t}^{Y}\right)\equiv1$
and ${\cal P}\left(\Omega_{t}^{W}\right)\equiv1$, such that, for
all $\boldsymbol{x}\in\mathbb{R}^{N}$,
\begin{flalign}
\mathbb{E}\left\{ \left.Y\left(\boldsymbol{x},t\right)\right|\mathscr{\mathscr{Y}}_{t-1}\right\} \left(\omega\right) & \equiv Y\left(\omega,\boldsymbol{x},t-1\right)\quad\text{and}\\
\mathbb{E}\left\{ \left.W\left(t\right)\right|\mathscr{\mathscr{Y}}_{t-1}\right\} \left(\omega\right) & \equiv0,
\end{flalign}
for all $\omega\in\Omega_{t}^{Y}\bigcap\Omega_{t}^{W}$, where, apparently,
${\cal P}\left(\Omega_{t}^{Y}\bigcap\Omega_{t}^{W}\right)\equiv1$.

Our first, probably most basic example of a $\mathbf{L.MD.G}$ field
is simply the martingale $Y\left(\boldsymbol{x},t\right)$ itself.
Of course, in order to verify this statement, we need to show that
it satisfies the technical requirements of Definition \ref{def:MD},
relative to a given filtration; in particular, let us choose $\left\{ \mathscr{Y}_{t}\right\} _{t\in\mathbb{N}}$
to be that filtration. Then, for every $\left(\boldsymbol{x},t\right)\in\mathbb{R}^{N}\times\mathbb{N}^{+}$,
it is trivial to see that 
\begin{flalign}
\mathbb{E}\left\{ \left.Y\left(\boldsymbol{x},t\right)\right|\mathscr{\mathscr{Y}}_{t-1}\right\} \left(\omega\right) & \equiv Y\left(\omega,\boldsymbol{x},t-1\right)\equiv\mathbb{E}\left\{ \left.Y\left(\boldsymbol{x},t-1\right)\right|\mathscr{\mathscr{Y}}_{t-1}\right\} \left(\omega\right),
\end{flalign}
for all $\omega\in\Omega_{t}^{Y}$, where $Y\left(\boldsymbol{x},t-1\right)$\textit{
is chosen as our version of} $\mathbb{E}\left\{ \left.Y\left(\boldsymbol{x},t-1\right)\right|\mathscr{\mathscr{Y}}_{t-1}\right\} $,
\textit{everywhere} in $\Omega$. As a result, the martingale $Y\left(\boldsymbol{x},t\right)$
is itself a $\mathbf{L.MD.G}\diamondsuit\left(\mathscr{\mathscr{Y}}_{t},1\right)$,
as expected.

The second, somewhat more interesting example of a $\mathbf{L.MD.G}$
field is defined as
\begin{equation}
X\left(\boldsymbol{x},t\right)\triangleq\varrho Y\left(\boldsymbol{x},t\right)+W\left(t\right),
\end{equation}
for all $\left(\boldsymbol{x},t\right)\in\mathbb{R}^{N}\times\mathbb{N}$,
where, say, $0<\varrho\le1$. In order to verify the technical requirements
of Definition \ref{def:MD}, let us again choose $\left\{ \mathscr{Y}_{t}\right\} _{t\in\mathbb{N}}$
as our filtration. Then, for every $\left(\boldsymbol{x},t\right)\in\mathbb{R}^{N}\times\mathbb{N}^{+}$,
there exists a measurable set $\Omega_{\boldsymbol{x},t}^{Y,W}\subseteq\Omega$,
with ${\cal P}\left(\Omega_{\boldsymbol{x},t}^{Y,W}\right)\equiv1$,
such that, for all $\omega\in\Omega_{\boldsymbol{x},t}^{Y,W}$,
\begin{flalign}
\mathbb{E}\left\{ \left.X\left(\boldsymbol{x},t\right)\right|\mathscr{\mathscr{Y}}_{t-1}\right\} \left(\omega\right) & \equiv\varrho Y\left(\omega,\boldsymbol{x},t-1\right)+\mathbb{E}\left\{ W\left(t\right)\right\} \nonumber \\
 & \equiv\varrho Y\left(\omega,\boldsymbol{x},t-1\right).
\end{flalign}
Therefore, we may choose our version for $\mathbb{E}\left\{ \left.X\left(\boldsymbol{x},t\right)\right|\mathscr{\mathscr{Y}}_{t-1}\right\} $
as 
\begin{equation}
\mathbb{E}\left\{ \left.X\left(\boldsymbol{x},t\right)\right|\mathscr{\mathscr{Y}}_{t-1}\right\} \left(\omega\right)\equiv\varrho Y\left(\omega,\boldsymbol{x},t-1\right),\quad\forall\omega\in\Omega.
\end{equation}
In exactly the same fashion, we may choose, for every $\left(\boldsymbol{x},t\right)\in\mathbb{R}^{N}\times\mathbb{N}^{+}$,
\begin{align}
\mathbb{E}\left\{ \left.X\left(\boldsymbol{x},t-1\right)\right|\mathscr{\mathscr{Y}}_{t-1}\right\} \left(\omega\right) & \equiv\varrho Y\left(\omega,\boldsymbol{x},t-1\right),\quad\forall\omega\in\Omega.
\end{align}
Consequently, for every $\left(\boldsymbol{x},t\right)\in\mathbb{R}^{N}\times\mathbb{N}^{+}$,
it will be true that
\begin{equation}
\mathbb{E}\left\{ \left.X\left(\boldsymbol{x},t\right)\right|\mathscr{\mathscr{Y}}_{t-1}\right\} \left(\omega\right)\equiv\varrho Y\left(\omega,\boldsymbol{x},t-1\right)\equiv\mathbb{E}\left\{ \left.X\left(\boldsymbol{x},t-1\right)\right|\mathscr{\mathscr{Y}}_{t-1}\right\} \left(\omega\right),
\end{equation}
for all $\omega\in\Omega$, showing that the field $X\left(\boldsymbol{x},t\right)$
is also $\mathbf{L.MD.G}\diamondsuit\left(\mathscr{\mathscr{Y}}_{t},1\right)$.\hfill{}\ensuremath{\blacksquare}
\end{rem}
Leveraging the notion of a $\mathbf{L.MD.G}$ field, the following
result may be proven, characterizing the temporal (in discrete time)
evolution of the objective of myopic stochastic programs of the form
of (\ref{eq:2STAGE-1}). In order to introduce the result, let us
consider the family $\left\{ \mathscr{P}_{t}^{\uparrow}\right\} _{t\in\mathbb{N}_{N_{T}}^{+}}$,
with $\mathscr{P}_{t}^{\uparrow}$ being the\textit{ limit $\sigma$-algebra
generated by all admissible policies at time slot $t$, }defined as\textit{
\begin{flalign}
\mathscr{P}_{t}^{\uparrow} & \triangleq\sigma\left\{ \bigcup_{{\bf p}\left(t\right)\in{\cal D}_{t}}\sigma\left\{ {\bf p}\left(t\right)\right\} \right\} \subseteq\mathscr{C}\left({\cal T}_{t-1}\right),\quad\forall t\in\mathbb{N}_{N_{T}}^{+},
\end{flalign}
}with $\mathscr{P}_{1}^{\uparrow}$ being the trivial $\sigma$-algebra;
recall that ${\bf p}\left(1\right)\in{\cal S}^{R}$ is assumed to
be a constant. Also, for every $t\in\mathbb{N}_{N_{T}}^{+}$, let
us define the class
\begin{equation}
\overline{{\cal D}}_{t}\equiv\left\{ {\bf p}:\Omega\rightarrow{\cal S}^{R}\left|{\bf p}^{-1}\left({\cal A}\right)\in\mathscr{P}_{t}^{\uparrow},\text{ for all }{\cal A}\in\mathscr{B}\left({\cal S}^{R}\right)\right.\hspace{-2pt}\right\} .
\end{equation}
The result now follows.
\begin{thm}
\textbf{\textup{($\mathbf{L.MD.G}$ Objectives Increase over Time)\label{lem:QoS_INCREASES}}}
Consider, for each $t\in\mathbb{N}_{N_{T}}^{2}$, the maximization
version of the $2$-stage stochastic program (\ref{eq:2STAGE-1}),
for some choice of the second-stage optimal value $V\left({\bf p},t\right)$,
${\bf p}\in{\cal S}^{R}$, $t\in\mathbb{N}_{N_{T}}^{2}$. Suppose
that conditions \textbf{\textup{C1-C6}} are satisfied at all times
and let ${\bf p}^{*}\left(t\right)$ denote an optimal solution to
(\ref{eq:2STAGE-1}), decided at $t-1\in\mathbb{N}_{N_{T}-1}^{+}$.
Suppose, further, that, for every $t\in\mathbb{N}_{N_{T}}^{+}$,
\begin{itemize}
\item $V\left({\bf p},t\right)$ is $\mathbf{L.MD.G}\diamondsuit\left(\mathscr{H}_{t},\mu\right)$,
for a filtration $\left\{ \mathscr{H}_{t}\supseteq\mathscr{P}_{t}^{\uparrow}\right\} _{t\in\mathbb{N}_{N_{T}}^{+}}$
and some $\mu\in\mathbb{R}$, and that
\item $V\left(\cdot,\cdot,t\right)$ is both $\boldsymbol{SP}\diamondsuit\mathfrak{C}_{\mathscr{H}_{t}}$
and $\boldsymbol{SP}\diamondsuit\mathfrak{C}_{\mathscr{H}_{t-1}}$,
with $\overline{{\cal D}}_{t}\subseteq\mathfrak{C}_{\mathscr{H}_{t}}\subseteq\mathfrak{I}_{\mathscr{H}_{t}}$
(Remark \ref{rem:FULLSP_1} / Section \ref{subsec:SP}).
\end{itemize}
Then, for any admissible policy ${\bf p}^{o}\left(t-1\right)$, it
is true that
\begin{flalign}
\mu\mathbb{E}\left\{ V\left({\bf p}^{o}\left(t-1\right),t-1\right)\right\}  & \equiv\mathbb{E}\left\{ V\left({\bf p}^{o}\left(t-1\right),t\right)\right\} \quad\text{and}\\
\mu\mathbb{E}\left\{ V\left({\bf p}^{*}\left(t-1\right),t-1\right)\right\}  & \le\mathbb{E}\left\{ V\left({\bf p}^{*}\left(t\right),t\right)\right\} ,\quad\forall t\in\mathbb{N}_{N_{T}}^{2}.
\end{flalign}
In particular, if $\mu\equiv1$, the objective: $\bullet$ does not
decrease by not updating the decision variable, and $\bullet$ is
nondecreasing over time, under optimal decision making.
\end{thm}
\begin{proof}[Proof of Theorem \ref{lem:QoS_INCREASES}]
See Appendix C.
\end{proof}
Of all possible choices for $\mu$, the one where $\mu\equiv1$ is
of special importance and practical relevance, as we will see in the
next. In particular, in this case, and provided that the respective
assumptions are fulfilled, Theorem \ref{lem:QoS_INCREASES} implies
that \textit{optimal myopic exploration of the random field $V\left({\bf p},t\right)$
is monotonic}, either under optimal decision making, or by retaining
the same policy next.

In the case where $\mu\neq1$, things can be quite interesting as
well. For instance, suppose that one focuses on the maximization counterpart
of the stochastic program (\ref{eq:2STAGE-1}). In this case, it is
of interest to sequentially, myopically and feasibly sample the field
$V\left({\bf p},t\right)$, such that it is \textit{maximized on average}.
Let us also refer to $V\left({\bf p},t\right)$ as the \textit{reward}
of the sampling process. Additionally, suppose that $V\left({\bf p},t\right)$
is a $\mathbf{L.MD.G}$ field, with parameter $\mu\equiv0.9<1$. Assuming
that the respective assumptions are satisfied, Theorem \ref{lem:QoS_INCREASES}
implies that, for any admissible sampling policy ${\bf p}^{o}\left(t-1\right)$,
\begin{flalign}
\mathbb{E}\hspace{-2pt}\left\{ V\left({\bf p}^{o}\left(t-1\right),t\right)\right\}  & \hspace{-2pt}\equiv\hspace{-2pt}0.9\mathbb{E}\hspace{-2pt}\left\{ V\left({\bf p}^{o}\left(t-1\right),t-1\right)\right\} \,\text{and}\hspace{-2pt}\hspace{-2pt}\\
\mathbb{E}\hspace{-2pt}\left\{ V\left({\bf p}^{*}\left(t\right),t\right)\right\}  & \hspace{-2pt}\ge\hspace{-2pt}0.9\mathbb{E}\hspace{-2pt}\left\{ V\left({\bf p}^{*}\left(t-1\right),t-1\right)\right\} ,\label{eq:INC-1}
\end{flalign}
for all $t\in\mathbb{N}_{N_{T}}^{2}$. In other words, either performing
optimal decision making, or retaining the same policy next, will result
in an \textit{at most $10\%$ loss of performance}. This means that
the performance of optimal sampling at the next time step cannot be
worse than $90\%$ of that at the current time slot. Of course, this
is important, because, in a sense, \textit{the risk of (non-)maintaining
the average reward achieved up to the current time slot is meaningfully
quantified}.
\begin{rem}
\label{rem:MD_Separable}When the stochastic program under study is
\textit{separable}, that is, when the objective is of the form
\begin{equation}
V\left({\bf p}\left(t\right),t\right)\equiv\sum_{i\in\mathbb{N}_{M}^{+}}V_{i}\left({\bf p}_{i}\left(t\right),t\right)
\end{equation}
(and the respective constraints of the problem decoupled), then, in
order to reach the conclusions of Theorem \ref{lem:QoS_INCREASES}
for $V$, it \textit{suffices} for Theorem \ref{lem:QoS_INCREASES}
to hold \textit{individually} for each $V_{i}$, $i\in\mathbb{N}_{M}^{+}$.
This is true, for instance, for the spatially controlled beamforming
problem (\ref{eq:2STAGE-SINR-2}).\hfill{}\ensuremath{\blacksquare}
\end{rem}
We may now return to the beamforming problem under consideration,
namely (\ref{eq:2STAGE-SINR-2}). By Remark \ref{rem:MD_Separable}
and Theorem \ref{lem:QoS_INCREASES}, it would suffice if we could
show that the field $V_{I}\left({\bf p},t\right)$ is a linear MD
generator, relative to a properly chosen filtration. Unfortunately,
though, this does not seem to be the case; the statistical structure
of $V_{I}\left({\bf p},t\right)$ does not match that of a linear
MD generator \textit{exactly}, relative to any reasonably chosen filtration.
Nevertheless, under the channel model of Section \ref{sec:Spatiotemporal-Wireless-Channel},
it is indeed possible to show that $V_{I}\left({\bf p},t\right)$
is \textit{approximately} $\mathbf{L.MD.G}\diamondsuit\left(\mathscr{C}\left({\cal T}_{t-1}\right),1\right)$,
a fact that explains, in an elegant manner, why our proposed spatially
controlled beamforming framework is expected to work so well, both
under optimal and suboptimal decision making.

To show that $V_{I}\left({\bf p},t\right)$ is \textit{approximately}
$\mathbf{L.MD.G}\diamondsuit\left(\mathscr{C}\left({\cal T}_{t-1}\right),1\right)$,
simply consider projecting $V_{I}\left({\bf p},t-1\right)$ onto $\mathscr{C}\left({\cal T}_{t-2}\right)$,
via the conditional expectation $\mathbb{E}\left\{ \left.V_{I}\left({\bf p},t-1\right)\right|\mathscr{C}\left({\cal T}_{t-2}\right)\right\} $.
Of course, and based on what we have seen so far, $\mathbb{E}\left\{ \left.V_{I}\left({\bf p},t-1\right)\right|\mathscr{C}\left({\cal T}_{t-2}\right)\right\} $
can be written as a Lebesgue integral of $V_{I}\left({\bf p},t-1\right)$
expressed in terms of the vector field $\left[F\left({\bf p},t-1\right)\,G\left({\bf p},t-1\right)\right]^{\boldsymbol{T}}$,
times its conditional density relative to $\mathscr{C}\left({\cal T}_{t-2}\right)$.
It then easy to see that this conditional density will be, of course,
Gaussian, and will be of exactly the same form as the conditional
density of $\left[F\left({\bf p},t\right)\,G\left({\bf p},t\right)\right]^{\boldsymbol{T}}$
relative to $\mathscr{C}\left({\cal T}_{t-1}\right)$, as presented
in Lemma \ref{thm:(Big-Expectations)}, but with $t$ replaced by
$t-1$. Likewise, $\mathbb{E}\left\{ \left.V_{I}\left({\bf p},t\right)\right|\mathscr{C}\left({\cal T}_{t-2}\right)\right\} $
is of the same form as $\mathbb{E}\left\{ \left.V_{I}\left({\bf p},t-1\right)\right|\mathscr{C}\left({\cal T}_{t-2}\right)\right\} $,
but with all terms
\begin{equation}
\exp\left(-\dfrac{1}{\gamma}\right),\exp\left(-\dfrac{2}{\gamma}\right),\ldots,\exp\left(-\dfrac{t-2}{\gamma}\right)
\end{equation}
simply replaced by 
\begin{equation}
\exp\left(-\dfrac{2}{\gamma}\right),\exp\left(-\dfrac{3}{\gamma}\right),\ldots,\exp\left(-\dfrac{t-1}{\gamma}\right),
\end{equation}
for all $t\in\mathbb{N}_{N_{T}}^{3}$. Of course, if $t\equiv2$,
we have
\begin{flalign}
\mathbb{E}\left\{ \left.V_{I}\left({\bf p},2\right)\right|\mathscr{C}\left({\cal T}_{0}\right)\right\}  & \equiv\mathbb{E}\left\{ V_{I}\left({\bf p},2\right)\right\} \nonumber \\
 & \equiv\mathbb{E}\left\{ V_{I}\left({\bf p},1\right)\right\} \equiv\mathbb{E}\left\{ \left.V_{I}\left({\bf p},1\right)\right|\mathscr{C}\left({\cal T}_{0}\right)\right\} .
\end{flalign}
Now, for $\gamma$ \textit{sufficiently large}, we may approximately
write
\begin{equation}
\exp\left(-\dfrac{x+1}{\gamma}\right)\approx\exp\left(-\dfrac{x}{\gamma}\right),\quad\forall x>1,
\end{equation}
and, therefore, due to continuity, it should be true that
\begin{equation}
\mathbb{E}\left\{ \left.V_{I}\left({\bf p},t\right)\right|\mathscr{C}\left({\cal T}_{t-2}\right)\right\} \approx\mathbb{E}\left\{ \left.V_{I}\left({\bf p},t-1\right)\right|\mathscr{C}\left({\cal T}_{t-2}\right)\right\} ,
\end{equation}
for all $t\in\mathbb{N}_{N_{T}}^{2}$ (and everywhere with respect
to $\omega\in\Omega$). As a result, we have shown that, \textit{at
least approximately}, $V_{I}\left({\bf p},t\right)$ is $\mathbf{L.MD.G}\diamondsuit\left(\mathscr{C}\left({\cal T}_{t-1}\right),1\right)$.
We may then invoke Theorem \ref{lem:QoS_INCREASES} in an approximate
manner, leading to the following important result. Hereafter, for
$x\in\mathbb{R}$ and $y\in\mathbb{R}$, $x\lesssim y$ will imply
that $x$ \textit{is approximately smaller or equal than} $y$, in
the sense that $x\le y+\varepsilon$, where $\varepsilon>0$ is some
small slack.
\begin{thm}
\textbf{\textup{(QoS Increases over Time Slots)\label{prop:QoS-Increases}}}
Consider the separable stochastic program (\ref{eq:2STAGE-SINR-2}).
For $\gamma$ sufficiently large, and for any admissible policy ${\bf p}^{o}\left(t-1\right)$,
it is true that
\begin{flalign}
\mathbb{E}\left\{ V_{I}\hspace{-2pt}\left({\bf p}_{i}^{o}\left(t-1\right),t-1\right)\right\}  & \approx\mathbb{E}\left\{ V_{I}\hspace{-2pt}\left({\bf p}_{i}^{o}\left(t-1\right),t\right)\right\} ,\\
\mathbb{E}\left\{ V_{I}\hspace{-2pt}\left({\bf p}_{i}^{*}\left(t-1\right),t-1\right)\right\}  & \lesssim\mathbb{E}\left\{ V_{I}\hspace{-2pt}\left({\bf p}_{i}^{*}\left(t\right),t\right)\right\} ,\quad\forall i\in\mathbb{N}_{R}^{+}\\
\mathbb{E}\left\{ V\hspace{-2pt}\left({\bf p}^{o}\left(t-1\right),t-1\right)\right\}  & \approx\mathbb{E}\left\{ V\hspace{-2pt}\left({\bf p}^{o}\left(t-1\right),t\right)\right\} \quad\text{and}\\
\mathbb{E}\left\{ V\hspace{-2pt}\left({\bf p}^{*}\left(t-1\right),t-1\right)\right\}  & \lesssim\mathbb{E}\left\{ V\hspace{-2pt}\left({\bf p}^{*}\left(t\right),t\right)\right\} ,
\end{flalign}
for all $t\in\mathbb{N}_{N_{T}}^{2}$. In other words, \textbf{approximately},
the average network QoS: $\bullet$ does not decrease by not updating
the positions of the relays and $\bullet$ is nondecreasing across
time slots, under (per relay) optimal decision making.
\end{thm}
Theorem \ref{prop:QoS-Increases} is very important from a practical
point of view, and has the following additional implications. Roughly
speaking, under the conditions of Theorem \ref{prop:QoS-Increases},
that is, if the temporal interactions of the channel are sufficiently
strong, the average network QoS is not (approximately) expected to,
at least \textit{abruptly,} decrease if one or more relays stop moving
at some point. Such event might indeed happen in an actual autonomous
network, possibly due to power limitations, or a failure in the motion
mechanisms of some network nodes. In the same framework, Theorem \ref{prop:QoS-Increases}
implies that the relays which continue moving contribute (approximately)
positively to increasing the average network QoS, across time slots.
Such behavior of the proposed spatially controlled beamforming system
may be also confirmed numerically, as discussed in Section \ref{sec:Numerical-Simulations}.
For the record, and as it will be also shown in Section \ref{sec:Numerical-Simulations},
relatively small values for the correlation time $\gamma$, such as
$\gamma\equiv5$, are sufficient in order to practically observe the
nice system behavior promised by Theorem \ref{prop:QoS-Increases}.
This fact makes the proposed spatially controlled beamforming system
attractive in terms of practical feasibility, and shows that such
an approach could actually enhance system performance in a well-behaved,
real world situation. 

\section{\label{sec:Numerical-Simulations}Numerical Simulations \& Experimental
Validation}

In this section, we present synthetic numerical simulations, which
essentially confirm that the proposed approach, previously presented
in Section \ref{sec:MobRelBeam}, actually works, and results in relay
motion control policies, which yield improved beamforming performance.
All synthetic experiments were conducted on an imaginary square terrain
of dimensions $30\times30$ squared units of length, with ${\cal W}\equiv\left[0,30\right]^{2}$,
uniformly divided into $30\times30\equiv900$ square regions. The
locations of the source and destination are fixed as ${\bf p}_{S}\equiv\left[15\,0\right]^{\boldsymbol{T}}$
and ${\bf p}_{D}\equiv\left[15\,30\right]^{\boldsymbol{T}}$. The
beamforming temporal horizon is chosen as $T\equiv40$ and the number
of relays is fixed at $R\equiv8$. The wavelength is chosen as $\lambda\equiv0.125$,
corresponding to a carrier frequency of $2.4\,GHz$. The various parameters
of the assumed channel model are set as $\ell\equiv3$, $\rho\equiv20$,
$\sigma_{\xi}^{2}\equiv20$, $\eta^{2}\equiv50$, $\beta\equiv10$,
$\gamma\equiv5$ and $\delta\equiv1$. The variances of the reception
noises at the relays and the destination are fixed as $\sigma^{2}\equiv\sigma_{D}^{2}\equiv1$.
Lastly, both the transmission power of the source and the \textit{total}
transmission power budget of the relays are chosen as $P\equiv P_{c}\equiv25$
($\approx14dB$) units of power.
\begin{figure}
\centering\includegraphics[bb=10bp 0bp 420bp 315bp,scale=0.8]{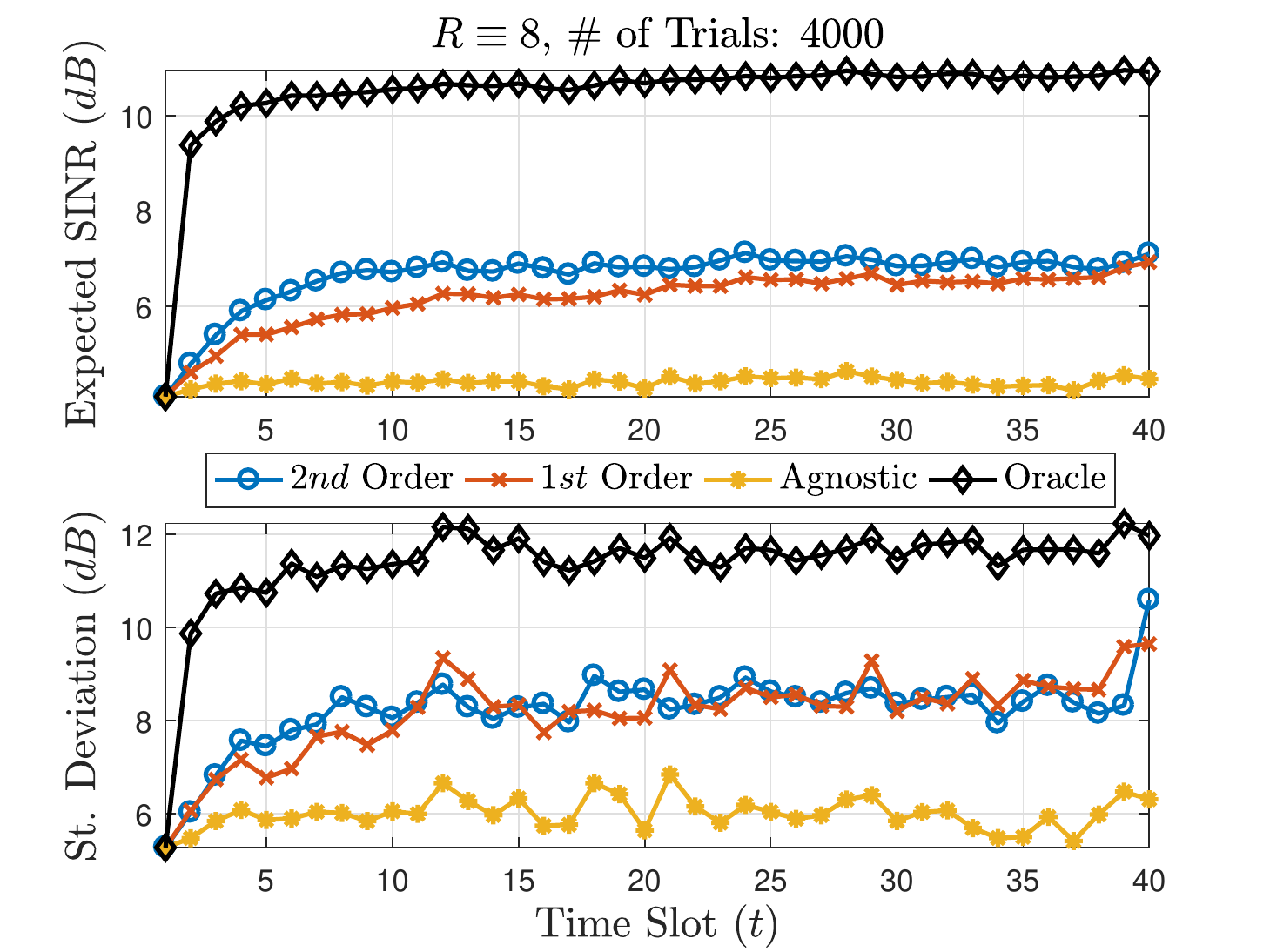}

\caption{\label{fig:Sim_1}Comparison of the proposed strategic relay planning
schemes, versus an agnostic, randomized relay motion policy.}
\end{figure}

The relays are restricted to the rectangular region ${\cal S}\equiv\left[0,30\right]\times\left[12,18\right]$.
Further, at each time instant, each of the relays is allowed to move
inside a $9$-region area, centered at each current position, thus
defining its closed set of feasible directions ${\cal C}_{i}$, for
each relay $i\in\mathbb{N}_{R}^{+}$. Basic collision and out-of-bounds
control was also considered and implemented.

In order to assess the effectiveness of our proposed approach, we
compare both heuristics (\ref{eq:SINR_APPROX_PROG_1}) and (\ref{eq:SINR_APPROX_PROG_2})
against the case where an \textit{agnostic}, \textit{purely randomized}
relay control policy is adopted; in this case, at each time slot,
each relay moves randomly to a new available position, without taking
previously observed CSI into consideration. For simplicity, we do
not consider the brute force method presented earlier in Section \ref{subsec:Brute-Force}.
For reference, we also consider the performance of an oracle control
policy at the relays, where, at each time slot $t-1\in\mathbb{N}_{N_{T}-1}^{+}$,
relay $i\in\mathbb{N}_{R}^{+}$ updates its position by \textit{noncausally}
looking into the future and choosing the position ${\bf p}_{i}$,
which maximizes directly the quantity $V_{I}\left({\bf p}_{i},t\right)$,
over ${\cal C}_{i}\left({\bf p}_{i}\left(t-1\right)\right)$. Of course,
the comparison of all controlled systems is made under exactly the
same communication environment.
\begin{figure}[!t]
\vspace{-10bp}
\negmedspace{}\subfloat[\label{fig:Sim_2-1}]{\centering\includegraphics[bb=10bp 0bp 420bp 315bp,scale=0.61]{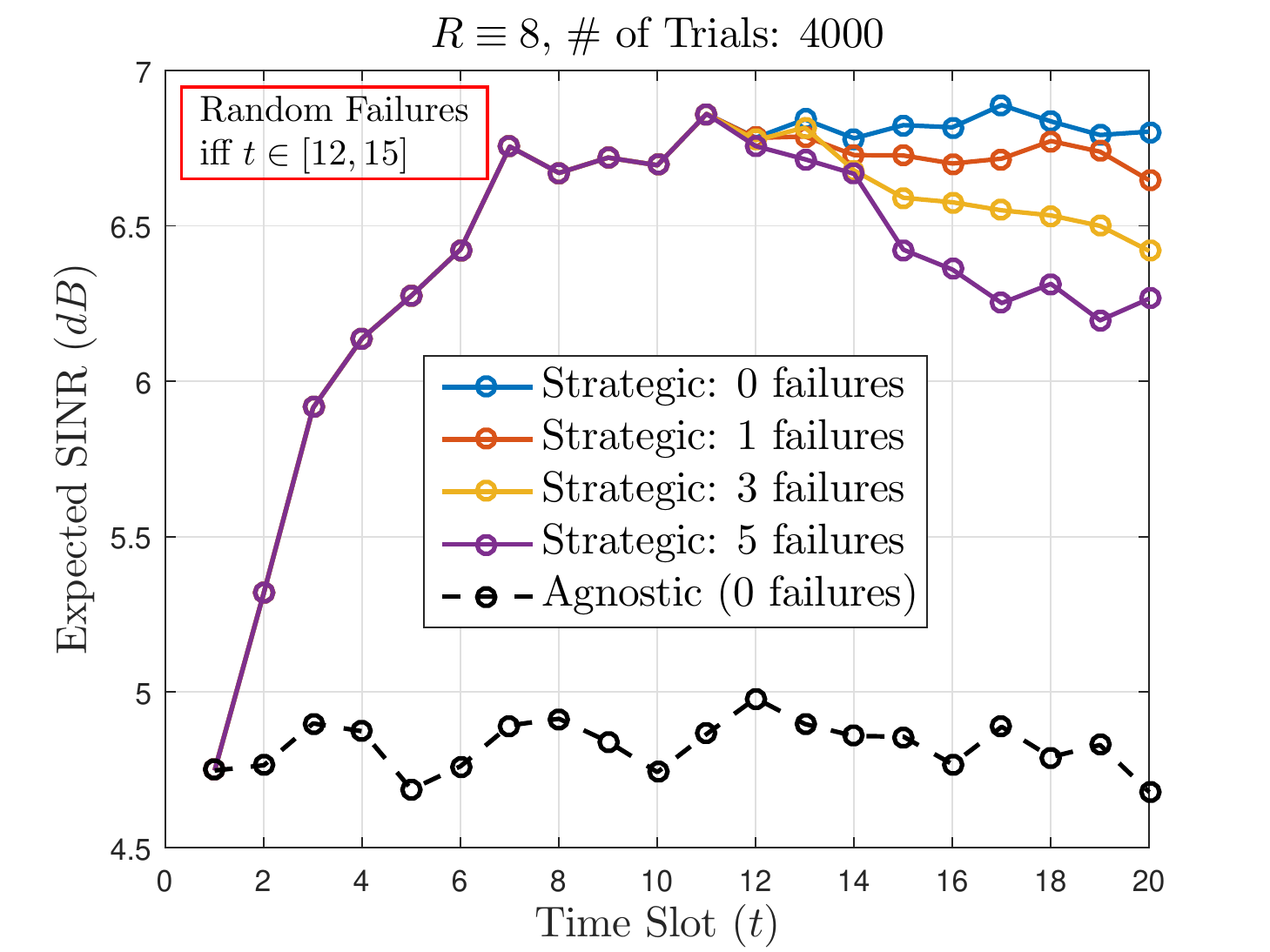}}\negmedspace{}\negmedspace{}\negmedspace{}\negmedspace{}\negmedspace{}\subfloat[\label{fig:Sim_2-2}]{\centering\includegraphics[bb=10bp 0bp 420bp 315bp,scale=0.61]{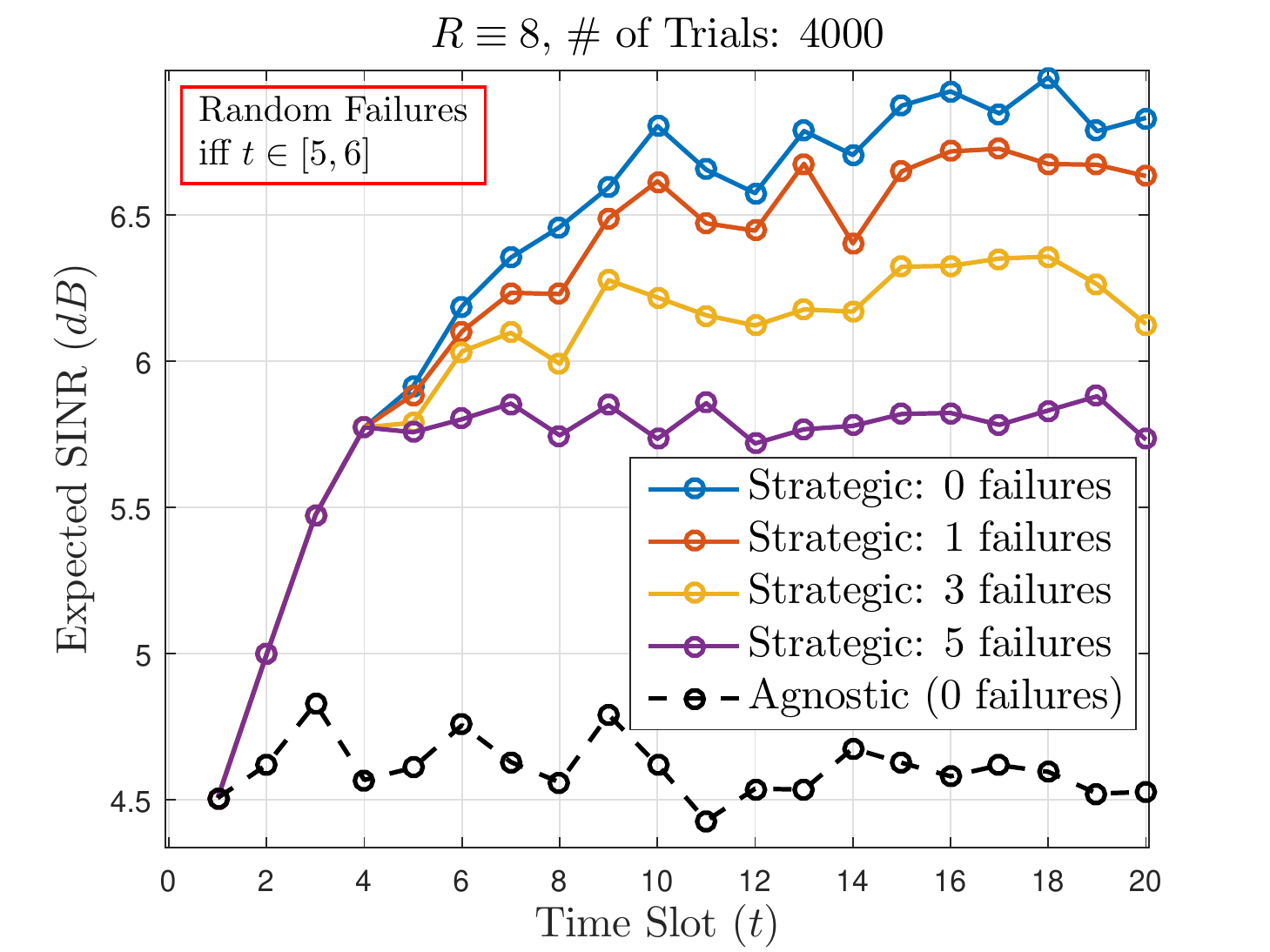}}\vspace{-5bp}
\\
\negmedspace{}\subfloat[\label{fig:Sim_2-3}]{\centering\includegraphics[bb=10bp 0bp 420bp 315bp,scale=0.61]{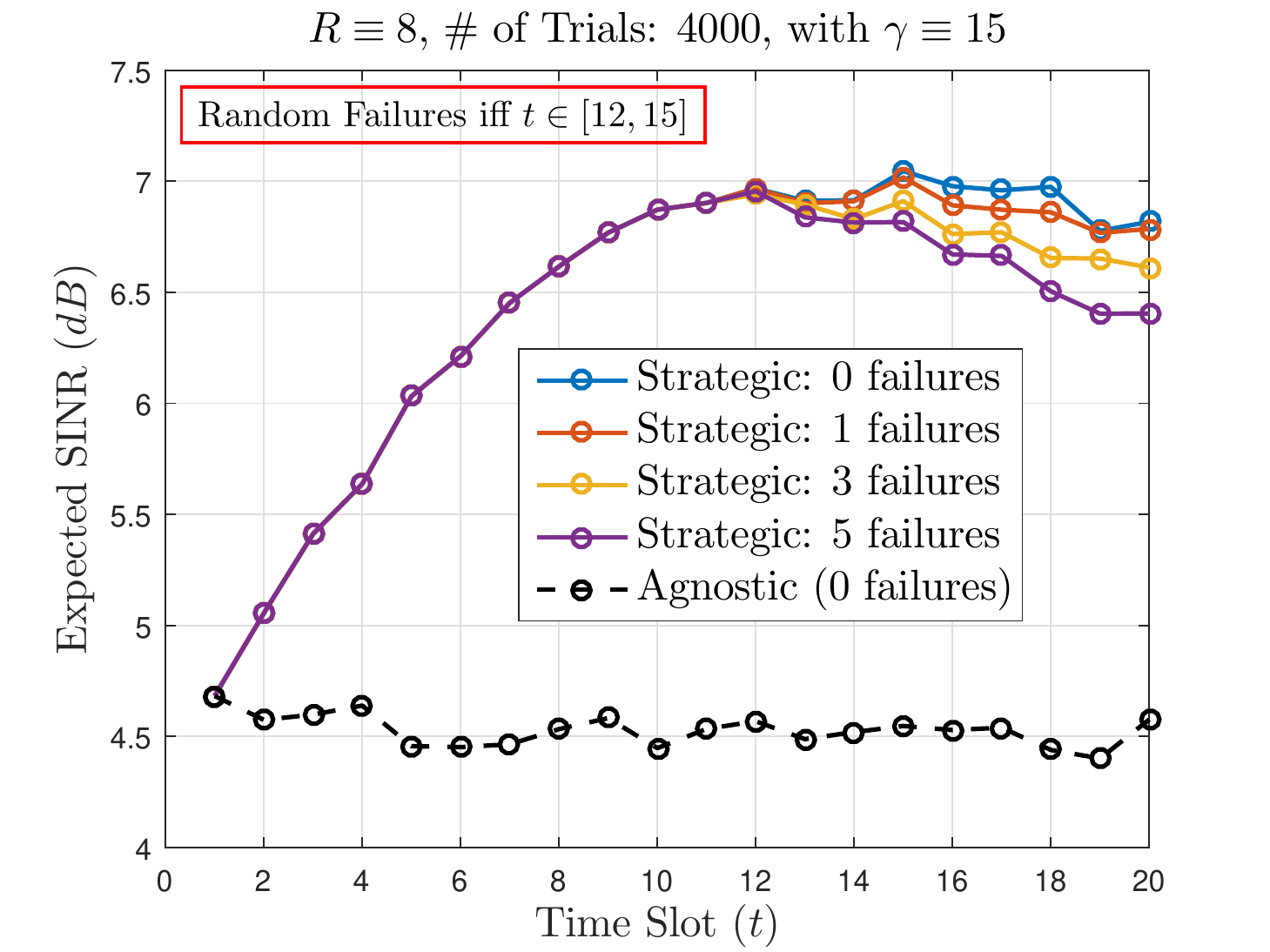}}\negmedspace{}\negmedspace{}\negmedspace{}\negmedspace{}\negmedspace{}\subfloat[\label{fig:Sim_2-4}]{\centering\includegraphics[bb=10bp 0bp 420bp 315bp,scale=0.61]{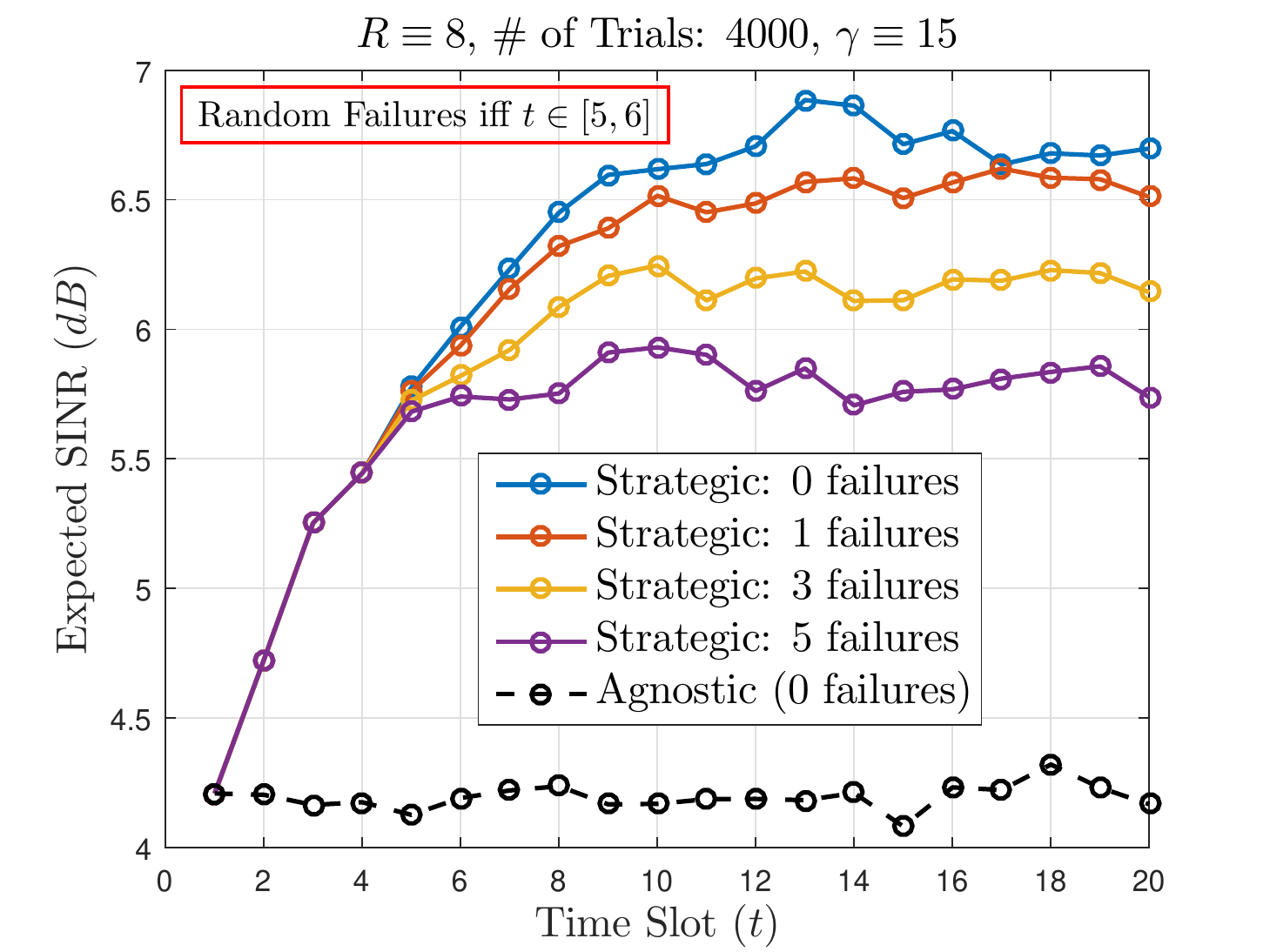}}

\caption{\label{fig:Sim_2-5}Performance of the proposed spatially controlled
system, at the presence of motion failures.}

\vspace{-10bp}
\end{figure}

Fig. \ref{fig:Sim_1} shows the expectation and standard deviation
of the achieved QoS for all controlled systems, approximated by executing
$4000$ trials of the whole experiment. As seen in the figure, there
is a clear advantage in exploiting strategically designed relay motion
control. Whereas the agnostic system maintains an average SINR of
about $4\,dB$ at all times, the system based on the proposed $2nd$
order heuristic (\ref{eq:SINR_APPROX_PROG_2}) is clearly superior,
exhibiting an increasing trend in the achieved SINR, with a gap starting
from about $0.5\,dB$ at time slot $t\equiv2$, up to $3\,dB$ at
time slots $t\equiv10,11,\ldots,40$. The $1st$ order heuristic (\ref{eq:SINR_APPROX_PROG_1})
comes second, with always slightly lower average SINR, and which also
exhibits a similar increasing trend as the $2nd$ order heuristic
(\ref{eq:SINR_APPROX_PROG_2}). Additionally, it seems to converge
to the performance achieved by (\ref{eq:SINR_APPROX_PROG_2}), across
time slots. The existence of an increasing trend in the achieved average
network QoS has already been predicted by Theorem \ref{lem:QoS_INCREASES}
\textit{for a strictly optimal policy}, and our experiments confirm
this behavior for both heuristics (\ref{eq:SINR_APPROX_PROG_1}) and
(\ref{eq:SINR_APPROX_PROG_2}), as well. This shows that both heuristics
constitute excellent approximations to the original problem (\ref{eq:SINR_DIST}).
Consequently, it is both theoretically and experimentally verified
that, although the proposed stochastic programming formulation is
essentially myopic, the resulting system performance is not, and this
is dependent on the fact that the channel exhibits non trivial temporal
statistical interactions. We should also comment on the standard deviation
of all systems, which, from Fig. \ref{fig:Sim_1}, seems somewhat
high, relative to the range of the respective average SINR. This is
\textit{exclusively} due to the wild variations of the channel, which,
in turn, are due to the effects of shadowing and multipath fading;
it is \textit{not} due to the adopted beamforming technique. This
is reasonable, since, when the channel is not \textit{actually} in
deep fade at time $t$ (an event which might happen with positive
probability), the relays, at time $t-1$, are predictively steered
to locations, which, most probably, incur higher network QoS. As clearly
shown in Fig. \ref{fig:Sim_1}, for all systems under study, including
that implementing the oracle policy, an increase in system performance
also implies a proportional increase in the respective standard deviation.

Next, we experimentally evaluate the performance of the system at
the presence of \textit{random motion failures} in the network. Hereafter,
we work with the $2nd$ order heuristic (\ref{eq:SINR_APPROX_PROG_2}),
and set $T\equiv20$. Random motion failures are modeled by choosing,
at each trial, a \textit{random sample} of a fixed number of relays
and a \textit{random} \textit{time} when the failures occur, that
is, at each time, the selected relays just stop moving; they continue
to beamform staying still, at the position each of them visited last.
Two cases are considered; in the first case, motion failures happen
if and only if $t\in\left[12,15\right]$ (Figs. \ref{fig:Sim_2-1}
and \ref{fig:Sim_2-3}), whereas, in the second case, $t\in\left[5,6\right]$
(Figs. \ref{fig:Sim_2-2} and \ref{fig:Sim_2-4}). In both cases,
zero, one, three and five relays (chosen at random, at each trial)
stop moving. Two cases for $\gamma$ are considered, $\gamma\equiv5$
(Figs. \ref{fig:Sim_2-1} and \ref{fig:Sim_2-2}) and $\gamma\equiv15$
(Figs. \ref{fig:Sim_2-3} and \ref{fig:Sim_2-4}).

Again, the results presented in Fig. \ref{fig:Sim_2-5} pleasingly
confirm our predictions implied by Theorem \ref{lem:QoS_INCREASES}
(note, however, that Theorem \ref{lem:QoS_INCREASES} does not support
\textit{randomized} motion failures; on the other hand, our simulations
are such in order to stress test the proposed system in more adverse
motion failure cases). In particular, Fig. \ref{fig:Sim_2-1} clearly
demonstrates that a larger number of motion failures induces a proportional,
relatively (depending on $\gamma$) slight decrease in performance;
this decrease, though, is smoothly evolving, and is not abrupt. This
behavior is more pronounced in Fig. \ref{fig:Sim_2-3}, where the
correlation time parameter $\gamma$ has been increased to $15$ (recall
that, in Theorem \ref{lem:QoS_INCREASES}, $\gamma$ is assumed to
be sufficiently large). We readily observe that, in this case, over
the same horizon, the operation of the system is smoother, and decrease
in performance, as well as its slope, are significantly smaller than
those in Fig. \ref{fig:Sim_2-1}, for all cases of motion failures.
Now, in Figs. \ref{fig:Sim_2-2} and \ref{fig:Sim_2-4}, when motion
failures happen early, well before the network QoS converges to its
maximal value, we observe that, although some relays might stop moving
at some point, the achieved expected network QoS continues exhibiting
its usual increasing trend. Of course, the performance of the system
converges values strictly proportional to the number of failures in
each of the cases considered. This means that the relays which continue
moving contribute positively to increasing network QoS. This has been
indeed predicted by Theorem \ref{lem:QoS_INCREASES}, as well.

\section{Conclusions}

We have considered the problem of enhancing QoS in time slotted relay
beamforming networks with one source/destination, via stochastic relay
motion control. Modeling the wireless channel as a spatiotemporal
stochastic field, we proposed a novel $2$-stage stochastic programming
formulation for predictively specifying relay positions, such that
the future expected network QoS is maximized, based on causal CSI
and under a total relay power constraint. We have shown that this
problem can be effectively approximated by a set of simple, two dimensional
subproblems, which can be distributively solved, one at each relay.
System optimality was tediously analyzed under a rigorous mathematical
framework, and our analysis resulted in the development of an extended
version of the Fundamental Lemma of Stochastic Control, which constitutes
a result of independent interest, as well. We have additionally provided
strong theoretical guarantees, characterizing the performance of the
proposed system, and showing that the average QoS achieved improves
over time. Our simulations confirmed the success of the proposed approach,
which results in relay motion control policies yielding significant
performance improvement, when compared to agnostic, randomized relay
motion.

\section{Acknowledgments}

Dionysios Kalogerias would like to kindly thank Dr. Nikolaos Chatzipanagiotis
for very fruitful discussions in the very early stages of the development
of this work, and Ioannis Manousakis and Ioannis Paraskevakos for
their very useful comments and suggestions, especially concerning
practical applicability, implementation of the proposed methods, as
well as simulation issues.

\section{Appendices}

\subsection{Appendix A: Proofs / Section \ref{sec:Spatiotemporal-Wireless-Channel} }

\subsubsection{Proof of Lemma \ref{lem:NonSingularity}}


In the following, we will rely on an \textit{incremental} construction
of $\boldsymbol{\Sigma}$. Initially, consider the matrix
\begin{equation}
\widetilde{\boldsymbol{\Sigma}}\triangleq\begin{bmatrix}\widetilde{\boldsymbol{\Sigma}}\left(1,1\right) & \widetilde{\boldsymbol{\Sigma}}\left(1,2\right) & \ldots & \widetilde{\boldsymbol{\Sigma}}\left(1,N_{T}\right)\\
\widetilde{\boldsymbol{\Sigma}}\left(2,1\right) & \widetilde{\boldsymbol{\Sigma}}\left(2,2\right) & \ldots & \widetilde{\boldsymbol{\Sigma}}\left(2,N_{T}\right)\\
\vdots & \vdots & \ddots & \vdots\\
\widetilde{\boldsymbol{\Sigma}}\left(N_{T},1\right) & \widetilde{\boldsymbol{\Sigma}}\left(N_{T},2\right) & \cdots & \widetilde{\boldsymbol{\Sigma}}\left(N_{T},N_{T}\right)
\end{bmatrix}\in\mathbb{S}^{RN_{T}},
\end{equation}
where, for each combination $\left(k,l\right)\in\mathbb{N}_{N_{T}}^{+}\times\mathbb{N}_{N_{T}}^{+}$,
$\widetilde{\boldsymbol{\Sigma}}\left(k,l\right)\in\mathbb{S}^{R}$,
with
\begin{flalign}
\widetilde{\boldsymbol{\Sigma}}\left(k,l\right)\left(i,j\right) & \triangleq\widetilde{\boldsymbol{\Sigma}}\left({\bf p}_{i}\left(k\right),{\bf p}_{j}\left(l\right)\right)\nonumber \\
 & \triangleq\eta^{2}\exp\left(-\frac{\left\Vert {\bf p}_{i}\left(k\right)-{\bf p}_{j}\left(l\right)\right\Vert _{2}}{\beta}\right),
\end{flalign}
for all $\left(i,j\right)\in\mathbb{N}_{R}^{+}\times\mathbb{N}_{R}^{+}$.
By construction, $\widetilde{\boldsymbol{\Sigma}}$ is positive semidefinite,
because the well known \textit{exponential kernel} $\widetilde{\boldsymbol{\Sigma}}:\mathbb{R}^{2}\times\mathbb{R}^{2}\rightarrow\mathbb{R}_{++}$
defined above is positive (semi)definite.

Next, define the positive definite matrix 
\begin{flalign}
{\bf K} & \triangleq\begin{bmatrix}1 & \kappa\\
\kappa & 1
\end{bmatrix},\quad\text{with}\\
\kappa & \triangleq\exp\left(-\dfrac{\left\Vert {\bf p}_{S}-{\bf p}_{D}\right\Vert _{2}}{\delta}\right)<1
\end{flalign}
and consider the \textit{Tracy-Singh} type of product of ${\bf K}$
and $\widetilde{\boldsymbol{\Sigma}}$
\begin{equation}
\widetilde{\boldsymbol{\Sigma}}_{{\bf K}}\triangleq{\bf K}\circ\widetilde{\boldsymbol{\Sigma}}\triangleq\begin{bmatrix}{\bf K}\varolessthan\widetilde{\boldsymbol{\Sigma}}\left(1,1\right) & {\bf K}\varolessthan\widetilde{\boldsymbol{\Sigma}}\left(1,2\right) & \ldots & {\bf K}\varolessthan\widetilde{\boldsymbol{\Sigma}}\left(1,N_{T}\right)\\
{\bf K}\varolessthan\widetilde{\boldsymbol{\Sigma}}\left(2,1\right) & {\bf K}\varolessthan\widetilde{\boldsymbol{\Sigma}}\left(2,2\right) & \ldots & {\bf K}\varolessthan\widetilde{\boldsymbol{\Sigma}}\left(2,N_{T}\right)\\
\vdots & \vdots & \ddots & \vdots\\
{\bf K}\varolessthan\widetilde{\boldsymbol{\Sigma}}\left(N_{T},1\right) & {\bf K}\varolessthan\widetilde{\boldsymbol{\Sigma}}\left(N_{T},2\right) & \cdots & {\bf K}\varolessthan\widetilde{\boldsymbol{\Sigma}}\left(N_{T},N_{T}\right)
\end{bmatrix}\in\mathbb{S}^{2RN_{T}},
\end{equation}
where ``$\varolessthan$'' denotes the operator of the Kronecker
product. Then, for each $\left(k,l\right)\in\mathbb{N}_{N_{T}}^{+}\times\mathbb{N}_{N_{T}}^{+}$,
we have
\begin{equation}
{\bf K}\varolessthan\widetilde{\boldsymbol{\Sigma}}\left(k,l\right)\equiv\begin{bmatrix}\widetilde{\boldsymbol{\Sigma}}\left(k,l\right) & \kappa\widetilde{\boldsymbol{\Sigma}}\left(k,l\right)\\
\kappa\widetilde{\boldsymbol{\Sigma}}\left(k,l\right) & \widetilde{\boldsymbol{\Sigma}}\left(k,l\right)
\end{bmatrix}\in\mathbb{S}^{2R}.
\end{equation}
It is easy to show that $\widetilde{\boldsymbol{\Sigma}}_{{\bf K}}$
is positive semidefinite, that is, in $\mathbb{S}_{+}^{2RN_{T}}$.
First, via a simple inductive argument, it can be shown that, for
compatible matrices ${\bf A},{\bf B},{\bf C},{\bf D}$, 
\begin{equation}
\left({\bf A}{\bf B}\right)\circ\left({\bf C}{\bf D}\right)\equiv\left({\bf A}\circ{\bf C}\right)\left({\bf B}\circ{\bf D}\right).
\end{equation}
Also, for compatible ${\bf A},{\bf B}$, it is true that $\left({\bf A}\circ{\bf B}\right)^{\boldsymbol{T}}\equiv{\bf A}^{\boldsymbol{T}}\circ{\bf B}^{\boldsymbol{T}}$.
Since ${\bf K}$ and $\widetilde{\boldsymbol{\Sigma}}$ are symmetric,
consider their spectral decompositions ${\bf K}\equiv{\bf U}_{{\bf K}}\boldsymbol{\Lambda}_{{\bf K}}{\bf U}_{{\bf K}}^{\boldsymbol{T}}$
and $\widetilde{\boldsymbol{\Sigma}}\equiv{\bf U}_{\widetilde{\boldsymbol{\Sigma}}}\boldsymbol{\Lambda}_{\widetilde{\boldsymbol{\Sigma}}}{\bf U}_{\widetilde{\boldsymbol{\Sigma}}}^{\boldsymbol{T}}$.
Given the identities stated above, we may write
\begin{flalign}
\widetilde{\boldsymbol{\Sigma}}_{{\bf K}}\equiv{\bf K}\circ\widetilde{\boldsymbol{\Sigma}} & \equiv\left({\bf U}_{{\bf K}}\boldsymbol{\Lambda}_{{\bf K}}{\bf U}_{{\bf K}}^{\boldsymbol{T}}\right)\circ\left({\bf U}_{\widetilde{\boldsymbol{\Sigma}}}\boldsymbol{\Lambda}_{\widetilde{\boldsymbol{\Sigma}}}{\bf U}_{\widetilde{\boldsymbol{\Sigma}}}^{\boldsymbol{T}}\right)\nonumber \\
 & \equiv\left({\bf U}_{{\bf K}}\circ{\bf U}_{\widetilde{\boldsymbol{\Sigma}}}\right)\left(\boldsymbol{\Lambda}_{{\bf K}}\circ\boldsymbol{\Lambda}_{\widetilde{\boldsymbol{\Sigma}}}\right)\left({\bf U}_{{\bf K}}^{\boldsymbol{T}}\circ{\bf U}_{\widetilde{\boldsymbol{\Sigma}}}^{\boldsymbol{T}}\right)\nonumber \\
 & \equiv\left({\bf U}_{{\bf K}}\circ{\bf U}_{\widetilde{\boldsymbol{\Sigma}}}\right)\left(\boldsymbol{\Lambda}_{{\bf K}}\circ\boldsymbol{\Lambda}_{\widetilde{\boldsymbol{\Sigma}}}\right)\left({\bf U}_{{\bf K}}\circ{\bf U}_{\widetilde{\boldsymbol{\Sigma}}}\right)^{\boldsymbol{T}},\label{eq:Spectral_DEC}
\end{flalign}
where $\left({\bf U}_{{\bf K}}\circ{\bf U}_{\widetilde{\boldsymbol{\Sigma}}}\right)\left({\bf U}_{{\bf K}}^{\boldsymbol{T}}\circ{\bf U}_{\widetilde{\boldsymbol{\Sigma}}}^{\boldsymbol{T}}\right)\equiv\left({\bf U}_{{\bf K}}{\bf U}_{{\bf K}}^{\boldsymbol{T}}\right)\circ\left({\bf U}_{\widetilde{\boldsymbol{\Sigma}}}{\bf U}_{\widetilde{\boldsymbol{\Sigma}}}^{\boldsymbol{T}}\right)\equiv{\bf I}_{2}\circ{\bf I}_{RN_{T}}\equiv{\bf I}_{2RN_{T}}$,
and where the matrix $\boldsymbol{\Lambda}_{{\bf K}}\circ\boldsymbol{\Lambda}_{\widetilde{\boldsymbol{\Sigma}}}$
is easily shown to be diagonal and with nonnegative elements. Thus,
since (\ref{eq:Spectral_DEC}) constitutes a valid spectral decomposition
for $\widetilde{\boldsymbol{\Sigma}}_{{\bf K}}$, it follows that
$\widetilde{\boldsymbol{\Sigma}}_{{\bf K}}\in\mathbb{S}_{+}^{2RN_{T}}$.

As a last step, let ${\bf E}\in\mathbb{S}^{N_{T}}$, such that 
\begin{equation}
{\bf E}\left(k,l\right)\triangleq\exp\left(-\dfrac{\left|k-l\right|}{\gamma}\right),
\end{equation}
for all $\left(k,l\right)\in\mathbb{N}_{N_{T}}^{+}\times\mathbb{N}_{N_{T}}^{+}$.
Again, ${\bf E}$ is positive semidefinite, because the well known
\textit{Laplacian kernel} is positive (semi)definite. Consider the
matrix
\begin{equation}
\widetilde{\boldsymbol{\Sigma}}_{{\bf E}}\triangleq\left({\bf E}\varolessthan{\bf 1}_{2R\times2R}\right)\odot\widetilde{\boldsymbol{\Sigma}}_{{\bf K}}\in\mathbb{S}^{2RN_{T}},
\end{equation}
where ``$\odot$'' denotes the operator of the Schur-Hadamard product.
Of course, since the matrix ${\bf 1}_{2R\times2R}$ is rank-$1$ and
positive semidefinite, ${\bf E}\varolessthan{\bf 1}_{2R\times2R}$
will be positive semidefinite as well. Consequently, by the Schur
Product Theorem, $\widetilde{\boldsymbol{\Sigma}}_{{\bf E}}$ will
also be positive semidefinite. Finally, observe that
\begin{equation}
\boldsymbol{\Sigma}\equiv\widetilde{\boldsymbol{\Sigma}}_{{\bf E}}+\sigma_{\xi}^{2}{\bf I}_{2RN_{T}},
\end{equation}
from where it follows that $\boldsymbol{\Sigma}\in\mathbb{S}_{++}^{2RN_{T}}$,
whenever $\sigma_{\xi}^{2}\neq0$. Our claims follow.\hfill{}\ensuremath{\blacksquare}

\subsubsection{Proof of Theorem \ref{thm:(Markov)}}

Obviously, the vector process $\boldsymbol{X}\left(t\right)$ is Gaussian
with mean zero. This is straightforward to show. Therefore, what remains
is, simply, to verify that the covariance structure of $\boldsymbol{X}\left(t\right)$
is the same as that of $\boldsymbol{C}\left(t\right)$, that is, we
need to show that
\begin{equation}
\mathbb{E}\left\{ \boldsymbol{X}\left(s\right)\boldsymbol{X}^{\boldsymbol{T}}\left(t\right)\right\} \equiv\mathbb{E}\left\{ \boldsymbol{C}\left(s\right)\boldsymbol{C}^{\boldsymbol{T}}\left(t\right)\right\} ,
\end{equation}
for all $\left(s,t\right)\in\mathbb{N}_{N_{T}}^{+}\times\mathbb{N}_{N_{T}}^{+}$.

First, consider the case where $s\equiv t$. Then, we have
\begin{flalign}
\mathbb{E}\left\{ \boldsymbol{X}\left(s\right)\boldsymbol{X}^{\boldsymbol{T}}\left(t\right)\right\}  & \equiv\mathbb{E}\left\{ \boldsymbol{X}\left(t\right)\boldsymbol{X}^{\boldsymbol{T}}\left(t\right)\right\} \nonumber \\
 & =\varphi^{2}\mathbb{E}\left\{ \boldsymbol{X}\left(t-1\right)\boldsymbol{X}^{\boldsymbol{T}}\left(t-1\right)\right\} +\left(1-\varphi^{2}\right)\widetilde{\boldsymbol{\Sigma}}_{\boldsymbol{C}}.
\end{flalign}
Observe, though, that, similarly to the scalar order-$1$ autoregressive
model, the quantity
\begin{equation}
\widetilde{\boldsymbol{\Sigma}}_{\boldsymbol{C}}\equiv\mathbb{E}\left\{ \boldsymbol{X}\left(0\right)\boldsymbol{X}^{\boldsymbol{T}}\left(0\right)\right\} 
\end{equation}
is a fixed point of the previously stated recursion for $\mathbb{E}\left\{ \boldsymbol{X}\left(t\right)\boldsymbol{X}^{\boldsymbol{T}}\left(t\right)\right\} $.
Therefore, it is true that
\begin{equation}
\mathbb{E}\left\{ \boldsymbol{X}\left(t\right)\boldsymbol{X}^{\boldsymbol{T}}\left(t\right)\right\} \equiv\widetilde{\boldsymbol{\Sigma}}_{\boldsymbol{C}}\equiv\boldsymbol{\Sigma}_{\boldsymbol{C}}\left(0\right)\equiv\mathbb{E}\left\{ \boldsymbol{C}\left(t\right)\boldsymbol{C}^{\boldsymbol{T}}\left(t\right)\right\} ,
\end{equation}
which the desired result. 

Now, consider the case where $s<t$. Then, it may be easily shown
that
\begin{equation}
\mathbb{E}\left\{ \boldsymbol{X}\left(s\right)\boldsymbol{X}^{\boldsymbol{T}}\left(t\right)\right\} \equiv\varphi^{2}\mathbb{E}\left\{ \boldsymbol{X}\left(s-1\right)\boldsymbol{X}^{\boldsymbol{T}}\left(t-1\right)\right\} +\varphi\mathbb{E}\left\{ \boldsymbol{W}\left(s\right)\boldsymbol{X}^{\boldsymbol{T}}\left(t-1\right)\right\} .\label{eq:AUTO_2}
\end{equation}
Let us consider the second term on the RHS of (\ref{eq:AUTO_2}).
Expanding the recursion, we may write
\begin{flalign}
\varphi\mathbb{E}\left\{ \boldsymbol{W}\left(s\right)\boldsymbol{X}^{\boldsymbol{T}}\left(t-1\right)\right\}  & \equiv\varphi\mathbb{E}\left\{ \boldsymbol{W}\left(s\right)\left(\varphi\boldsymbol{X}^{\boldsymbol{T}}\left(t-2\right)+\boldsymbol{W}^{\boldsymbol{T}}\left(t-1\right)\right)\right\} \nonumber \\
 & \equiv\varphi^{2}\mathbb{E}\left\{ \boldsymbol{W}\left(s\right)\boldsymbol{X}^{\boldsymbol{T}}\left(t-2\right)\right\} \nonumber \\
 & \;\:\vdots\nonumber \\
 & \equiv\varphi^{t-s}\mathbb{E}\left\{ \boldsymbol{W}\left(s\right)\boldsymbol{W}^{\boldsymbol{T}}\left(s\right)\right\} \nonumber \\
 & \equiv\varphi^{t-s}\left(1-\varphi^{2}\right)\widetilde{\boldsymbol{\Sigma}}_{\boldsymbol{C}}.
\end{flalign}
We observe that this term depends only on the lag $t-s$. Thus, it
is true that
\begin{flalign}
\mathbb{E}\left\{ \boldsymbol{X}\left(s\right)\boldsymbol{X}^{\boldsymbol{T}}\left(t\right)\right\}  & \equiv\varphi^{2}\mathbb{E}\left\{ \boldsymbol{X}\left(s-1\right)\boldsymbol{X}^{\boldsymbol{T}}\left(t-1\right)\right\} +\varphi^{t-s}\left(1-\varphi^{2}\right)\widetilde{\boldsymbol{\Sigma}}_{\boldsymbol{C}}\nonumber \\
 & =\varphi^{2\cdot2}\mathbb{E}\left\{ \boldsymbol{X}\left(s-2\right)\boldsymbol{X}^{\boldsymbol{T}}\left(t-2\right)\right\} +\varphi^{t-s}\left(1-\varphi^{2}\right)\widetilde{\boldsymbol{\Sigma}}_{\boldsymbol{C}}\left(1+\varphi^{2}\right)\nonumber \\
 & \;\:\vdots\nonumber \\
 & \equiv\varphi^{2s}\mathbb{E}\left\{ \boldsymbol{X}\left(0\right)\boldsymbol{X}^{\boldsymbol{T}}\left(t-s\right)\right\} +\varphi^{t-s}\left(1-\varphi^{2}\right)\widetilde{\boldsymbol{\Sigma}}_{\boldsymbol{C}}\sum_{i\in\mathbb{N}_{s-1}}\left(\varphi^{2}\right)^{s-1}\nonumber \\
 & =\varphi^{2s}\mathbb{E}\left\{ \boldsymbol{X}\left(0\right)\boldsymbol{X}^{\boldsymbol{T}}\left(t-s\right)\right\} +\varphi^{t-s}\left(1-\varphi^{2s}\right)\widetilde{\boldsymbol{\Sigma}}_{\boldsymbol{C}}.
\end{flalign}
Further, we may expand $\mathbb{E}\left\{ \boldsymbol{X}\left(0\right)\boldsymbol{X}^{\boldsymbol{T}}\left(t-s\right)\right\} $
in similar fashion as above, to get that
\begin{equation}
\mathbb{E}\left\{ \boldsymbol{X}\left(0\right)\boldsymbol{X}^{\boldsymbol{T}}\left(t-s\right)\right\} \equiv\varphi^{t-s}\widetilde{\boldsymbol{\Sigma}}_{\boldsymbol{C}}.
\end{equation}
Exactly the same arguments may be made for the symmetric case where
$t<s$. Therefore, it follows that
\begin{flalign}
\mathbb{E}\left\{ \boldsymbol{X}\left(s\right)\boldsymbol{X}^{\boldsymbol{T}}\left(t\right)\right\}  & \equiv\varphi^{\left|t-s\right|}\widetilde{\boldsymbol{\Sigma}}_{\boldsymbol{C}}\nonumber \\
 & \equiv\exp\left(-\dfrac{\left|t-s\right|}{\gamma}\right)\widetilde{\boldsymbol{\Sigma}}_{\boldsymbol{C}}\nonumber \\
 & \equiv\boldsymbol{\Sigma}_{\boldsymbol{C}}\left(t-s\right)
\end{flalign}
for all $\left(s,t\right)\in\mathbb{N}_{N_{T}}^{+}\times\mathbb{N}_{N_{T}}^{+}$,
and we are done.\hfill{}\ensuremath{\blacksquare}

\subsection{\label{subsec:Appendix-B:-Measurability}Appendix B: Measurability
\& The Fundamental Lemma of Stochastic Control}


In the following, aligned with the purposes of this paper, a detailed
discussion is presented, which is related to important technical issues,
arising towards the analysis and simplification of variational problems
of the form of (\ref{eq:2STAGE-1}).

At this point, it would be necessary to introduce some important concepts.
Let us first introduce the useful class of \textit{Carath\'eodory
functions} \cite{Aliprantis2006_Inf,Shapiro2009STOCH_PROG}\footnote{Instead of working with the class of Carath\'eodory functions, we
could also consider the more general class of \textit{random lower
semicontinuous functions} \cite{Shapiro2009STOCH_PROG}, which includes
the former. However, this might lead to overgeneralization and, thus,
we prefer not to do so; the class of Carath\'eodory functions will
be perfectly sufficient for our purposes.}.
\begin{defn}
\textbf{(Carath\'eodory Function \cite{Aliprantis2006_Inf,Shapiro2009STOCH_PROG})}\label{def:Caratheodory}
On $\left(\Omega,\mathscr{F}\right)$, the mapping $H:\Omega\times\mathbb{R}^{N}\rightarrow\overline{\mathbb{R}}$
is called Carath\'eodory, if and only if $H\left(\cdot,\boldsymbol{x}\right)$
is $\mathscr{F}$-measurable for all $\boldsymbol{x}\in\mathbb{R}^{N}$
and $H\left(\omega,\cdot\right)$ is continuous for all $\omega\in\Omega$.
\end{defn}

\begin{rem}
As the reader might have already observed, Carath\'eodory functions
and random fields with (everywhere) continuous sample paths are essentially
the same thing. Nevertheless, the term ``Carath\'eodory function''
is extensively used in our references \cite{Aliprantis2006_Inf,Shapiro2009STOCH_PROG,Rockafellar2009VarAn}.
This is the main reason why we still define and use the term.\hfill{}\ensuremath{\blacksquare}
\end{rem}
In the analysis that follows, we will exploit the notion of \textit{measurability}
for closed-valued multifunctions.
\begin{defn}
\textbf{(Measurable Multifunctions }\cite{Rockafellar2009VarAn,Shapiro2009STOCH_PROG}\textbf{)}\label{def:MEAS_mult}
On the measurable space $\left(\Omega,\mathscr{F}\right)$, a closed-valued
multifunction ${\cal X}:\Omega\rightrightarrows\mathbb{R}^{N}$ is
\textit{$\mathscr{F}$-}measurable if and only if, for all closed
${\cal A}\subseteq\mathbb{R}^{N}$, the preimage
\begin{equation}
{\cal X}^{-1}\left({\cal A}\right)\triangleq\left\{ \omega\in\Omega\left|{\cal X}\left(\omega\right)\bigcap{\cal A}\neq\varnothing\right.\right\} 
\end{equation}
is in $\mathscr{F}$. If $\mathscr{F}$ constitutes a Borel $\sigma$-algebra,
generated by a topology on $\Omega$, then an $\mathscr{F}$-measurable
${\cal X}$ will be equivalently called \textit{Borel measurable}.
\end{defn}
We will also make use of the concept of a \textit{closed multifunction}
(Remark 28 in \cite{Shapiro2009STOCH_PROG}, p. 365), whose definition
is also presented below, restricted to the case of Euclidean spaces,
of interest in this work. 
\begin{defn}
\textbf{(Closed Multifunction \cite{Shapiro2009STOCH_PROG})}\label{def:CLODES_mult}
A closed-valued multifunction ${\cal X}:\mathbb{R}^{M}\rightrightarrows\mathbb{R}^{N}$
(a function from $\mathbb{R}^{M}$ to closed sets in $\mathbb{R}^{N}$)
is closed if and only if, for all sequences $\left\{ \boldsymbol{x}_{k}\right\} _{k\in\mathbb{N}}$
and $\left\{ \boldsymbol{y}_{k}\right\} _{k\in\mathbb{N}}$, such
that $\boldsymbol{x}_{k}\underset{k\rightarrow\infty}{\longrightarrow}\boldsymbol{x}$,
$\boldsymbol{y}_{k}\underset{k\rightarrow\infty}{\longrightarrow}\boldsymbol{y}$
and $\boldsymbol{x}_{k}\in{\cal X}\left(\boldsymbol{y}_{k}\right)$,
for all $k\in\mathbb{N}$, it is true that $\boldsymbol{x}\in{\cal X}\left(\boldsymbol{y}\right)$.
\end{defn}

\subsubsection{\label{subsec:SP}Random Functions \& The Substitution Rule for Conditional
Expectations}

Given a random function $g\left(\omega,\boldsymbol{x}\right)$, a
sub $\sigma$-algebra $\mathscr{Y}$, another $\mathscr{Y}$-measurable
random element $X$, and as long as $\mathbb{E}\left\{ \left.g\left(\cdot,\boldsymbol{x}\right)\right|\mathscr{Y}\right\} $
exists for all $\boldsymbol{x}$ in the range of $X$, we would also
need to make extensive use of the \textit{substitution rule}
\begin{flalign}
\mathbb{E}\left\{ \left.g\left(\cdot,X\right)\right|\mathscr{Y}\right\} \left(\omega\right) & \equiv\mathbb{E}\left\{ \left.g\left(\cdot,X\left(\omega\right)\right)\right|\mathscr{Y}\right\} \left(\omega\right)\nonumber \\
 & \equiv\hspace{-1.2pt}\left.\mathbb{E}\left\{ \left.g\left(\cdot,\boldsymbol{x}\right)\right|\mathscr{Y}\right\} \left(\omega\right)\right|_{\boldsymbol{x}\equiv X\left(\omega\right)},\quad{\cal P}-a.e.,
\end{flalign}
which would allow us to evaluate conditional expectations, by essentially
fixing the quantities that are constant relative to the information
we are conditioning on, carry out the evaluation, and then let those
quantities vary in $\omega$ again. Although the substitution rule
is a concept readily taken for granted when conditional expectations
of Borel measurable functions of random elements (say, from products
of Euclidean spaces to $\mathbb{R}$) are considered, it does not
hold, in general, for arbitrary random functions. As far as our general
formulation is concerned, it is necessary to consider random functions,
whose domain is a product of a well behaved space (such as $\mathbb{R}^{N}$)
and the sample space, $\Omega$, whose structure is assumed to be
and should be arbitrary, at least in regard to the applications of
interest in this work. 

One common way to ascertain the validity of the substitution rule
is by exploiting the representation of conditional expectations via
integrals with respect to the relevant regular conditional distributions,
whenever the latter exist. But because of the arbitrary structure
of the base space $\left(\Omega,\mathscr{F},{\cal P}\right)$, regular
conditional distributions defined on points in the sample space $\Omega$
cannot be guaranteed to exist and, therefore, the substitution rule
may fail to hold. However, as we will see, the substitution rule will
be very important for establishing the Fundamental Lemma. Therefore,
we may choose to impose it as a property on the structures of $g$
and/or $X$ instead, as well as establish sufficient conditions for
this property to hold. The relevant definition follows.
\begin{defn}
\textbf{(Substitution Property ($\boldsymbol{SP}$))}\label{def:SUB}
On $\left(\Omega,\mathscr{F},{\cal P}\right)$, consider a random
element $Y:\Omega\rightarrow\mathbb{R}^{M}$, the associated sub $\sigma$-algebra
$\mathscr{Y}\triangleq\sigma\left\{ Y\right\} \subseteq\mathscr{F}$,
and a random function $g:\Omega\times\mathbb{R}^{N}\rightarrow\mathbb{R}$,
such that $\mathbb{E}\left\{ g\left(\cdot,\boldsymbol{x}\right)\right\} $
exists for all $\boldsymbol{x}\in\mathbb{R}^{N}$. Let \textit{$\mathfrak{C}_{\mathscr{Y}}$
}be any functional class, such that\footnote{Hereafter, statements of type ``$\mathbb{E}\left\{ g\left(\cdot,X\right)\right\} \text{ exists}$''
will \textit{implicitly }imply that $g\left(\cdot,X\right)$ is an
$\mathscr{F}$-measurable function. }
\begin{equation}
\mathfrak{C}_{\mathscr{Y}}\subseteq\mathfrak{I}_{\mathscr{Y}}\triangleq\left\{ X:\Omega\rightarrow\mathbb{R}^{N}\hspace{-2pt}\left|\hspace{-2pt}\hspace{-2pt}\begin{array}{c}
X^{-1}\left({\cal A}\right)\in\mathscr{Y},\text{ for all }{\cal A}\in\mathscr{B}\left(\mathbb{R}^{N}\right)\\
\mathbb{E}\left\{ g\left(\cdot,X\right)\right\} \text{ exists}
\end{array}\right.\hspace{-2pt}\hspace{-2pt}\hspace{-2pt}\right\} .
\end{equation}
We say that \textit{$g$ possesses the Substitution Property within
$\mathfrak{C}_{\mathscr{Y}}$}, or, equivalently, that\textit{ $g$
is $\boldsymbol{SP}\diamondsuit\mathfrak{C}_{\mathscr{Y}}$}, if and
only if there exists a jointly Borel measurable function $h:\mathbb{R}^{M}\times\mathbb{R}^{N}\rightarrow\overline{\mathbb{R}}$,
with $h\left(Y\left(\omega\right),\boldsymbol{x}\right)\equiv\mathbb{E}\left\{ \left.g\left(\cdot,\boldsymbol{x}\right)\right|\mathscr{Y}\right\} \left(\omega\right)$,
everywhere in $\left(\omega,\boldsymbol{x}\right)\in\Omega\times\mathbb{R}^{N}$,
such that, for any $X\in\mathfrak{C}_{\mathscr{Y}}$, it is true that
\begin{equation}
\mathbb{E}\left\{ \left.g\left(\cdot,X\right)\right|\mathscr{Y}\right\} \left(\omega\right)\equiv h\left(Y\left(\omega\right),X\left(\omega\right)\right),\label{eq:SP_Core}
\end{equation}
almost everywhere in $\omega\in\Omega$ with respect to ${\cal P}$.
\end{defn}
\begin{rem}
Observe that, in Definition \ref{def:SUB}, $h$ is required to be
the same for all $X\in\mathfrak{C}_{\mathscr{Y}}$. That is, $h$
should be determined only by the structure of $g$, relative to $\mathscr{Y}$,
regardless of the specific $X$ within $\mathfrak{C}_{\mathscr{Y}}$,
considered each time. On the other hand, it is also important to note
that the set of unity measure, where (\ref{eq:SP_Core}) is valid,
\textit{might indeed be dependent on the particular} $X$.\hfill{}\ensuremath{\blacksquare}
\end{rem}
\begin{rem}
\label{rem:DETAIL_h}Another detail of Definition \ref{def:SUB} is
that, because $\mathbb{E}\left\{ g\left(\cdot,\boldsymbol{x}\right)\right\} $
is assumed to exist for all $\boldsymbol{x}\in\mathbb{R}^{N}$, $\mathbb{E}\left\{ \left.g\left(\cdot,\boldsymbol{x}\right)\right|\mathscr{Y}\right\} $
also exists and, as an extended $\mathscr{Y}$-measurable random variable,
for every $\boldsymbol{x}\in\mathbb{R}^{N}$, there exists a Borel
measurable function $h_{\boldsymbol{x}}:\mathbb{R}^{M}\rightarrow\overline{\mathbb{R}}$,
such that
\begin{equation}
h_{\boldsymbol{x}}\left(Y\left(\omega\right)\right)\equiv\mathbb{E}\left\{ \left.g\left(\cdot,\boldsymbol{x}\right)\right|\mathscr{Y}\right\} \left(\omega\right),\quad\forall\omega\in\Omega.
\end{equation}
One may then readily define a function $h:\mathbb{R}^{M}\times\mathbb{R}^{N}\rightarrow\overline{\mathbb{R}}$,
such that $h\left(Y\left(\omega\right)\hspace{-2pt},\boldsymbol{x}\right)\hspace{-2pt}\equiv\hspace{-2pt}\mathbb{E}\left\{ \left.g\left(\cdot,\boldsymbol{x}\right)\right|\mathscr{Y}\right\} \left(\omega\right)$,
uniformly \textbf{\textit{for all points}}\textit{, $\omega$, }\textbf{\textit{of
the sample space}}\textit{, $\Omega$}. This is an extremely important
fact, in regard to the analysis that follows. Observe, however, that,
in general, $h$ will be Borel measurable \textit{only} \textit{in
its first argument}; $h$ is not guaranteed to be measurable in $\boldsymbol{x}\in\mathbb{R}^{N}$,
for each $Y\in\mathbb{R}^{M}$, let alone jointly measurable in both
its arguments.\hfill{}\ensuremath{\blacksquare}
\end{rem}
\begin{rem}
\label{rem:FULLSP_1}\textbf{(Generalized $\boldsymbol{SP}$)} Definition
\ref{def:SUB} may be reformulated in a more general setting. In particular,
$\mathscr{Y}$ may be assumed to be any arbitrary sub $\sigma$-algebra
of $\mathscr{F}$, but with the subtle difference that, in such case,
one would instead directly demand that the random function $h:\Omega\times\mathbb{R}^{N}\rightarrow\overline{\mathbb{R}}$,
with $h\left(\omega,\boldsymbol{x}\right)\equiv\mathbb{E}\left\{ \left.g\left(\cdot,\boldsymbol{x}\right)\right|\mathscr{Y}\right\} \left(\omega\right)$,
everywhere in $\left(\omega,\boldsymbol{x}\right)\in\Omega\times\mathbb{R}^{N}$,
is jointly $\mathscr{Y}\otimes\mathscr{B}\left(\mathbb{R}^{N}\right)$-measurable
and such that, for any $X\in\mathfrak{C}_{\mathscr{Y}}$ (with $\mathfrak{C}_{\mathscr{Y}}$
defined accordingly), it is true that
\begin{equation}
\mathbb{E}\left\{ \left.g\left(\cdot,X\right)\right|\mathscr{Y}\right\} \left(\omega\right)\equiv h\left(\omega,X\left(\omega\right)\right),\quad{\cal P}-a.e..\label{eq:SP_Core-1}
\end{equation}
Although such a generalized definition of the substitution property
is certainly less enlightening, it is still useful. Specifically,
this version of \textbf{$\boldsymbol{SP}$} is explicitly used in
the statement and proof of Theorem \ref{lem:QoS_INCREASES}, presented
in Section \ref{subsec:Theoretical-Guarantees}.\hfill{}\ensuremath{\blacksquare}
\end{rem}
Keeping $\left(\Omega,\mathscr{F},{\cal P}\right)$ of arbitrary structure,
we will be interested in the set of $g$'s which are $\boldsymbol{SP}\diamondsuit\mathfrak{I}_{\mathscr{Y}}$.
The next result provides a large class of such random functions, which
is sufficient for our purposes.
\begin{thm}
\textbf{\textup{(Sufficient Conditions for the }}$\boldsymbol{SP}\diamondsuit\mathfrak{I}_{\mathscr{Y}}$\textbf{\textup{)\label{thm:REP_EXP}}}
On $\left(\Omega,\mathscr{F},{\cal P}\right)$, consider a random
element $Y:\Omega\rightarrow\mathbb{R}^{M}$, the associated sub $\sigma$-algebra
$\mathscr{Y}\triangleq\sigma\left\{ Y\right\} \subseteq\mathscr{F}$,
and a random function $g:\Omega\times\mathbb{R}^{N}\rightarrow\mathbb{R}$.
Suppose that:
\begin{itemize}
\item g is dominated by a ${\cal P}$-integrable function; that is,
\begin{equation}
\exists\psi\in{\cal L}_{1}\left(\Omega,\mathscr{F},{\cal P};\mathbb{R}\right),\text{ such that }\sup_{\boldsymbol{x}\in\mathbb{R}^{N}}\left|g\left(\omega,\boldsymbol{x}\right)\right|\le\psi\left(\omega\right),\quad\forall\omega\in\Omega,
\end{equation}
\item $g$ is Carath\'eodory on $\Omega\times\mathbb{R}^{N}$, and that
\item the extended real valued function $\mathbb{E}\left\{ \left.g\left(\cdot,\boldsymbol{x}\right)\right|\mathscr{Y}\right\} $
is Carath\'eodory on $\Omega\times\mathbb{R}^{N}$.
\end{itemize}
Then, $g$ is $\boldsymbol{SP}\diamondsuit\mathfrak{I}_{\mathscr{Y}}$.
\end{thm}
\begin{proof}[Proof of Theorem \ref{thm:REP_EXP}]
Under the setting of the theorem, consider any $\mathscr{Y}$-measurable
random element $X:\Omega\rightarrow\mathbb{R}^{N}$, for which $\mathbb{E}\left\{ g\left(\cdot,X\right)\right\} $
exists. Then, $\mathbb{E}\left\{ \left.g\left(\cdot,X\right)\right|\mathscr{Y}\right\} $
exists. Also, by domination of $g$ by $\psi$, for all $\boldsymbol{x}\in\mathbb{R}^{N}$,
$\mathbb{E}\left\{ \left.g\left(\cdot,\boldsymbol{x}\right)\right|\mathscr{Y}\right\} $
exists and constitutes a ${\cal P}$-integrable, $\mathscr{Y}$-measurable
random variable. By Remark \ref{rem:DETAIL_h}, we know that
\begin{equation}
\mathbb{E}\left\{ \left.g\left(\cdot,\boldsymbol{x}\right)\right|\mathscr{Y}\right\} \left(\omega\right)\equiv h\left(Y\left(\omega\right),\boldsymbol{x}\right),\quad\forall\left(\omega,\boldsymbol{x}\right)\in\Omega\times\mathbb{R}^{N},
\end{equation}
where $h:\mathbb{R}^{M}\times\mathbb{R}^{N}\rightarrow\overline{\mathbb{R}}$
is Borel measurable \textit{in its first argument}. However, since
$\mathbb{E}\left\{ \left.g\left(\cdot,\boldsymbol{x}\right)\right|\mathscr{Y}\right\} \left(\omega\right)$\linebreak{}
$\equiv h\left(Y\left(\omega\right),\boldsymbol{x}\right)$ is Carath\'eodory
on $\Omega\times\mathbb{R}^{N}$, $h$ is Carath\'eodory on $\mathbb{R}^{M}\times\mathbb{R}^{N}$,
as well. Thus, $h$ will be jointly $\mathscr{B}\left(\mathbb{R}^{M}\right)\otimes\mathscr{B}\left(\mathbb{R}^{N}\right)$-measurable
(Lemma 4.51 in \cite{Aliprantis2006_Inf}, along with the fact that
$\overline{\mathbb{R}}$ is metrizable).

We claim that, actually, $h$ is such that 
\begin{equation}
\mathbb{E}\left\{ \left.g\left(\cdot,X\right)\right|\mathscr{Y}\right\} \equiv h\left(Y,X\right),\quad{\cal P}-a.e..
\end{equation}
Employing a common technique, the result will be proven in steps,
starting from indicators and building up to arbitrary measurable functions,
as far as $X$ is concerned. Before embarking with the core of the
proof, note that, for any $\boldsymbol{x}_{1}$ and $\boldsymbol{x}_{2}$
in $\mathbb{R}^{N}$ and any ${\cal A}\in\mathscr{F}$, the sum $g\left(\cdot,\boldsymbol{x}_{1}\right)\mathds{1}_{{\cal A}}+g\left(\cdot,\boldsymbol{x}_{2}\right)\mathds{1}_{{\cal A}^{c}}$
is always well defined, and $\mathbb{E}\left\{ g\left(\cdot,\boldsymbol{x}_{1}\right)\mathds{1}_{{\cal A}}\right\} $
and $\mathbb{E}\left\{ g\left(\cdot,\boldsymbol{x}_{2}\right)\mathds{1}_{{\cal A}^{c}}\right\} $
both exist and are finite by domination. This implies that $\mathbb{E}\left\{ g\left(\cdot,\boldsymbol{x}_{1}\right)\mathds{1}_{{\cal A}}\right\} +\mathbb{E}\left\{ g\left(\cdot,\boldsymbol{x}_{2}\right)\mathds{1}_{{\cal A}^{c}}\right\} $
is always well-defined, which in turn implies the validity of the
additivity properties (Theorem 1.6.3 and Theorem 5.5.2 in \cite{Ash2000Probability})
\begin{flalign}
\mathbb{E}\left\{ g\left(\cdot,\boldsymbol{x}_{1}\right)\mathds{1}_{{\cal A}}+g\left(\cdot,\boldsymbol{x}_{2}\right)\mathds{1}_{{\cal A}^{c}}\right\}  & \equiv\mathbb{E}\left\{ g\left(\cdot,\boldsymbol{x}_{1}\right)\mathds{1}_{{\cal A}}\right\} +\mathbb{E}\left\{ g\left(\cdot,\boldsymbol{x}_{2}\right)\mathds{1}_{{\cal A}^{c}}\right\} \in\mathbb{R},\quad\text{and}\\
\mathbb{E}\left\{ \left.g\left(\cdot,\boldsymbol{x}_{1}\right)\mathds{1}_{{\cal A}}+g\left(\cdot,\boldsymbol{x}_{2}\right)\mathds{1}_{{\cal A}^{c}}\right|\mathscr{Y}\right\}  & \equiv\mathbb{E}\left\{ \left.g\left(\cdot,\boldsymbol{x}_{1}\right)\mathds{1}_{{\cal A}}\right|\mathscr{Y}\right\} +\mathbb{E}\left\{ \left.g\left(\cdot,\boldsymbol{x}_{2}\right)\mathds{1}_{{\cal A}^{c}}\right|\mathscr{Y}\right\} ,{\cal P}-a.e..
\end{flalign}
Hence, under our setting, any such manipulation is technically justified.

Suppose first that $X\left(\omega\right)\equiv\widetilde{\boldsymbol{x}}\mathds{1}_{{\cal A}}\left(\omega\right)$,
for some $\widetilde{\boldsymbol{x}}\in\mathbb{R}^{N}$ and some ${\cal A}\in\mathscr{Y}$.
Then, by (\cite{Ash2000Probability}, Theorem 5.5.11 \& Comment 5.5.12),
it is true that
\begin{flalign}
\mathbb{E}\left\{ \left.g\left(\cdot,X\right)\right|\mathscr{Y}\right\}  & \equiv\mathbb{E}\left\{ \left.g\left(\cdot,\widetilde{\boldsymbol{x}}\right)\mathds{1}_{{\cal A}}+g\left(\cdot,{\bf 0}\right)\mathds{1}_{{\cal A}^{c}}\right|\mathscr{Y}\right\} \nonumber \\
 & \equiv\mathbb{E}\left\{ \left.g\left(\cdot,\widetilde{\boldsymbol{x}}\right)\mathds{1}_{{\cal A}}\right|\mathscr{Y}\right\} +\mathbb{E}\left\{ \left.g\left(\cdot,{\bf 0}\right)\mathds{1}_{{\cal A}^{c}}\right|\mathscr{Y}\right\} \nonumber \\
 & \equiv\mathds{1}_{{\cal A}}\mathbb{E}\left\{ \left.g\left(\cdot,\widetilde{\boldsymbol{x}}\right)\right|\mathscr{Y}\right\} +\mathds{1}_{{\cal A}^{c}}\mathbb{E}\left\{ \left.g\left(\cdot,{\bf 0}\right)\right|\mathscr{Y}\right\} \nonumber \\
 & \equiv\mathds{1}_{{\cal A}}h\left(Y,\widetilde{\boldsymbol{x}}\right)+\mathds{1}_{{\cal A}^{c}}h\left(Y,{\bf 0}\right)\nonumber \\
 & \equiv h\left(Y,\widetilde{\boldsymbol{x}}\mathds{1}_{{\cal A}}\right)\nonumber \\
 & \equiv h\left(Y,X\right),\quad{\cal P}-a.e.,
\end{flalign}
proving the claim for indicators.

Consider now simple functions of the form
\begin{equation}
X\left(\omega\right)\equiv\sum_{i\in\mathbb{N}_{I}^{+}}\widetilde{\boldsymbol{x}}_{i}\mathds{1}_{{\cal A}_{i}}\left(\omega\right),\label{eq:Simple}
\end{equation}
where $\widetilde{\boldsymbol{x}}_{i}\in\mathbb{R}^{N},$ ${\cal A}_{i}\in\mathscr{Y}$,
for all $i\in\mathbb{N}_{I}^{+}$, with ${\cal A}_{i}\bigcap{\cal A}_{j}\equiv\varnothing$,
for $i\neq j$ and $\bigcup_{i\in\mathbb{N}_{I}^{+}}{\cal A}_{i}\equiv\Omega$.
Then, we again have
\begin{flalign}
\mathbb{E}\left\{ \left.g\left(\cdot,X\right)\right|\mathscr{Y}\right\}  & \equiv\mathbb{E}\left\{ \left.\sum_{i\in\mathbb{N}_{I}^{+}}g\left(\cdot,\widetilde{\boldsymbol{x}}_{i}\right)\mathds{1}_{{\cal A}_{i}}\right|\mathscr{Y}\right\} \nonumber \\
 & \equiv\sum_{i\in\mathbb{N}_{I}^{+}}\mathbb{E}\left\{ \left.g\left(\cdot,\widetilde{\boldsymbol{x}}_{i}\right)\mathds{1}_{{\cal A}_{i}}\right|\mathscr{Y}\right\} \nonumber \\
 & \equiv\sum_{i\in\mathbb{N}_{I}^{+}}\mathds{1}_{{\cal A}_{i}}\mathbb{E}\left\{ \left.g\left(\cdot,\widetilde{\boldsymbol{x}}_{i}\right)\right|\mathscr{Y}\right\} \nonumber \\
 & \equiv\sum_{i\in\mathbb{N}_{I}^{+}}\mathds{1}_{{\cal A}_{i}}h\left(Y,\widetilde{\boldsymbol{x}}_{i}\right)\nonumber \\
 & \equiv h\left(Y,\sum_{i\in\mathbb{N}_{I}^{+}}\widetilde{\boldsymbol{x}}_{i}\mathds{1}_{{\cal A}_{i}}\right)\nonumber \\
 & \equiv h\left(Y,X\right),\quad{\cal P}-a.e.,
\end{flalign}
and the proved is claimed for simple functions.

To show that our claims are true for any arbitrary random function
$g$, we take advantage of the continuity of both $h$ and $g$ in
$\boldsymbol{x}$. First, we know that $h$ is Carath\'eodory, which
means that, for every $\omega\in\Omega$, if any sequence $\left\{ \boldsymbol{x}_{n}\in\mathbb{R}^{N}\right\} _{n\in\mathbb{N}}$
is such that $\boldsymbol{x}_{n}\underset{n\rightarrow\infty}{\longrightarrow}\boldsymbol{x}$
(for arbitrary $\boldsymbol{x}\in\mathbb{R}^{N}$), it is true that

\begin{equation}
h\left(Y\left(\omega\right),\boldsymbol{x}_{n}\right)\equiv\mathbb{E}\left\{ \left.g\left(\cdot,\boldsymbol{x}_{n}\right)\right|\mathscr{Y}\right\} \left(\omega\right)\underset{n\rightarrow\infty}{\longrightarrow}\mathbb{E}\left\{ \left.g\left(\cdot,\boldsymbol{x}\right)\right|\mathscr{Y}\right\} \left(\omega\right)\equiv h\left(Y\left(\omega\right),\boldsymbol{x}\right).\label{eq:hhhhhh}
\end{equation}
Second, we know that $g$ is Carath\'eodory as well, also implying
that, for every $\omega\in\Omega$, if any sequence $\left\{ \boldsymbol{x}_{n}\in\mathbb{R}^{N}\right\} _{n\in\mathbb{N}}$
is such that $\boldsymbol{x}_{n}\underset{n\rightarrow\infty}{\longrightarrow}\boldsymbol{x}$,
it is true that

\begin{equation}
g\left(\omega,\boldsymbol{x}_{n}\right)\underset{n\rightarrow\infty}{\longrightarrow}g\left(\omega,\boldsymbol{x}\right).
\end{equation}
Next, let $\left\{ X_{n}:\Omega\rightarrow\mathbb{R}^{N}\right\} _{n\in\mathbb{N}}$
be a sequence of simple Borel functions, such that, for all $\omega\in\Omega$,
\begin{equation}
X_{n}\left(\omega\right)\underset{n\rightarrow\infty}{\longrightarrow}X\left(\omega\right).
\end{equation}
Note that such a sequence always exists (see Theorem 1.5.5 (b) in
\cite{Ash2000Probability}). Consequently, for each $\omega\in\Omega$,
we may write (note that $g$ is $\mathscr{F}\otimes\mathscr{B}\left(\mathbb{R}^{N}\right)$-measurable;
see (\cite{Aliprantis2006_Inf}, Lemma 4.51))
\begin{equation}
g\left(\omega,X_{n}\left(\omega\right)\right)\underset{n\rightarrow\infty}{\longrightarrow}g\left(\omega,X\left(\omega\right)\right),
\end{equation}
that is, the sequence $\left\{ g\left(\cdot,X_{n}\right)\right\} _{n\in\mathbb{N}}$
converges to $g\left(\cdot,X\right)$, everywhere in $\Omega$.

Now, let us try to apply the Dominated Convergence Theorem for conditional
expectations (Theorem 5.5.5 in \cite{Ash2000Probability}) to the
aforementioned sequence of functions. Of course, we have to show that
all members of the sequence $\left\{ g\left(\cdot,X_{n}\right)\right\} _{n\in\mathbb{N}}$
are dominated by another integrable function, uniformly in $n\in\mathbb{N}$.
By assumption, there exists an integrable function $\psi:\Omega\rightarrow\mathbb{R}$,
such that
\begin{equation}
\left|g\left(\omega,\boldsymbol{x}\right)\right|\le\psi\left(\omega\right),\quad\forall\left(\omega,\boldsymbol{x}\right)\in\Omega\times\mathbb{R}^{N}.
\end{equation}
In particular, it must also be true that
\begin{equation}
\left|g\left(\omega,X_{n}\left(\omega\right)\right)\right|\le\psi\left(\omega\right),\quad\forall\left(\omega,n\right)\in\Omega\times\mathbb{N},
\end{equation}
verifying the domination requirement. Thus, Dominated Convergence
implies the existence of an event $\Omega_{\Pi_{1}}\subseteq\Omega$,
with ${\cal P}\left(\Omega_{\Pi_{1}}\right)\equiv1$, such that, for
all $\omega\in\Omega_{\Pi_{1}}$,
\begin{equation}
\mathbb{E}\left\{ \left.g\left(\cdot,X_{n}\right)\right|\mathscr{Y}\right\} \left(\omega\right)\underset{n\rightarrow\infty}{\longrightarrow}\mathbb{E}\left\{ \left.g\left(\cdot,X\right)\right|\mathscr{Y}\right\} \left(\omega\right).\label{eq:PROP_1}
\end{equation}
Also, for every $\omega\in\Omega\bigcap\Omega_{\Pi_{1}}\equiv\Omega_{\Pi_{1}}$,
(\ref{eq:hhhhhh}) yields
\begin{equation}
h\left(Y\left(\omega\right),X_{n}\left(\omega\right)\right)\underset{n\rightarrow\infty}{\longrightarrow}h\left(Y\left(\omega\right),X\left(\omega\right)\right).\label{eq:PROP_2}
\end{equation}
However, by what we have shown above, because the sequence $\left\{ X_{n}\right\} _{n\in\mathbb{N}}$
consists of simple functions, then, for every $n\in\mathbb{N}$, there
exists $\Omega_{\Pi^{n}}\subseteq\Omega$, with ${\cal P}\left(\Omega_{\Pi^{n}}\right)\equiv1$,
such that, for all $\omega\in\Omega_{\Pi^{n}}$,
\begin{equation}
\mathbb{E}\left\{ \left.g\left(\cdot,X_{n}\right)\right|\mathscr{Y}\right\} \left(\omega\right)\equiv h\left(Y\left(\omega\right),X_{n}\left(\omega\right)\right).
\end{equation}
Since $\mathbb{N}$ is countable, there exists a ``global'' event
$\Omega_{\Pi_{2}}\subseteq\Omega$, with ${\cal P}\left(\Omega_{\Pi_{2}}\right)\equiv1$,
such that, for all $\omega\in\Omega_{\Pi_{2}}$,
\begin{equation}
\mathbb{E}\left\{ \left.g\left(\cdot,X_{n}\right)\right|\mathscr{Y}\right\} \left(\omega\right)\equiv h\left(Y\left(\omega\right),X_{n}\left(\omega\right)\right),\quad\forall n\in\mathbb{N}.\label{eq:PROP_3}
\end{equation}
Now define the event $\Omega_{\Pi_{3}}\triangleq\Omega_{\Pi_{1}}\bigcap\Omega_{\Pi_{2}}$.
Of course, ${\cal P}\left(\Omega_{\Pi_{3}}\right)\equiv1$. Then,
for every $\omega\in\Omega_{\Pi_{3}}$, (\ref{eq:PROP_1}), (\ref{eq:PROP_2})
and (\ref{eq:PROP_3}) all hold simultaneously. Therefore, for every
$\omega\in\Omega_{\Pi_{3}}$, it is true that (say)
\begin{flalign}
h\left(Y\left(\omega\right),X_{n}\left(\omega\right)\right) & \underset{n\rightarrow\infty}{\longrightarrow}\mathbb{E}\left\{ \left.g\left(\cdot,X\right)\right|\mathscr{Y}\right\} \left(\omega\right)\quad\text{and}\\
h\left(Y\left(\omega\right),X_{n}\left(\omega\right)\right) & \underset{n\rightarrow\infty}{\longrightarrow}h\left(Y\left(\omega\right),X\left(\omega\right)\right),
\end{flalign}
which immediately yields
\begin{equation}
\mathbb{E}\left\{ \left.g\left(\cdot,X\right)\right|\mathscr{Y}\right\} \left(\omega\right)\equiv h\left(Y\left(\omega\right),X\left(\omega\right)\right),\quad{\cal P}-a.e.,
\end{equation}
showing that $g$ is \textit{$\boldsymbol{SP}\diamondsuit\mathfrak{I}_{\mathscr{Y}}$.}
\end{proof}
\begin{rem}
We would like to note that the assumptions of Theorem \ref{thm:REP_EXP}
can be significantly weakened, guaranteeing the validity of the substitution
rule for vastly discontinuous random functions, including, for instance,
cases with random discontinuities, or random jumps. This extended
analysis, though, is out of the scope of the paper and will be presented
elsewhere.\hfill{}\ensuremath{\blacksquare}
\end{rem}

\subsubsection{\label{subsec:Base-Form}A Base Form of the Lemma}

We will first state a base, very versatile version of the Fundamental
Lemma, treating a general class of problems, which includes the particular
stochastic problem of interest, (\ref{eq:2STAGE-1}), as a subcase.
\begin{lem}
\textbf{\textup{(Fundamental Lemma / Base Version)\label{lem:FUND_Lemma}}}
On $\left(\Omega,\mathscr{F},{\cal P}\right)$, consider a random
element $Y:\Omega\rightarrow\mathbb{R}^{M}$, the sub $\sigma$-algebra
$\mathscr{Y}\triangleq\sigma\left\{ Y\right\} \subseteq\mathscr{F}$,
a random function $g:\Omega\times\mathbb{R}^{N}\rightarrow\mathbb{R}$,
such that $\mathbb{E}\left\{ g\left(\cdot,\boldsymbol{x}\right)\right\} $
exists for all $\boldsymbol{x}\in\mathbb{R}^{N}$, a Borel measurable
closed-valued multifunction ${\cal X}:\mathbb{R}^{N}\rightrightarrows\mathbb{R}^{N}$,
with $\mathrm{dom}\left({\cal X}\right)\equiv\mathbb{R}^{N}$, as
well as another $\mathscr{Y}$-measurable random element $Z_{Y}:\Omega\rightarrow\mathbb{R}^{N}$,
with $Z_{Y}\left(\omega\right)\equiv{\cal Z}\left(Y\left(\omega\right)\right)$,
for all $\omega\in\Omega$, for some Borel ${\cal Z}:\mathbb{R}^{M}\rightarrow\mathbb{R}^{N}$.
Consider also the decision set
\begin{equation}
{\cal F}_{{\cal X}\left(Z_{Y}\right)}^{\mathscr{Y}}\triangleq\left\{ X:\Omega\rightarrow\mathbb{R}^{N}\hspace{-2pt}\left|\hspace{-2pt}\hspace{-2pt}\begin{array}{c}
X\left(\omega\right)\in{\cal X}\left(Z_{Y}\left(\omega\right)\right),\text{ }a.e.\text{ in }\omega\in\Omega\\
X^{-1}\left({\cal A}\right)\in\mathscr{Y},\text{ for all }{\cal A}\in\mathscr{B}\left(\mathbb{R}^{N}\right)
\end{array}\right.\hspace{-2pt}\hspace{-2pt}\hspace{-2pt}\right\} ,
\end{equation}
containing all $\mathscr{Y}$-measurable selections of ${\cal X}\left(Z_{Y}\right)$.
Then, ${\cal F}_{{\cal X}\left(Z_{Y}\right)}^{\mathscr{Y}}$ is nonempty.
Suppose that:
\begin{itemize}
\item $\mathbb{E}\left\{ g\left(\cdot,X\right)\right\} $ exists for all
$X\in{\cal F}_{{\cal X}\left(Z_{Y}\right)}^{\mathscr{Y}}$, with $\inf_{X\in{\cal F}_{{\cal X}\left(Z_{Y}\right)}^{\mathscr{Y}}}\mathbb{E}\left\{ g\left(\cdot,X\right)\right\} <+\infty$,
and that
\item $g$ is $\boldsymbol{SP}\diamondsuit{\cal F}_{{\cal X}\left(Z_{Y}\right)}^{\mathscr{Y}}$.
\end{itemize}
\noindent Then, if $\overline{\mathscr{Y}}$ denotes the completion
of $\mathscr{Y}$ relative to the restriction $\left.{\cal P}\right|_{\mathscr{Y}}$,
then the optimal value function $\inf_{\boldsymbol{x}\in{\cal X}\left(Z_{Y}\right)}\mathbb{E}\left\{ \left.g\left(\cdot,\boldsymbol{x}\right)\right|\mathscr{Y}\right\} \triangleq\vartheta$
is $\overline{\mathscr{Y}}$-measurable and it is true that
\begin{equation}
\inf_{X\in{\cal F}_{{\cal X}\left(Z_{Y}\right)}^{\mathscr{Y}}}\mathbb{E}\left\{ g\left(\cdot,X\right)\right\} \equiv\mathbb{E}\left\{ \inf_{\boldsymbol{x}\in{\cal X}\left(Z_{Y}\right)}\mathbb{E}\left\{ \left.g\left(\cdot,\boldsymbol{x}\right)\right|\mathscr{Y}\right\} \right\} \equiv\mathbb{E}\left\{ \vartheta\right\} .\label{eq:Lemma}
\end{equation}
In other words, variational minimization over ${\cal F}_{{\cal X}\left(Z_{Y}\right)}^{\mathscr{Y}}$
is exchangeable by pointwise (over constants) minimization over the
random multifunction ${\cal X}\left(Z_{Y}\right)$, relative to $\mathscr{Y}$.
\end{lem}
\begin{rem}
Note that, in the statement of Lemma \ref{lem:FUND_Lemma}, assuming
that the infimum of $\mathbb{E}\left\{ g\left(\cdot,X\right)\right\} $
over ${\cal F}_{{\cal X}\left(Z_{Y}\right)}^{\mathscr{Y}}$ is less
than $+\infty$ is \textit{equivalent }to assuming the existence of
an $X$ in ${\cal F}_{{\cal X}\left(Z_{Y}\right)}^{\mathscr{Y}}$,
such that $\mathbb{E}\left\{ g\left(\cdot,X\right)\right\} $ is less
than $+\infty$.\hfill{}\ensuremath{\blacksquare}
\end{rem}
Before embarking with the proof of Lemma \ref{lem:FUND_Lemma}, it
would be necessary to state an old, fundamental selection theorem,
due to Mackey \cite{Mackey1957Borel}.
\begin{thm}
\textbf{\textup{(Borel Measurable Selections }}\textup{\cite{Mackey1957Borel}}\textbf{\textup{)\label{thm:Selections}}}
Let $\left({\cal S}_{1},\mathscr{B}\left({\cal S}_{1}\right)\right)$
and $\left({\cal S}_{2},\mathscr{B}\left({\cal S}_{2}\right)\right)$
be Borel spaces and let $\left({\cal S}_{2},\mathscr{B}\left({\cal S}_{2}\right)\right)$
be standard. Let $\mu:\mathscr{B}\left({\cal S}_{1}\right)\rightarrow\left[0,\infty\right]$
be a standard measure on $\left({\cal S}_{1},\mathscr{B}\left({\cal S}_{1}\right)\right)$.
Suppose that ${\cal A}\in\mathscr{B}\left({\cal S}_{1}\right)\otimes\mathscr{B}\left({\cal S}_{2}\right)$,
such that, for each $y\in{\cal S}_{1}$, there exists $x_{y}\in{\cal S}_{2}$,
so that $\left(y,x_{y}\right)\in{\cal A}$. Then, there exists a Borel
subset ${\cal O}\in\mathscr{B}\left({\cal S}_{1}\right)$ with $\mu\left({\cal O}\right)\equiv0$,
as well as a Borel measurable function $\phi:{\cal S}_{1}\rightarrow{\cal S}_{2}$,
such that $\left(y,\phi\left(y\right)\right)\in{\cal A}$, for all
$y\in{\cal S}_{1}\setminus{\cal O}$.
\end{thm}
\begin{rem}
Theorem \ref{thm:Selections} refers to the concepts of a \textit{Borel
space, }a\textit{ standard Borel space} and a \textit{standard measure}.
These are employed as structural assumptions, in order for the conclusions
of the theorem to hold true. In this paper, except for the base probability
space $\left(\Omega,\mathscr{F},{\cal P}\right)$, whose structure
may be arbitrary, all other spaces and measures considered will satisfy
those assumptions by default. We thus choose not to present the respective
definitions; instead, the interested reader is referred to the original
article, \cite{Mackey1957Borel}.\hfill{}\ensuremath{\blacksquare}
\end{rem}
\noindent We are now ready to prove Lemma \ref{lem:FUND_Lemma}, as
follows.
\begin{proof}[Proof of Lemma \ref{lem:FUND_Lemma}]
As usual with such results, the proof will rely on showing a double
sided inequality \cite{Speyer2008STOCHASTIC,Astrom1970CONTROL,Bertsekas1978Stochastic,Bertsekas_Vol_2,DynamicProg_1971_FUND,Rockafellar2009VarAn}.
There is one major difficulty, though, in the optimization setting
considered, because all infima may be potentially unattainable, within
the respective decision sets. However, it is immediately evident that,
because $g$ is assumed to be \textit{$\boldsymbol{SP}\diamondsuit{\cal F}_{{\cal X}\left(Z_{Y}\right)}^{\mathscr{Y}}$},
and via a simple application of the tower property, it will suffice
to show that
\begin{equation}
\inf_{X\in{\cal F}_{{\cal X}\left(Z_{Y}\right)}^{\mathscr{Y}}}\mathbb{E}\left\{ h\left(Y,X\right)\right\} \equiv\mathbb{E}\left\{ \inf_{\boldsymbol{x}\in{\cal X}\left(Z_{Y}\right)}h\left(Y,\boldsymbol{x}\right)\right\} .
\end{equation}
This is because it is true that, for any $\mathscr{Y}$-measurable
selection of ${\cal X}\left(Z_{Y}\right)$, say $X:\Omega\rightarrow\mathbb{R}^{N}$,
for which $\mathbb{E}\left\{ g\left(\cdot,X\right)\right\} $ exists,
\begin{align}
\mathbb{E}\left\{ \left.g\left(\cdot,X\right)\right|\mathscr{Y}\right\} \left(\omega\right) & \equiv\left.h\left(Y\left(\omega\right),\boldsymbol{x}\right)\right|_{\boldsymbol{x}=X\left(\omega\right)},\quad\forall\omega\in\Omega_{\Pi_{X}},
\end{align}
where the event $\Omega_{\Pi_{X}}\in\mathscr{F}$ is such that ${\cal P}\left(\Omega_{\Pi_{X}}\right)\equiv1$
and $h$ is jointly Borel, satisfying
\begin{equation}
\mathbb{E}\left\{ \left.g\left(\cdot,\boldsymbol{x}\right)\right|\mathscr{Y}\right\} \equiv h\left(Y\left(\omega\right),\boldsymbol{x}\right),
\end{equation}
\textit{everywhere} in $\left(\omega,\boldsymbol{x}\right)\in\Omega\times\mathbb{R}^{N}$. 

For the sake of clarity in the exposition, we will break the proof
into a number of discrete subsections, providing a tractable roadmap
to the final result.

\medskip{}

\noindent \textbf{Step 1.} \textit{${\cal F}_{{\cal X}\left(Z_{Y}\right)}^{\mathscr{Y}}$
is nonempty.}

\noindent \rule[0.5ex]{0.2\columnwidth}{0.5pt}

It suffices to show that there exists at least one $\mathscr{Y}$-measurable
selection of ${\cal X}\left(Z_{Y}\right)$, that is, a $\mathscr{Y}$-measurable
random variable, say $X:\Omega\rightarrow\mathbb{R}^{N}$, such that
$X\left(\omega\right)\in{\cal X}\left(Z_{Y}\left(\omega\right)\right)$,
for all $\omega$ in the domain of ${\cal X}\left(Z_{Y}\right)$.

We first show that the composite multifunction ${\cal X}\left(Z_{Y}\left(\cdot\right)\right):\Omega\rightrightarrows\mathbb{R}^{N}$
is $\mathscr{Y}$-measurable. Recall from Definition \ref{def:MEAS_mult}
that it suffices to show that
\begin{equation}
{\cal X}Z_{Y}^{-1}\left({\cal A}\right)\triangleq\left\{ \omega\in\Omega\left|{\cal X}\left(Z_{Y}\left(\omega\right)\right)\bigcap{\cal A}\neq\varnothing\right.\right\} \in\mathscr{Y},
\end{equation}
for every closed ${\cal A}\subseteq\mathbb{R}^{N}$. Since the closed-valued
multifunction ${\cal X}$ is Borel measurable, it is true that ${\cal X}^{-1}\left({\cal A}\right)\in\mathscr{B}\left(\mathbb{R}^{N}\right)$,
for all closed ${\cal A}\subseteq\mathbb{R}^{N}$. We also know that
$Z_{Y}$ is $\mathscr{Y}$-measurable, or that $Z_{Y}^{-1}\left({\cal B}\right)\in\mathscr{Y}$,
for all ${\cal B}\in\mathscr{B}\left(\mathbb{R}^{N}\right)$. Setting
${\cal B}\equiv{\cal X}^{-1}\left({\cal A}\right)\in\mathscr{B}\left(\mathbb{R}^{N}\right)$,
for any arbitrary closed ${\cal A}\subseteq\mathbb{R}^{N}$, it is
true that
\begin{flalign}
\mathscr{Y}\ni Z_{Y}^{-1}\left({\cal X}^{-1}\left({\cal A}\right)\right) & \equiv\left\{ \omega\in\Omega\left|Z_{Y}\left(\omega\right)\in{\cal X}^{-1}\left({\cal A}\right)\right.\right\} \nonumber \\
 & \equiv\left\{ \omega\in\Omega\left|{\cal X}\left(Z_{Y}\left(\omega\right)\right)\bigcap{\cal A}\neq\varnothing\right.\right\} \nonumber \\
 & \equiv{\cal X}Z_{Y}^{-1}\left({\cal A}\right),
\end{flalign}
and, thus, the composition ${\cal X}\left(Z_{Y}\left(\cdot\right)\right)$
is $\mathscr{Y}$-measurable, or, in other words, measurable on the
measurable (sub)space $\left(\Omega,\mathscr{Y}\right)$.

Now, since the closed-valued multifunction ${\cal X}\left(Z_{Y}\right)$
is measurable on $\left(\Omega,\mathscr{Y}\right)$, it admits a \textit{Castaing
Representation} (Theorem 14.5 in \cite{Rockafellar2009VarAn} \& Theorem
7.34 in \cite{Shapiro2009STOCH_PROG}). Therefore, there exists at
least one $\mathscr{Y}$-measurable selection of ${\cal X}\left(Z_{Y}\right)$,
which means that ${\cal F}_{{\cal X}\left(Z_{Y}\right)}^{\mathscr{Y}}$
contains at least one element.\hfill{}\ensuremath{\bigstar}

\medskip{}

\noindent \textbf{Step 2.} \textit{$\vartheta$ is $\overline{\mathscr{Y}}$-measurable.}

\noindent \rule[0.5ex]{0.2\columnwidth}{0.5pt}

To show the validity of this statement, we first demonstrate that,
for any chosen $h:\mathbb{R}^{M}\times\mathbb{R}^{N}\rightarrow\overline{\mathbb{R}}$,
as in Definition \ref{def:SUB}, the function $\xi:\mathbb{R}^{M}\rightarrow\overline{\mathbb{R}}$,
defined as
\begin{flalign}
\xi\left(\boldsymbol{y}\right) & \triangleq\inf_{\boldsymbol{x}\in{\cal X}\left({\cal Z}\left(\boldsymbol{y}\right)\right)}h\left(\boldsymbol{y},\boldsymbol{x}\right),\quad\forall\boldsymbol{y}\in\mathbb{R}^{M},\label{eq:Pointwise_1-2}
\end{flalign}
is measurable relative to $\overline{\mathscr{B}}\hspace{-2pt}\left(\mathbb{R}^{M}\right)$,
the completion of $\mathscr{B}\hspace{-2pt}\left(\mathbb{R}^{M}\right)$
relative to the pushforward ${\cal P}_{Y}$. This follows easily from
the following facts. First, the \textit{graph} of the measurable multifunction
${\cal X}\left({\cal Z}\left(\cdot\right)\right)$ is itself measurable
and in $\mathscr{B}\left(\mathbb{R}^{M}\right)\otimes\mathscr{B}\left(\mathbb{R}^{N}\right)$
(Theorem 14.8 in \cite{Rockafellar2009VarAn}), and, therefore, \textit{analytic}
(Appendix A.2 in \cite{Bertsekas_Vol_2}). Second, $h$ is jointly
Borel measurable and, therefore, a\textit{ lower semianalytic} function
(Appendix A.2 in \cite{Bertsekas_Vol_2}). As a result, (\cite{Bertsekas1978Stochastic},
Proposition 7.47) implies that $\xi$ is also lower semianalytic,
and, consequently, \textit{universally measurable} (Appendix A.2 in
\cite{Bertsekas_Vol_2}). Being universally measurable, $\xi$ is
also measurable relative to $\overline{\mathscr{B}}\hspace{-2pt}\left(\mathbb{R}^{M}\right)$,
thus proving our claim. We also rely on the definitions of both $\overline{\mathscr{Y}}$
and $\overline{\mathscr{B}}\left(\mathbb{R}^{M}\right)$, stated as
(Theorem 1.9 in \cite{Folland2013_RealVar})
\begin{flalign}
{\cal B}\in\overline{\mathscr{Y}} & \iff{\cal B}\equiv{\cal C}\bigcup{\cal D}\left.\vphantom{\mathscr{B}\left(\mathbb{R}^{M}\right)}\right|{\cal C}\in\mathscr{Y}\text{ and }{\cal D\subseteq{\cal O}\in\mathscr{Y}},\text{ with }\left.{\cal P}\right|_{\mathscr{Y}}\left({\cal O}\right)\equiv0\quad\text{and}\\
{\cal B}\in\overline{\mathscr{B}}\left(\mathbb{R}^{M}\right) & \iff{\cal B}\equiv{\cal C}\bigcup{\cal D}\left.\vphantom{\mathscr{B}\left(\mathbb{R}^{M}\right)}\right|{\cal C}\in\mathscr{B}\left(\mathbb{R}^{M}\right)\text{ and }{\cal D\subseteq{\cal O}}\in\mathscr{B}\left(\mathbb{R}^{M}\right),\text{ with }{\cal P}_{Y}\left({\cal O}\right)\equiv0.\label{eq:COMPLETE_2}
\end{flalign}

Now, specifically, to show that $\vartheta$ is measurable relative
to $\overline{\mathscr{Y}}$, it suffices to show that, for every
Borel ${\cal A}\in\mathscr{B}\left(\overline{\mathbb{R}}\right)$,
\begin{equation}
\vartheta^{-1}\left({\cal A}\right)\triangleq\left\{ \left.\omega\in\Omega\right|\vartheta\left(\omega\right)\in{\cal A}\right\} \in\overline{\mathscr{Y}}.\label{eq:COMPLETE_1}
\end{equation}
Recall, that, by definition of $\xi$, it is true that $\xi\left(Y\left(\omega\right)\right)\equiv\vartheta\left(\omega\right)$,
for all $\omega\in\Omega$. Then, for every ${\cal A}\in\mathscr{B}\left(\overline{\mathbb{R}}\right)$,
we may write
\begin{flalign}
\vartheta^{-1}\left({\cal A}\right) & \equiv\xi Y^{-1}\left({\cal A}\right)\nonumber \\
 & \equiv\left\{ \left.\omega\in\Omega\right|\xi\left(Y\left(\omega\right)\right)\in{\cal A}\right\} \nonumber \\
 & \equiv\left\{ \left.\omega\in\Omega\right|Y\left(\omega\right)\in\xi^{-1}\left({\cal A}\right)\right\} \nonumber \\
 & \triangleq Y^{-1}\left(\xi^{-1}\left({\cal A}\right)\right).
\end{flalign}
But $\xi^{-1}\left({\cal A}\right)\in\overline{\mathscr{B}}\left(\mathbb{R}^{M}\right)$,
which, by (\ref{eq:COMPLETE_2}), equivalently means that $\xi^{-1}\left({\cal A}\right)\equiv{\cal G}_{{\cal A}}\bigcup{\cal H}_{{\cal A}}$,
for some ${\cal G}_{{\cal A}}\in\mathscr{B}\left(\mathbb{R}^{M}\right)$
and some ${\cal H}_{{\cal A}}\subseteq{\cal E}_{{\cal A}}\in\mathscr{B}\left(\mathbb{R}^{M}\right)$,
with ${\cal P}_{Y}\left({\cal E}_{{\cal A}}\right)\equiv0$. Thus,
we may further express any ${\cal A}$-preimage of $\vartheta$ as
\begin{align}
\vartheta^{-1}\left({\cal A}\right) & \equiv Y^{-1}\left({\cal G}_{{\cal A}}\bigcup{\cal H}_{{\cal A}}\right)\nonumber \\
 & \equiv Y^{-1}\left({\cal G}_{{\cal A}}\right)\bigcup Y^{-1}\left({\cal H}_{{\cal A}}\right).\label{eq:COMPLETE_3}
\end{align}
Now, because ${\cal G}_{{\cal A}}$ is Borel and $Y$ is a random
element, it is true that $Y^{-1}\left({\cal G}_{{\cal A}}\right)\in\mathscr{Y}$.
On the other hand, ${\cal H}_{{\cal A}}\subseteq{\cal E}_{{\cal A}}$,
which implies that $Y^{-1}\left({\cal H}_{{\cal A}}\right)\subseteq Y^{-1}\left({\cal E}_{{\cal A}}\right)$,
where
\begin{equation}
\left.{\cal P}\right|_{\mathscr{Y}}\left(Y^{-1}\left({\cal E}_{{\cal A}}\right)\right)\equiv{\cal P}_{Y}\left({\cal E}_{{\cal A}}\right)\equiv0.
\end{equation}
Therefore, we have shown that, for every ${\cal A}\in\mathscr{B}\left(\overline{\mathbb{R}}\right)$,
$\vartheta^{-1}\left({\cal A}\right)$ may always be written as a
union of an element in $\mathscr{Y}$ and some subset of a $\left.{\cal P}\right|_{\mathscr{Y}}$-null
set, also in $\mathscr{Y}$. Enough said.\hfill{}\ensuremath{\bigstar}

\medskip{}

\noindent \textbf{Step 3.} \textit{For every $X\in{\cal F}_{{\cal X}\left(Z_{Y}\right)}^{\mathscr{Y}}$,
it is true that $h\left(Y,X\right)\ge\inf_{\boldsymbol{x}\in{\cal X}\left(Z_{Y}\right)}h\left(Y,\boldsymbol{x}\right)\equiv\vartheta.$}

\noindent \rule[0.5ex]{0.2\columnwidth}{0.5pt}

For each $\omega\in\Omega$ (which also determines $Y$), we may write
\begin{align}
\vartheta\left(\omega\right) & \equiv\inf_{\boldsymbol{x}\in{\cal X}\left({\cal Z}\left(Y\left(\omega\right)\right)\right)}h\left(Y\left(\omega\right),\boldsymbol{x}\right)\nonumber \\
 & \equiv\inf_{{\cal M}\left(Y\left(\omega\right)\right)\in{\cal X}\left(Z_{Y}\left(\omega\right)\right)}h\left(Y\left(\omega\right),{\cal M}\left(Y\left(\omega\right)\right)\right),
\end{align}
where ${\cal M}:\mathbb{R}^{M}\rightarrow\mathbb{R}^{N}$ is of arbitrary
nature. Therefore, $\vartheta$ may be equivalently regarded as the
result of infimizing $h$ over the set of all, \textit{measurable
or not}, functionals of $Y$, which are also selections of ${\cal X}\left(Z_{Y}\right)$.
This set, of course, includes ${\cal F}_{{\cal X}\left(Z_{Y}\right)}^{\mathscr{Y}}$.
Now, choose an $X\equiv{\cal M}_{X}\left(Y\right)\in{\cal F}_{{\cal X}\left(Z_{Y}\right)}^{\mathscr{Y}}$,
as above, for some Borel measurable ${\cal M}_{X}:\mathbb{R}^{M}\rightarrow\mathbb{R}^{N}$.
Then, it must be true that
\begin{equation}
\vartheta\left(\omega\right)\le h\left(Y\left(\omega\right),{\cal M}_{X}\left(Y\left(\omega\right)\right)\right)\equiv h\left(Y\left(\omega\right),X\left(\omega\right)\right),
\end{equation}
everywhere in $\omega\in\Omega$.\hfill{}\ensuremath{\bigstar}

\medskip{}

\noindent \textbf{Step 4.} \textit{It is also true that
\begin{equation}
\inf_{X\in{\cal F}_{{\cal X}\left(Z_{Y}\right)}^{\mathscr{Y}}}\mathbb{E}\left\{ h\left(Y,X\right)\right\} \ge\mathbb{E}\left\{ \inf_{\boldsymbol{x}\in{\cal X}\left(Z_{Y}\right)}h\left(Y,\boldsymbol{x}\right)\right\} .
\end{equation}
}\rule[0.5ex]{0.2\columnwidth}{0.5pt}

From \textbf{Step 3}, we know that, for every $X\in{\cal F}_{{\cal X}\left(Z_{Y}\right)}^{\mathscr{Y}}$,
we have
\begin{equation}
h\left(Y,X\right)\ge\vartheta.\label{eq:UNFORGIVEN}
\end{equation}
At this point, we exploit measurability of $\vartheta$, proved in
\textbf{Step 2}. Since, by assumption, 
\begin{equation}
\inf_{X\in{\cal F}_{{\cal X}\left(Z_{Y}\right)}^{\mathscr{Y}}}\mathbb{E}\left\{ h\left(Y,X\right)\right\} \equiv\inf_{X\in{\cal F}_{{\cal X}\left(Z_{Y}\right)}^{\mathscr{Y}}}\mathbb{E}\left\{ g\left(\cdot,X\right)\right\} <+\infty,
\end{equation}
it follows that there exists $X_{F}\in{\cal F}_{{\cal X}\left(Z_{Y}\right)}^{\mathscr{Y}}$,
such that $\mathbb{E}\left\{ h\left(Y,X_{F}\right)\right\} <+\infty$
(recall that the integral $\mathbb{E}\left\{ g\left(\cdot,X_{F}\right)\right\} $
exists anyway, also by assumption). Since (\ref{eq:UNFORGIVEN}) holds
for every $X\in{\cal F}_{{\cal X}\left(Z_{Y}\right)}^{\mathscr{Y}}$,
it also holds for $X_{F}\in{\cal F}_{{\cal X}\left(Z_{Y}\right)}^{\mathscr{Y}}$
and, consequently, the integral of $\vartheta$ exists, with $\mathbb{E}\left\{ \vartheta\right\} <+\infty$.
Then, we may take expectations on both sides of (\ref{eq:UNFORGIVEN})
(Theorem 1.5.9 (b) in \cite{Ash2000Probability}), yielding
\begin{equation}
\mathbb{E}\left\{ h\left(Y,X\right)\right\} \ge\mathbb{E}\left\{ \vartheta\right\} ,\quad\forall X\in{\cal F}_{{\cal X}\left(Z_{Y}\right)}^{\mathscr{Y}}.
\end{equation}
Infimizing additionally both sides over $X\in{\cal F}_{{\cal X}\left(Z_{Y}\right)}^{\mathscr{Y}}$,
we obtain the desired inequality.

We may also observe that, if $\inf_{X\in{\cal F}_{{\cal X}\left(Z_{Y}\right)}^{\mathscr{Y}}}\mathbb{E}\left\{ h\left(Y,X\right)\right\} \equiv-\infty$,
then 
\begin{equation}
\inf_{X\in{\cal F}_{{\cal X}\left(Z_{Y}\right)}^{\mathscr{Y}}}\mathbb{E}\left\{ h\left(Y,X\right)\right\} \equiv\mathbb{E}\left\{ \vartheta\right\} \equiv-\infty,
\end{equation}
and the conclusion of Lemma \ref{lem:FUND_Lemma} holds immediately.
Therefore, in the following, we may assume that $\inf_{X\in{\cal F}_{{\cal X}\left(Z_{Y}\right)}^{\mathscr{Y}}}\mathbb{E}\left\{ h\left(Y,X\right)\right\} >-\infty$.\hfill{}\ensuremath{\bigstar}

\medskip{}

\noindent \textbf{Step 5.}\textit{ For every $\varepsilon>0$, $n\in\mathbb{N}$
and every $\boldsymbol{y}\in\mathbb{R}^{M}$, there exists $\boldsymbol{x}\equiv\boldsymbol{x}_{\boldsymbol{y}}\in{\cal X}\left({\cal Z}\left(\boldsymbol{y}\right)\right)$,
such that
\begin{equation}
h\left(\boldsymbol{y},\boldsymbol{x}_{\boldsymbol{y}}\right)\le\max\left\{ \xi\left(\boldsymbol{y}\right),-n\right\} +\varepsilon.\label{eq:E_OPT_POINT}
\end{equation}
}\rule[0.5ex]{0.2\columnwidth}{0.5pt}

This simple fact may be shown by contradiction; replacing the universal
with existential quantifiers and vice versa in the above statement,
suppose that there exists $\varepsilon>0$ , $n\in\mathbb{N}$, and
$\boldsymbol{y}\in\mathbb{R}^{M}$ such that, for all $\boldsymbol{x}\in{\cal X}\left({\cal Z}\left(\boldsymbol{y}\right)\right)$,
$h\left(\boldsymbol{y},\boldsymbol{x}\right)>\max\left\{ \xi\left(\boldsymbol{y}\right),-n\right\} +\varepsilon$.
There are two cases: \textbf{1)} $\xi\left(\boldsymbol{y}\right)>-\infty$.
In this case, $\max\left\{ \xi\left(\boldsymbol{y}\right),-n\right\} \ge\xi\left(\boldsymbol{y}\right)$,
which would imply that, for all $\boldsymbol{x}\in{\cal X}\left({\cal Z}\left(\boldsymbol{y}\right)\right)$,
\begin{equation}
h\left(\boldsymbol{y},\boldsymbol{x}\right)>\xi\left(\boldsymbol{y}\right)+\varepsilon,
\end{equation}
contradicting the fact that $\xi\left(\boldsymbol{y}\right)$ is the
infimum (the greatest lower bound) of $h\left(\boldsymbol{y},\boldsymbol{x}\right)$
over ${\cal X}\left({\cal Z}\left(\boldsymbol{y}\right)\right)$,
since $\varepsilon>0$. \textbf{2)} $\xi\left(\boldsymbol{y}\right)\equiv-\infty$.
Here, $\max\left\{ \xi\left(\boldsymbol{y}\right),-n\right\} \equiv-n$,
and, for all $\boldsymbol{x}\in{\cal X}_{{\cal Z}}\left(\boldsymbol{y}\right)$,
we would write 
\begin{equation}
h\left(\boldsymbol{y},\boldsymbol{x}\right)>-n+\varepsilon\in\mathbb{R},
\end{equation}
which, again, contradicts the fact that $-\infty\equiv\xi\left(\boldsymbol{y}\right)$
is the infimum of $h\left(\boldsymbol{y},\boldsymbol{x}\right)$ over
${\cal X}\left({\cal Z}\left(\boldsymbol{y}\right)\right)$. Therefore,
in both cases, we are led to a contradiction, implying that the statement
preceding and including (\ref{eq:E_OPT_POINT}) is true. The idea
of using the maximum operator, so that $\xi\left(\boldsymbol{y}\right)$
may be allowed to take the value $-\infty$, is credited to and borrowed
from (\cite{Rockafellar2009VarAn}, proof of Theorem 14.60).\hfill{}\ensuremath{\bigstar}\medskip{}

\noindent \textbf{Step 6.} \textit{There exists a Borel measurable
function $\widetilde{\xi}:\mathbb{R}^{M}\rightarrow\overline{\mathbb{R}}$,
such that
\begin{equation}
\widetilde{\xi}\left(\boldsymbol{y}\right)\equiv\xi\left(\boldsymbol{y}\right),\quad\forall\boldsymbol{y}\in\overline{{\cal R}}_{\xi}\supseteq{\cal R}_{\xi},
\end{equation}
where ${\cal R}_{\xi}\in\mathscr{B}\left(\mathbb{R}^{M}\right)$ is
such that ${\cal P}_{Y}\left({\cal R}_{\xi}\right)\equiv1$, and $\overline{{\cal R}}_{\xi}\in\overline{\mathscr{B}}\left(\mathbb{R}^{M}\right)$
is such that $\overline{{\cal P}}_{Y}\left(\overline{{\cal R}}_{\xi}\right)\equiv1$,
where $\overline{{\cal P}}_{Y}$ denotes the completion of the pushforward
${\cal P}_{Y}$.}

\noindent \rule[0.5ex]{0.2\columnwidth}{0.5pt}

From (\cite{Folland2013_RealVar}, Proposition 2.12), we know that,
since \textit{$\xi$ }is $\overline{\mathscr{B}}\left(\mathbb{R}^{M}\right)$-measurable,
there exists a $\mathscr{B}\left(\mathbb{R}^{M}\right)$-measurable
function $\widetilde{\xi}:\mathbb{R}^{M}\rightarrow\overline{\mathbb{R}}$,
such that
\begin{equation}
\widetilde{\xi}\left(\boldsymbol{y}\right)\equiv\xi\left(\boldsymbol{y}\right),\quad\forall\boldsymbol{y}\in\overline{{\cal R}}_{\xi},
\end{equation}
where $\overline{{\cal R}}_{\xi}$ is an event in $\overline{\mathscr{B}}\left(\mathbb{R}^{M}\right)$,
such that $\overline{{\cal P}}_{Y}\left(\overline{{\cal R}}_{\xi}\right)\equiv1$.
However, from \textbf{Step 2} (see (\ref{eq:COMPLETE_2})), we know
that $\overline{{\cal R}}_{\xi}\equiv{\cal R}_{\xi}\bigcup\overline{{\cal R}}_{\xi}^{E}$,
where ${\cal R}_{\xi}\in\mathscr{B}\left(\mathbb{R}^{M}\right)$ and
$\overline{{\cal P}}_{Y}\left(\overline{{\cal R}}_{\xi}^{E}\right)\equiv0$.
Then, it may be easily shown that $\overline{{\cal P}}_{Y}\left(\overline{{\cal R}}_{\xi}\right)\equiv\overline{{\cal P}}_{Y}\left({\cal R}_{\xi}\right)\equiv1$
and, since $\overline{{\cal P}}_{Y}$ and ${\cal P}_{Y}$ agree on
the elements of $\mathscr{B}\left(\mathbb{R}^{M}\right)$, ${\cal P}_{Y}\left({\cal R}_{\xi}\right)\equiv1$,
as well.\hfill{}\ensuremath{\bigstar}

\medskip{}

\noindent \textbf{Step 7.} \textit{There exists a $\left({\cal P},\varepsilon,n\right)$-optimal
selector $X_{n}^{\varepsilon}\in{\cal F}_{{\cal X}\left(Z_{Y}\right)}^{\mathscr{Y}}$:
For every $\varepsilon>0$ and for every $n\in\mathbb{N}$, there
exists $X_{n}^{\varepsilon}\in{\cal F}_{{\cal X}\left(Z_{Y}\right)}^{\mathscr{Y}}$,
such that
\begin{equation}
h\left(Y,X_{n}^{\varepsilon}\right)\le\max\left\{ \inf_{\boldsymbol{x}\in{\cal X}\left(Z_{Y}\right)}h\left(Y,\boldsymbol{x}\right),-n\right\} +\varepsilon,\quad{\cal P}-a.e..\label{eq:EPSILON_CORE}
\end{equation}
}\rule[0.5ex]{0.2\columnwidth}{0.5pt}

This is the most crucial property of the problem that needs to be
established, in order to reach to the final conclusions of Lemma \ref{lem:FUND_Lemma}.
In this step, we make use of Theorem \ref{thm:Selections}. Because
Theorem \ref{thm:Selections} works on Borel spaces, in the following,
it will be necessary to work directly on the state space of the random
element $Y$, equipped with its Borel $\sigma$-algebra, and the pushforward
${\cal P}_{Y}$. In the following, we will also make use of the results
proved in \textbf{Step 5 }and \textbf{Step 6}.

Recall the definition of $\xi$ in the statement of Lemma \ref{lem:FUND_Lemma}.
We may readily show that the multifunction ${\cal X}\left({\cal Z}\left(\cdot\right)\right)$
is $\mathscr{B}\left(\mathbb{R}^{M}\right)$-measurable. This may
be shown in exactly the same way as in \textbf{Step 1}, exploiting
the hypotheses that the multifunction ${\cal X}$ and the function
${\cal Z}$ are both Borel measurable. Borel measurability of ${\cal X}\left({\cal Z}\left(\cdot\right)\right)$
will be exploited shortly.

Compare the result of \textbf{Step 5} with what we would like to prove
here; the statement preceding and including (\ref{eq:E_OPT_POINT})
is not enough for our purposes; what we would actually like is to
be able to generate a \textit{selector}, that is, a function of $\boldsymbol{y}$
such that (\ref{eq:E_OPT_POINT}) would hold at least almost everywhere
with respect to ${\cal P}_{Y}$. This is why we need Theorem \ref{thm:Selections}.
The idea of using Theorem \ref{thm:Selections} into this context
is credited to and borrowed from \cite{Strauch1966Negative}.

From \textbf{Step 2} and \textbf{Step 6}, we know that $\xi$ is $\overline{\mathscr{B}}\left(\mathbb{R}^{M}\right)$-measurable
and that there exists a Borel measurable function $\widetilde{\xi}:\mathbb{R}^{M}\rightarrow\overline{\mathbb{R}}$,
such that $\widetilde{\xi}\left(\boldsymbol{y}\right)\equiv\xi\left(\boldsymbol{y}\right),$
everywhere in $\boldsymbol{y}\in{\cal R}_{\xi}$, where ${\cal R}_{\xi}\in\mathscr{B}\left(\mathbb{R}^{M}\right)$
is such that ${\cal P}_{Y}\left({\cal R}_{\xi}\right)\equiv\overline{{\cal P}}_{Y}\left({\cal R}_{\xi}\right)\equiv1$.
Then, it follows that 
\begin{equation}
\widetilde{\xi}\left(\boldsymbol{y}\right)\equiv\inf_{\boldsymbol{x}\in{\cal X}\left({\cal Z}\left(\boldsymbol{y}\right)\right)}h\left(\boldsymbol{y},\boldsymbol{x}\right),
\end{equation}
for all $\boldsymbol{y}\in{\cal R}_{\xi}$.

Define, for brevity, ${\cal X}_{{\cal Z}}\left(\boldsymbol{y}\right)\triangleq{\cal X}\left({\cal Z}\left(\boldsymbol{y}\right)\right)$,
for all $\boldsymbol{y}\in\mathbb{R}^{M}$. Towards the application
of Theorem \ref{thm:Selections}, fix any $\varepsilon>0$ and any
$n\in\mathbb{N}$ and consider the set
\begin{equation}
\Pi_{{\cal X}_{{\cal Z}}}^{\varepsilon,n}\equiv\left\{ \left(\boldsymbol{y},\boldsymbol{x}\right)\in\mathbb{R}^{M}\times\mathbb{R}^{N}\left|\begin{array}{cc}
\begin{array}{c}
\boldsymbol{x}\in{\cal X}\left({\cal Z}\left(\boldsymbol{y}\right)\right)\\
h\left(\boldsymbol{y},\boldsymbol{x}\right)\le\max\left\{ \widetilde{\xi}\left(\boldsymbol{y}\right),-n\right\} +\varepsilon
\end{array}, & \text{if }\boldsymbol{y}\in{\cal R}_{\xi}\\
\boldsymbol{x}\in{\cal X}\left({\cal Z}\left(\boldsymbol{y}\right)\right), & \text{if }\boldsymbol{y}\in{\cal R}_{\xi}^{c}
\end{array}\right.\right\} .
\end{equation}
We will show that $\Pi_{\varepsilon}^{n}$ constitutes a measurable
set in $\mathscr{B}\left(\mathbb{R}^{M}\right)\otimes\mathscr{B}\left(\mathbb{R}^{N}\right)$.
Observe that $\Pi_{{\cal X}_{{\cal Z}}}^{\varepsilon,n}\equiv\Pi_{{\cal X}_{{\cal Z}}}\bigcap\left(\Pi^{\varepsilon,n}\bigcup\Pi_{rem}\right)$,
where we define
\begin{align}
\Pi_{{\cal X}_{{\cal Z}}} & \triangleq\left\{ \left.\left(\boldsymbol{y},\boldsymbol{x}\right)\in\mathbb{R}^{M}\times\mathbb{R}^{N}\right|\boldsymbol{x}\in{\cal X}\left({\cal Z}\left(\boldsymbol{y}\right)\right)\right\} ,\\
\Pi^{\varepsilon,n} & \triangleq\left\{ \left(\boldsymbol{y},\boldsymbol{x}\right)\in\mathbb{R}^{M}\times\mathbb{R}^{N}\left|\boldsymbol{y}\in{\cal R}_{\xi},h\left(\boldsymbol{y},\boldsymbol{x}\right)\le\max\left\{ \widetilde{\xi}\left(\boldsymbol{y}\right),-n\right\} +\varepsilon\right.\right\} \quad\text{and}\\
\Pi_{rem} & \triangleq\left\{ \left.\left(\boldsymbol{y},\boldsymbol{x}\right)\in\mathbb{R}^{M}\times\mathbb{R}^{N}\right|\boldsymbol{y}\in{\cal R}_{\xi}^{c}\right\} .
\end{align}
Clearly, it suffices to show that both $\Pi_{{\cal X}_{{\cal Z}}}$
and $\Pi^{\varepsilon,n}$ are in $\mathscr{B}\left(\mathbb{R}^{M}\right)\otimes\mathscr{B}\left(\mathbb{R}^{N}\right)$.
First, the set $\Pi_{{\cal X}_{{\cal Z}}}$ is the \textit{graph}
of the multifunction ${\cal X}_{{\cal Z}}$, and, because ${\cal X}_{{\cal Z}}$
is measurable, it follows from (\cite{Rockafellar2009VarAn}, Theorem
14.8) that $\Pi_{{\cal X}_{{\cal Z}}}\in\mathscr{B}\left(\mathbb{R}^{M}\right)\otimes\mathscr{B}\left(\mathbb{R}^{N}\right)$.
Second, because $g$ is \textit{$\boldsymbol{SP}\diamondsuit{\cal F}_{{\cal X}\left(Z_{Y}\right)}^{\mathscr{Y}}$},
$h$ is jointly Borel measurable. Additionally, ${\cal R}_{\xi}$
is Borel and $\widetilde{\xi}$ is Borel as well. Consequently, $\Pi^{\varepsilon,n}$
can be written as the intersection of two measurable sets, implying
that it is in $\mathscr{B}\left(\mathbb{R}^{M}\right)\otimes\mathscr{B}\left(\mathbb{R}^{N}\right)$,
as well. And third, $\Pi_{rem}\in\mathscr{B}\left(\mathbb{R}^{M}\right)\otimes\mathscr{B}\left(\mathbb{R}^{N}\right)$,
since ${\cal R}_{\xi}^{c}$ is Borel, as a complement of a Borel set.
Therefore, $\Pi_{{\cal X}_{{\cal Z}}}^{\varepsilon,n}\in\mathscr{B}\left(\mathbb{R}^{M}\right)\otimes\mathscr{B}\left(\mathbb{R}^{N}\right)$.

Now, we have to verify the selection property, set as a requirement
in the statement of Theorem \ref{thm:Selections}. Indeed, for every
$\boldsymbol{y}\in{\cal R}_{\xi}$, there exists $\boldsymbol{x}_{\boldsymbol{y}}\in{\cal X}\left({\cal Z}\left(\boldsymbol{y}\right)\right)$,
such that (\ref{eq:E_OPT_POINT}) holds, where $\xi\left(\boldsymbol{y}\right)\equiv\widetilde{\xi}\left(\boldsymbol{y}\right)$
(see \textbf{Step 6 }and above), while, for every $\boldsymbol{y}\in{\cal R}_{\xi}^{c}$,
any $\boldsymbol{x}_{\boldsymbol{y}}\in{\cal X}\left({\cal Z}\left(\boldsymbol{y}\right)\right)$
will do. Thus, for every $\boldsymbol{y}\in\mathbb{R}^{M}$, there
exists $\boldsymbol{x}_{\boldsymbol{y}}\in\mathbb{R}^{N}$, such that
$\left(\boldsymbol{y},\boldsymbol{x}_{\boldsymbol{y}}\right)\in\Pi_{{\cal X}_{{\cal Z}}}^{\varepsilon,n}$.
As a result, Theorem \ref{thm:Selections} applies and implies that
there exists a Borel subset ${\cal R}_{\Pi_{{\cal X}_{{\cal Z}}}^{\varepsilon,n}}^{c}$
of ${\cal P}_{Y}$-measure $0$, as well as a Borel measurable selector
$\mathsf{S}_{n}^{\varepsilon}:\mathbb{R}^{M}\rightarrow\mathbb{R}^{N}$,
such that, $\left(\boldsymbol{y},\mathsf{S}_{n}^{\varepsilon}\left(\boldsymbol{y}\right)\right)\in\Pi_{{\cal X}_{{\cal Z}}}^{\varepsilon,n}$,
for all $\boldsymbol{y}\in{\cal R}_{\Pi_{{\cal X}_{{\cal Z}}}^{\varepsilon,n}}$.
In other words, the Borel measurable selector $\mathsf{S}_{n}^{\varepsilon}$
is such that
\begin{equation}
\begin{array}{cc}
\mathsf{S}_{n}^{\varepsilon}\left(\boldsymbol{y}\right)\in{\cal X}\left({\cal Z}\left(\boldsymbol{y}\right)\right)\quad\text{and}\\
h\left(\boldsymbol{y},\mathsf{S}_{n}^{\varepsilon}\left(\boldsymbol{y}\right)\right)\le\max\left\{ \widetilde{\xi}\left(\boldsymbol{y}\right),-n\right\} +\varepsilon, & \quad\forall\boldsymbol{y}\in{\cal R}_{\xi}\bigcap{\cal R}_{\Pi_{{\cal X}_{{\cal Z}}}^{\varepsilon,n}}\triangleq{\cal R}_{\Pi_{{\cal X}_{{\cal Z}}}^{\varepsilon,n}}^{\xi},
\end{array}\label{eq:EPSILON_1}
\end{equation}
where, of course, ${\cal P}_{Y}\left({\cal R}_{\xi}\bigcap{\cal R}_{\Pi_{{\cal X}_{{\cal Z}}}^{\varepsilon,n}}\right)\equiv1$.
Additionally, (\ref{eq:EPSILON_1}) must be true at $\boldsymbol{y}=Y\left(\omega\right)$,
as long as $\omega$ is such that the values of $Y$ are restricted
to ${\cal R}_{\Pi_{{\cal X}_{{\cal Z}}}^{\varepsilon,n}}^{\xi}$.
Equivalently, we demand that
\begin{equation}
\omega\in\left\{ \left.\omega\in\Omega\right|Y\left(\omega\right)\in{\cal R}_{\Pi_{{\cal X}_{{\cal Z}}}^{\varepsilon,n}}^{\xi}\right\} \equiv Y^{-1}\left({\cal R}_{\Pi_{{\cal X}_{{\cal Z}}}^{\varepsilon,n}}^{\xi}\right)\triangleq\Omega_{n}^{\varepsilon}.
\end{equation}
But ${\cal R}_{\Pi_{{\cal X}_{{\cal Z}}}^{\varepsilon,n}}^{\xi}\in\mathscr{B}\left(\mathbb{R}^{M}\right)$
and $Y$ is a random element and, hence, a measurable function for
$\Omega$ to $\mathbb{R}^{M}$. This means that $\Omega_{n}^{\varepsilon}\in\mathscr{Y}$
and we are allowed to write
\begin{flalign}
{\cal P}\left(\Omega_{n}^{\varepsilon}\right) & \equiv\int_{Y^{-1}\left({\cal R}_{\Pi_{{\cal X}_{{\cal Z}}}^{\varepsilon,n}}^{\xi}\right)}{\cal P}\left(\mathrm{d}\omega\right)\nonumber \\
 & =\int_{\left\{ \left.\boldsymbol{y}\in\mathbb{R}^{M}\right|\boldsymbol{y}\in{\cal R}_{\Pi_{{\cal X}_{{\cal Z}}}^{\varepsilon,n}}^{\xi}\right\} }{\cal P}_{Y}\left(\mathrm{d}\boldsymbol{y}\right)\nonumber \\
 & \equiv{\cal P}_{Y}\left({\cal R}_{\Pi_{{\cal X}_{{\cal Z}}}^{\varepsilon,n}}^{\xi}\right)\equiv1.
\end{flalign}
Therefore, we may pull (\ref{eq:EPSILON_1}) back to the base space,
and restate it as
\begin{equation}
\begin{array}{cc}
\mathsf{S}_{n}^{\varepsilon}\left(Y\left(\omega\right)\right)\in{\cal X}\left(Z_{Y}\left(\omega\right)\right)\quad\text{and}\\
h\left(Y\left(\omega\right),\mathsf{S}_{n}^{\varepsilon}\left(Y\left(\omega\right)\right)\right)\le\max\left\{ \xi\left(Y\left(\omega\right)\right),-n\right\} +\varepsilon, & \quad\forall\omega\in\Omega_{n}^{\varepsilon},
\end{array}
\end{equation}
where $\Omega_{n}^{\varepsilon}\subseteq\Omega$ is an event, such
that ${\cal P}\left(\Omega_{n}^{\varepsilon}\right)\equiv1$. Then,
by construction, $\mathsf{S}_{n}^{\varepsilon}\left(Y\right)\in{\cal F}_{{\cal X}\left(Z_{Y}\right)}^{\mathscr{Y}}$.
As a result, for any choice of $\varepsilon>0$ and $n\in\mathbb{N}$,
the selector $X_{n}^{\varepsilon}\triangleq\mathsf{S}_{n}^{\varepsilon}\left(Y\right)\in{\cal F}_{{\cal X}\left(Z_{Y}\right)}^{\mathscr{Y}}$
is such that
\begin{equation}
\mathbb{E}\left\{ \left.g\left(\cdot,X_{n}^{\varepsilon}\right)\right|\mathscr{Y}\right\} \left(\omega\right)\le\max\left\{ \vartheta\left(\omega\right),-n\right\} +\varepsilon,\quad{\cal P}-a.e..
\end{equation}
We are done.\hfill{}\ensuremath{\bigstar}

\medskip{}

\noindent \textbf{Step 8.} \textit{It is true that}
\begin{equation}
\inf_{X\in{\cal F}_{{\cal X}\left(Z_{Y}\right)}^{\mathscr{Y}}}\mathbb{E}\left\{ h\left(Y,X\right)\right\} \le\mathbb{E}\left\{ \inf_{\boldsymbol{x}\in{\cal X}\left(Z_{Y}\right)}h\left(Y,\boldsymbol{x}\right)\right\} .
\end{equation}
\rule[0.5ex]{0.2\columnwidth}{0.5pt}

Define the sequence of random variables $\left\{ \varpi_{n}:\Omega\rightarrow\overline{\mathbb{R}}\right\} _{n\in\mathbb{N}}$
as (see the RHS of (\ref{eq:EPSILON_CORE}))
\begin{equation}
\varpi_{n}\left(\omega\right)\triangleq\max\left\{ \vartheta\left(\omega\right),-n\right\} ,\quad\forall\left(\omega,n\right)\in\Omega\times\mathbb{N}.
\end{equation}
Also, recall that $\mathbb{E}\left\{ \vartheta\right\} <+\infty$.
Additionally, observe that
\begin{equation}
\varpi_{n}\left(\omega\right)\le\max\left\{ \vartheta\left(\omega\right),0\right\} \ge0,\quad\forall\left(\omega,n\right)\in\Omega\times\mathbb{N},
\end{equation}
where it is easy to show that $\mathbb{E}\left\{ \max\left\{ \vartheta,0\right\} \right\} <+\infty$.
Thus, all members of $\left\{ \varpi_{n}\right\} _{n\in\mathbb{N}}$
are bounded by an integrable random variable, everywhere in $\omega$
and uniformly in $n$, whereas it is trivial that, for every $\omega\in\Omega$,
$\varpi_{n}\left(\omega\right)\underset{n\rightarrow\infty}{\searrow}\vartheta\left(\omega\right).$

Consider now the result of \textbf{Step 7}, where we showed that,
for every $\varepsilon>0$ and for every $n\in\mathbb{N}$, there
exists a selector $X_{n}^{\varepsilon}\in{\cal F}_{{\cal X}\left(Z_{Y}\right)}^{\mathscr{Y}}$,
such that
\begin{equation}
h\left(Y,X_{n}^{\varepsilon}\right)\le\varpi_{n}+\varepsilon,\quad{\cal P}-a.e..
\end{equation}
We can then take expectations on both sides (note that all involved
integrals exist), to obtain
\begin{equation}
\mathbb{E}\left\{ h\left(Y,X_{n}^{\varepsilon}\right)\right\} \le\mathbb{E}\left\{ \varpi_{n}\right\} +\varepsilon.
\end{equation}
Since $X_{n}^{\varepsilon}\in{\cal F}_{{\cal X}\left(Z_{Y}\right)}^{\mathscr{Y}}$,
it also follows that
\begin{equation}
\inf_{X\in{\cal F}_{{\cal X}\left(Z_{Y}\right)}^{\mathscr{Y}}}\mathbb{E}\left\{ h\left(Y,X\right)\right\} \le\mathbb{E}\left\{ \varpi_{n}\right\} +\varepsilon.\label{eq:LAST?}
\end{equation}
It is also easy to see that $\varpi_{n}$ fulfills the requirements
of the Extended Monotone Convergence Theorem (\cite{Ash2000Probability},
Theorem 1.6.7 (b)). Therefore, we may pass to the limit on both sides
of (\ref{eq:LAST?}) as $n\rightarrow\infty$, yielding
\begin{equation}
\inf_{X\in{\cal F}_{{\cal X}\left(Z_{Y}\right)}^{\mathscr{Y}}}\mathbb{E}\left\{ h\left(Y,X\right)\right\} \le\mathbb{E}\left\{ \vartheta\right\} +\varepsilon.
\end{equation}
But $\varepsilon>0$ is arbitrary.\hfill{}\ensuremath{\bigstar}

Finally, just combine the statements of \textbf{Step 4} and \textbf{Step
8}, and the result follows, completing the proof of Lemma \ref{lem:FUND_Lemma}.
\end{proof}

\begin{rem}
Obviously, Lemma \ref{lem:FUND_Lemma} holds also for maximization
problems as well, by defining $g\equiv-f$, for some random function
$f:\Omega\times\mathbb{R}^{N}\rightarrow\mathbb{R}$, under the corresponding
setting and assumptions. Note that, in this case, we have to assume
that $\sup_{X\in{\cal F}_{{\cal X}\left(Z_{Y}\right)}^{\mathscr{Y}}}\mathbb{E}\left\{ f\left(\cdot,X\right)\right\} >-\infty$.\hfill{}\ensuremath{\blacksquare}
\end{rem}
\begin{rem}
Lemma \ref{lem:FUND_Lemma} may be considered a useful variation of
Theorem 14.60 in \cite{Rockafellar2009VarAn}, in the following sense.
First, it is specialized for conditional expectations of random functions,
which are additionally $\boldsymbol{SP}\diamondsuit{\cal F}_{{\cal X}\left(Z_{Y}\right)}^{\mathscr{Y}}$,
in the context of stochastic control. The latter property allows these
conditional expectations to be expressed as (Borel) random functions
themselves. This is in contrast to (\cite{Rockafellar2009VarAn},
Theorem 14.60), where it is assumed that the random function, whose
role is played by the respective conditional expectation in Lemma
\ref{lem:FUND_Lemma}, is somehow provided apriori. Second, Lemma
\ref{lem:FUND_Lemma} extends (\cite{Rockafellar2009VarAn}, Theorem
14.60), in the sense that the decision set ${\cal F}_{{\cal X}\left(Z_{Y}\right)}^{\mathscr{Y}}$
confines any solution to the respective optimization problem to be
a $\mathscr{Y}$-measurable selection of a closed-valued measurable
multifunction, while at the same time, apart from natural (and important)
measurability requirements, no continuity assumptions are imposed
on the structure of the random function induced by the respective
conditional expectation; only the validity of the substitution property
is required. In (\cite{Rockafellar2009VarAn}, Theorem 14.60), on
the other hand, it is respectively assumed that the involved random
function is a normal integrand, or, in other words, that it is random
lower semicontinuous.\hfill{}\ensuremath{\blacksquare}
\end{rem}
\begin{rem}
In Lemma \ref{lem:FUND_Lemma}, variational optimization is performed
over some subset of functions\textit{ measurable relative to $\mathscr{Y}\equiv\sigma\left\{ Y\right\} $},
where $Y$ is some given random element. Although we do not pursue
such an approach here, it would most probably be possible to develop
a more general version of Lemma \ref{lem:FUND_Lemma}, where the decision
set would be appropriately extended to include $\overline{\mathscr{Y}}$-measurable
random elements, as well. In such case, the definition of the substitution
property could be extended under the framework of lower semianalytic
functions and universal measurability, and would allow the development
of arguments showing existence of everywhere $\varepsilon$-optimal
and potentially everywhere optimal policies (decisions), in the spirit
of \cite{Bertsekas1978Stochastic,Bertsekas_Vol_2}.\hfill{}\ensuremath{\blacksquare}
\end{rem}

\subsubsection{\label{subsec:Partial_MIN}Guaranteeing the Existence of Measurable
Optimal Controls}

Although Lemma \ref{lem:FUND_Lemma} constitutes a very useful result,
which enables the simplification of a stochastic variational problem,
by essentially replacing it by an at least structurally simpler, pointwise
optimization problem, it does not provide insight on the existence
of a common optimal solution, \textit{within the respective decision
sets}.

On the one hand, it is easy to observe that, similarly to (\cite{Rockafellar2009VarAn},
Theorem 14.60), if there exists an optimal selection $X^{*}\in{\cal F}_{{\cal X}\left(Z_{Y}\right)}^{\mathscr{Y}}$,
such that
\begin{equation}
X^{*}\in\underset{\boldsymbol{x}\in{\cal X}\left(Z_{Y}\right)}{\mathrm{arg\,min}}\,\mathbb{E}\left\{ \left.g\left(\cdot,\boldsymbol{x}\right)\right|\mathscr{Y}\right\} \neq\varnothing,\quad{\cal P}-a.e.,\label{eq:OPTIMAL_POINT}
\end{equation}
and Lemma \ref{lem:FUND_Lemma} applies, then, exploiting the fact
that $g$ is $\boldsymbol{SP}\diamondsuit{\cal F}_{{\cal X}\left(Z_{Y}\right)}^{\mathscr{Y}}$,
we may write
\begin{flalign}
\inf_{X\in{\cal F}_{{\cal X}\left(Z_{Y}\right)}^{\mathscr{Y}}}\mathbb{E}\left\{ g\left(\cdot,X\right)\right\}  & \equiv\mathbb{E}\left\{ \vartheta\right\} \equiv\mathbb{E}\left\{ \xi\left(Y\right)\right\} \nonumber \\
 & =\mathbb{E}\left\{ h\left(Y,X^{*}\right)\right\} \nonumber \\
 & =\mathbb{E}\left\{ \mathbb{E}\left\{ \left.g\left(\cdot,X^{*}\right)\right|\mathscr{Y}\right\} \right\} \nonumber \\
 & \equiv\mathbb{E}\left\{ g\left(\cdot,X^{*}\right)\right\} ,
\end{flalign}
implying that the infimum of $\mathbb{E}\left\{ g\left(\cdot,X\right)\right\} $
over $X\in{\cal F}_{{\cal X}\left(Z_{Y}\right)}^{\mathscr{Y}}$ is
attained by $X^{*}$; therefore, $X^{*}$ is also an optimal solution
to the respective variational problem. \textit{Conversely}, if $X^{*}$
attains the infimum of $\mathbb{E}\left\{ g\left(\cdot,X\right)\right\} $
over $X\in{\cal F}_{{\cal X}\left(Z_{Y}\right)}^{\mathscr{Y}}$ \textit{and}
the infimum is greater than $-\infty$, then both $\mathbb{E}\left\{ g\left(\cdot,X^{*}\right)\right\} $
and $\mathbb{E}\left\{ \vartheta\right\} $ are finite, which also
implies that $\mathbb{E}\left\{ \left.g\left(\cdot,X^{*}\right)\right|\mathscr{Y}\right\} $
and $\vartheta$ are finite ${\cal P}-a.e.$. As a result, and recalling
\textbf{Step 3} in the proof of Lemma \ref{lem:FUND_Lemma}, we have
\begin{flalign}
\mathbb{E}\left\{ \mathbb{E}\left\{ \left.g\left(\cdot,X^{*}\right)\right|\mathscr{Y}\right\} -\vartheta\right\}  & \equiv0\quad\text{and}\\
\mathbb{E}\left\{ \left.g\left(\cdot,X^{*}\right)\right|\mathscr{Y}\right\} -\vartheta & \ge0,\quad{\cal P}-a.e..
\end{flalign}
and, consequently, $\vartheta\equiv\mathbb{E}\left\{ \left.g\left(\cdot,X^{*}\right)\right|\mathscr{Y}\right\} $,
${\cal P}-a.e.$.

Unfortunately, it is not possible to guarantee existence of such an
$X^{*}\in{\cal F}_{{\cal X}\left(Z_{Y}\right)}^{\mathscr{Y}}$, in
general. However, at least for the purposes of this paper, it is both
reasonable and desirable to demand the existence of an optimal solution
$X^{*}$, satisfying (\ref{eq:OPTIMAL_POINT}) (in our spatially controlled
beamforming problem, we \textit{need} to make a \textit{feasible}
decision on the position of the relays at the next time slot). Additionally,
such an optimal solution, if it exists, will not be available in closed
form, and, consequently, it will be impossible to verify measurability
directly. Therefore, we have to be able to show both existence and
measurability of $X^{*}$ \textit{indirectly}, and specifically, by
imposing constraints on the structure of the stochastic optimization
problem under consideration. One way to do this, \textit{emphasizing
on our spatially controlled beamforming problem formulation}, is to
restrict our attention to pointwise optimization problems involving
Carath\'eodory objectives, over closed-valued multifunctions, which
are additionally \textit{closed} -see Definition \ref{def:CLODES_mult}.

Focusing on Carath\'eodory functions is not particularly restrictive,
since it is already clear that, in order to guarantee the validity
of the substitution rule (the $\boldsymbol{SP}$ Property), similar
continuity assumptions would have to be imposed on both random functions
$g$ and $\mathbb{E}\left\{ \left.g\left(\cdot,\cdot\right)\right|\mathscr{Y}\right\} \equiv h$,
as Theorem \ref{thm:REP_EXP} suggests. At the same time, restricting
our attention to optimizing Carath\'eodory functions over measurable
multifunctions, measurability of optimal values and optimal decisions
is preserved, as the next theorem suggests.
\begin{thm}
\textbf{\textup{(Measurability under Partial Minimization)}} \label{thm:Measurability-Preservation}
On the base subspace $\left(\Omega,\mathscr{Y},\left.{\cal P}\right|_{\mathscr{Y}}\right)$,
where $\mathscr{Y}\subseteq\mathscr{F}$, let the random function
$H:\Omega\times\mathbb{R}^{N}\rightarrow\overline{\mathbb{R}}$ be
Carath\'eodory, and consider another random element $Z:\Omega\rightarrow\mathbb{R}^{N}$,
as well as any compact-valued multifunction ${\cal X}:\mathbb{R}^{N}\rightrightarrows\mathbb{R}^{N}$,
with $\mathrm{dom}\left({\cal X}\right)\equiv\mathbb{R}^{N}$, which
is also \textbf{closed}. Additionally, define $H^{*}:\Omega\rightarrow\overline{\mathbb{R}}$
as the optimal value to the optimization problem
\begin{equation}
\begin{array}{rl}
\underset{\boldsymbol{x}}{\mathrm{minimize}} & H\left(\omega,\boldsymbol{x}\right)\\
\mathrm{subject\,to} & \boldsymbol{x}\in{\cal X}\left(Z\left(\omega\right)\right)
\end{array},\quad\forall\omega\in\Omega.\label{eq:PROG_1}
\end{equation}
Then, $H^{*}$ is $\mathscr{Y}$-measurable and attained for at least
one $\mathscr{Y}$-measurable minimizer $X^{*}:\Omega\rightarrow\mathbb{R}^{N}$.
If the minimizer $X^{*}$ is unique, then it has to be $\mathscr{Y}$-measurable.
\end{thm}
\begin{proof}[Proof of Theorem \ref{thm:Measurability-Preservation}]
From (\cite{Shapiro2009STOCH_PROG}, pp. 365 - 367 and/or \cite{Rockafellar2009VarAn},
Example 14.32 \& Theorem 14.37), we may immediately deduce that $H^{*}$
is $\mathscr{Y}$-measurable and attained for at least one $\mathscr{Y}$-measurable
minimizer $X^{*}$, as long as the compact (therefore closed, as well)-valued
multifunction ${\cal X}\left(Z\left(\cdot\right)\right):\Omega\rightrightarrows\mathbb{R}^{N}$
is measurable relative to $\mathscr{Y}$. In order to show that the
composition ${\cal X}\left(Z\left(\cdot\right)\right)$ is $\mathscr{Y}$-measurable,
we use the assumption that the compact-valued multifunction ${\cal X}:\mathbb{R}^{N}\rightrightarrows\mathbb{R}^{N}$
is closed and, therefore, Borel measurable (Remark 28 in \cite{Shapiro2009STOCH_PROG},
p. 365). Then, $\mathscr{Y}$-measurability of ${\cal X}\left(Z\left(\cdot\right)\right)$
follows by the same arguments as in \textbf{Step 1}, in the proof
of Lemma \ref{lem:FUND_Lemma}.
\end{proof}
\begin{rem}
It would be important to mention that if one replaces $\mathbb{R}^{N}$
with any compact (say) subset ${\cal H}\subset\mathbb{R}^{N}$ in
the statement of Theorem \ref{thm:Measurability-Preservation}, then
the result continues to hold as is. No modification is necessary.
In our spatially controlled beamforming problem, this compact set
${\cal H}$ is specifically identified either with the hypercubic
region ${\cal S}^{R}$, or with some compact subset of it.\hfill{}\ensuremath{\blacksquare}
\end{rem}

\subsubsection{\label{subsec:FINAL}Fusion \& Derivation of Conditions C1-C6}

Finally, combining Theorem \ref{thm:REP_EXP}, Lemma \ref{lem:FUND_Lemma}
and Theorem \ref{thm:Measurability-Preservation}, we may directly
formulate the following constrained version of the Fundamental Lemma,
which is of central importance regarding the special class of stochastic
problems considered in this work and, in particular, (\ref{eq:2STAGE-1}).
\begin{lem}
\textbf{\textup{(Fundamental Lemma / Fused Version)\label{lem:FUND_Lemma_FINAL}}}
On $\left(\Omega,\mathscr{F},{\cal P}\right)$, consider a random
element $Y:\Omega\rightarrow\mathbb{R}^{M}$, the sub $\sigma$-algebra
$\mathscr{Y}\triangleq\sigma\left\{ Y\right\} \subseteq\mathscr{F}$,
a random function $g:\Omega\times\mathbb{R}^{N}\rightarrow\mathbb{R}$,
such that $\mathbb{E}\left\{ g\left(\cdot,\boldsymbol{x}\right)\right\} $
exists for all $\boldsymbol{x}\in\mathbb{R}^{N}$, a multifunction
${\cal X}:\mathbb{R}^{N}\rightrightarrows\mathbb{R}^{N}$, with $\mathrm{dom}\left({\cal X}\right)\equiv\mathbb{R}^{N}$
and , as well as another function $Z_{Y}:\Omega\rightarrow\mathbb{R}^{N}$.
Assume that:
\begin{description}
\item [{C1.}] ${\cal X}$ is compact-valued and closed, and that
\item [{C2.}] $Z_{Y}$ is a $\mathscr{Y}$-measurable random element.
\end{description}
Consider also the nonempty decision set ${\cal F}_{{\cal X}\left(Z_{Y}\right)}^{\mathscr{Y}}$.
Additionally, suppose that:
\begin{description}
\item [{C3.}] $\mathbb{E}\left\{ g\left(\cdot,X\right)\right\} $ exists
for all $X\in{\cal F}_{{\cal X}\left(Z_{Y}\right)}^{\mathscr{Y}}$,
with $\inf_{X\in{\cal F}_{{\cal X}\left(Z_{Y}\right)}^{\mathscr{Y}}}\mathbb{E}\left\{ g\left(\cdot,X\right)\right\} <+\infty$,
\item [{C4.}] $g$ is dominated by a ${\cal P}$-integrable function, uniformly
in $\boldsymbol{x}\in\mathbb{R}^{N}$,
\item [{C5.}] $g$ is Carath\'eodory on $\Omega\times\mathbb{R}^{N}$,
and that
\item [{C6.}] $\mathbb{E}\left\{ \left.g\left(\cdot,\boldsymbol{x}\right)\right|\mathscr{Y}\right\} \equiv h\left(Y,\boldsymbol{x}\right)$
is Carath\'eodory on $\Omega\times\mathbb{R}^{N}$.
\end{description}
\noindent Then, the optimal value function $\inf_{\boldsymbol{x}\in{\cal X}\left(Z_{Y}\right)}\mathbb{E}\left\{ \left.g\left(\cdot,\boldsymbol{x}\right)\right|\mathscr{Y}\right\} \triangleq\vartheta$
is $\mathscr{Y}$-measurable, and it is true that
\begin{equation}
\inf_{X\in{\cal F}_{{\cal X}\left(Z_{Y}\right)}^{\mathscr{Y}}}\mathbb{E}\left\{ g\left(\cdot,X\right)\right\} \equiv\mathbb{E}\left\{ \vartheta\right\} \equiv\mathbb{E}\left\{ g\left(\cdot,X^{*}\right)\right\} ,\label{eq:Lemma-FUSED}
\end{equation}
for at least one $X^{*}\in{\cal F}_{{\cal X}\left(Z_{Y}\right)}^{\mathscr{Y}}$
, such that $X^{*}\left(\omega\right)\in\mathrm{arg\,min}_{\boldsymbol{x}\in{\cal X}\left(Z_{Y}\right)}\,\mathbb{E}\left\{ \left.g\left(\cdot,\boldsymbol{x}\right)\right|\mathscr{Y}\right\} \left(\omega\right)$,
everywhere in $\omega\in\Omega$. If there is only one minimizer attaining
$\vartheta$, then it has to be $\mathscr{Y}$-measurable.
\end{lem}
\begin{proof}[Proof of Lemma \ref{lem:FUND_Lemma_FINAL}]
We just carefully combine Theorem \ref{thm:REP_EXP}, Lemma \ref{lem:FUND_Lemma}
and Theorem \ref{thm:Measurability-Preservation}. First, if conditions
\textbf{C4-C6} are satisfied, then, from Theorem \ref{thm:REP_EXP},
it follows that $g$ is $\boldsymbol{SP}\diamondsuit\mathfrak{I}_{\mathscr{Y}}$.
Then, since ${\cal F}_{{\cal X}\left(Z_{Y}\right)}^{\mathscr{Y}}\subseteq\mathfrak{I}_{\mathscr{Y}}$,
$g$ is $\boldsymbol{SP}\diamondsuit{\cal F}_{{\cal X}\left(Z_{Y}\right)}^{\mathscr{Y}}$,
as well. Consequently, with \textbf{C1-C3 }being true, all assumptions
of Lemma \ref{lem:FUND_Lemma} are satisfied, and the first equivalence
of (\ref{eq:Lemma-FUSED}) from the left is true. Additionally, from
Theorem \ref{thm:Measurability-Preservation}, it easily follows that
the optimal value $\vartheta$ is $\mathscr{Y}$-measurable, attained
by an at least one $\mathscr{Y}$-measurable $X^{*}$, which, of course,
constitutes a selection of ${\cal X}\left({\cal Z}\left(Y\right)\right)\equiv{\cal X}\left(Z_{Y}\right)$,
or, equivalently, $X^{*}\in{\cal F}_{{\cal X}\left(Z_{Y}\right)}^{\mathscr{Y}}$.
Then, because $g$ is $\boldsymbol{SP}\diamondsuit{\cal F}_{{\cal X}\left(Z_{Y}\right)}^{\mathscr{Y}},$
we may write
\begin{flalign}
\vartheta & \equiv\inf_{\boldsymbol{x}\in{\cal X}\left(Z_{Y}\right)}\mathbb{E}\left\{ \left.g\left(\cdot,\boldsymbol{x}\right)\right|\mathscr{Y}\right\} \nonumber \\
 & \equiv\inf_{\boldsymbol{x}\in{\cal X}\left(Z_{Y}\right)}h\left(Y,\boldsymbol{x}\right)\nonumber \\
 & \equiv h\left(Y,X^{*}\right)\nonumber \\
 & \equiv\mathbb{E}\left\{ \left.g\left(\cdot,X^{*}\right)\right|\mathscr{Y}\right\} ,\quad{\cal P}-a.e.,
\end{flalign}
which yields the equivalence $\mathbb{E}\left\{ \vartheta\right\} \equiv\mathbb{E}\left\{ g\left(\cdot,X^{*}\right)\right\} $.
The proof is complete.
\end{proof}
\begin{rem}
Note that, because, in Lemma \ref{lem:FUND_Lemma_FINAL}, $X^{*}\left(\omega\right)\in{\cal X}\left(Z_{Y}\left(\omega\right)\right)$,
\textit{everywhere} in $\omega\in\Omega$, it is true that $X^{*}$
is actually a minimizer of the slightly more constrained problem of
infimizing $\mathbb{E}\left\{ g\left(\cdot,X\right)\right\} $ over
the set of \textit{precisely} all $\mathscr{Y}$-measurable selections
of ${\cal X}\left(Z_{Y}\right)$. Denoting this decision set as ${\cal F}_{{\cal X}\left(Z_{Y}\right)}^{\mathscr{Y},E}\subseteq{\cal F}_{{\cal X}\left(Z_{Y}\right)}^{\mathscr{Y}}$,
the aforementioned statement is true since, simply,
\begin{flalign}
\inf_{X\in{\cal F}_{{\cal X}\left(Z_{Y}\right)}^{\mathscr{Y},E}}\mathbb{E}\left\{ g\left(\cdot,X\right)\right\}  & \ge\inf_{X\in{\cal F}_{{\cal X}\left(Z_{Y}\right)}^{\mathscr{Y}}}\mathbb{E}\left\{ g\left(\cdot,X\right)\right\} \equiv\mathbb{E}\left\{ g\left(\cdot,X^{*}\right)\right\} \\
\implies\inf_{X\in{\cal F}_{{\cal X}\left(Z_{Y}\right)}^{\mathscr{Y},E}}\mathbb{E}\left\{ g\left(\cdot,X\right)\right\}  & \equiv\mathbb{E}\left\{ g\left(\cdot,X^{*}\right)\right\} .
\end{flalign}
where we have used the fact that $X^{*}\in{\cal F}_{{\cal X}\left(Z_{Y}\right)}^{\mathscr{Y},E}$.
This type of decision set is considered, for simplicity, in (\ref{eq:2STAGE-1}),
which corresponds to the original formulation of the spatially controlled
beamforming problem.\hfill{}\ensuremath{\blacksquare}
\end{rem}
Lemma \ref{lem:FUND_Lemma_FINAL} is of major importance, as it directly
provides us with conditions \textbf{C1-C6}, which, being relatively
easily verifiable, at least for our spatially controlled beamforming
setting, ensure strict theoretical consistency of the methods developed
in this paper. At this point, our discussion concerning the Fundamental
Lemma has been concluded.\hfill{}\ensuremath{\blacksquare}

\subsection{Appendix C: Proofs / Section \ref{sec:MobRelBeam}}

\subsubsection{Proof of Theorem \ref{lem:C1C4_SAT}}


Since, in the following, we are going to verify conditions \textbf{C1-C6}
of Lemma \ref{lem:FUND_Lemma_FINAL} in Section \ref{subsec:FINAL}
(Appendix B) for the $2$-stage problem (\ref{eq:2STAGE-SINR-2}),
it will be useful to first match it to the setting of Lemma \ref{lem:FUND_Lemma_FINAL},
term-by-term. Table \ref{tab:Variable-matching} shows how the components
of (\ref{eq:2STAGE-SINR-2}) are matched to the respective components
of the optimization problem considered in Lemma \ref{lem:FUND_Lemma_FINAL}.
For the rest of the proof, we consider this variable matching automatic.

Keep $t\in\mathbb{N}_{N_{T}}^{2}$ \textit{fixed}. As in the statement
of Theorem \ref{lem:C1C4_SAT}, suppose that, at time slot $t-1\in\mathbb{N}_{N_{T}-1}^{+}$,
${\bf p}^{o}\left(t-1\right)\equiv{\bf p}^{o}\left(\omega,t-1\right)$
is measurable relative to $\mathscr{C}\left({\cal T}_{t-2}\right)$.
Then, condition \textbf{C2 }is automatically verified.\hfill{}\ensuremath{\bigstar}

Next, let us verify \textbf{C1}. For this, we will simply show directly
that closed-valued translated multifunctions, in the sense of Definition
\ref{def:(Translated-Multifunctions)}, are also closed. Given two
\textit{closed} sets ${\cal H}\subset\mathbb{R}^{N}$, ${\cal A}\subseteq\mathbb{R}^{N}$
and a fixed reference $\boldsymbol{h}\in{\cal H}$, let ${\cal D}:\mathbb{R}^{N}\rightrightarrows\mathbb{R}^{N}$
be $\left({\cal H},\boldsymbol{h}\right)$-translated in ${\cal A}$
and consider any two arbitrary sequences 
\begin{equation}
\left\{ \boldsymbol{x}_{k}\in{\cal A}\right\} _{k\in\mathbb{N}}\quad\text{and}\quad\left\{ \boldsymbol{y}_{k}\in{\cal A}-\boldsymbol{h}\right\} _{k\in\mathbb{N}},
\end{equation}
such that $\boldsymbol{x}_{k}\underset{k\rightarrow\infty}{\longrightarrow}\boldsymbol{x}$,
$\boldsymbol{y}_{k}\underset{k\rightarrow\infty}{\longrightarrow}\boldsymbol{y}$
and $\boldsymbol{x}_{k}\in{\cal D}\left(\boldsymbol{y}_{k}\right)$,
for all $k\in\mathbb{N}$. By Definition \ref{def:(Translated-Multifunctions)},
$\boldsymbol{x}_{k}\in{\cal D}\left(\boldsymbol{y}_{k}\right)$ if
and only if $\boldsymbol{x}_{k}-\boldsymbol{y}_{k}\in{\cal H}$, for
all $k\in\mathbb{N}$. But $\boldsymbol{x}_{k}-\boldsymbol{y}_{k}\underset{k\rightarrow\infty}{\longrightarrow}\boldsymbol{x}-\boldsymbol{y}$
and ${\cal H}$ is closed. Therefore, it is true that $\boldsymbol{x}-\boldsymbol{y}\in{\cal H}$,
as well, showing that ${\cal D}$ is closed. By Assumption \ref{assu:AS_(TransMultifunctions)},
${\cal C}:\mathbb{R}^{2R}\rightrightarrows\mathbb{R}^{2R}$ is the
$\left({\cal G},{\bf 0}\right)$-translated multifunction in ${\cal S}^{R}$,
for some compact and, hence, closed, ${\cal G}\subset{\cal S}^{R}$.
Consequently, the restriction of ${\cal C}$ in ${\cal S}^{R}$ is
closed and \textbf{C3} is verified.\hfill{}\ensuremath{\bigstar}

Condition \textbf{C5 }is also easily verified; it suffices to show
that both functions $\left|f\left(\cdot,\cdot,t\right)\right|^{2}$
and $\left|g\left(\cdot,\cdot,t\right)\right|^{2}$ are Carath\'eodory
on $\Omega\times{\cal S}$, or, in other words, that the fields $\left|f\left({\bf p},t\right)\right|^{2}$
and $\left|g\left({\bf p},t\right)\right|^{2}$ are everywhere sample
path continuous. Indeed, if this holds, $V_{I}\left(\cdot,\cdot,t\right)$
will be Carath\'eodory, as a continuous functional of $\left|f\left(\cdot,\cdot,t\right)\right|^{2}$
and $\left|g\left(\cdot,\cdot,t\right)\right|^{2}$, and since
\begin{equation}
V\left(\left[{\bf p}_{1}^{\boldsymbol{T}}\,\ldots\,{\bf p}_{R}^{\boldsymbol{T}}\right]^{\boldsymbol{T}},t\right)\equiv\sum_{i\in\mathbb{N}_{R}^{+}}V_{I}\left({\bf p}_{i},t\right),
\end{equation}
it readily follows that $V\left(\cdot,\cdot,t\right)$ is Carath\'eodory
on $\Omega\times{\cal S}^{R}$. In order to show (everywhere) sample
path continuity of $\left|f\left({\bf p},t\right)\right|^{2}$ (respectively
$\left|g\left({\bf p},t\right)\right|^{2}$) on ${\cal S}$, we may
utilize (\ref{eq:CONVERTER}). As a result, sample path continuity
of $\left|f\left({\bf p},t\right)\right|^{2}$ is equivalent to sample
path continuity of
\begin{equation}
F\left({\bf p},t\right)\equiv\alpha_{S}\left({\bf p}\right)\ell+\sigma_{S}\left({\bf p},t\right)+\xi_{S}\left({\bf p},t\right),\quad\forall{\bf p}\in{\cal S}.
\end{equation}
Of course, $\alpha_{S}$ is a continuous function of ${\bf p}$. As
long as the fields $\sigma_{S}\left({\bf p},t\right)$ and $\xi_{S}\left({\bf p},t\right)$
are concerned, these are also sample path continuous; see Section
\ref{subsec:TECHNICAL_1}. Enough said.
\begin{table}
\centering%
\begin{tabular}{|>{\centering}m{0.25\paperwidth}|>{\centering}m{0.35\paperwidth}|}
\hline 
Problem of Lemma \ref{lem:FUND_Lemma_FINAL} & $2$-Stage Problem (\ref{eq:2STAGE-SINR-2}) \tabularnewline
\hline 
\hline 
Random element $Y:\Omega\rightarrow\mathbb{R}^{M}$ & \textit{All} relay positions \textit{and} channel observations, \\
up to (current) time slot $t-1$ \tabularnewline
\hline 
$\sigma$-Algebra $\mathscr{Y}\triangleq\sigma\left\{ Y\right\} $ & $\sigma$-Algebra $\mathscr{C}\left({\cal T}_{t-1}\right)$, jointly
generated\\
by the above random vector\tabularnewline
\hline 
Random Function $g:\Omega\times\mathbb{R}^{N}\rightarrow\mathbb{R}$ & Optimal value of the second-stage problem,\\
 $V\left(\cdot,\cdot,t-1\right):\Omega\times{\cal S}^{R}\rightarrow\mathbb{R}_{++}$\tabularnewline
\hline 
Multifunction ${\cal X}:\mathbb{R}^{N}\rightrightarrows\mathbb{R}^{N}$, 

with $\mathrm{dom}\left({\cal X}\right)\equiv\mathbb{R}^{N}$ & Spatially feasible motion region\\
${\cal C}:{\cal S}^{R}\rightrightarrows{\cal S}^{R}$, with $\mathrm{dom}\left({\cal C}\right)\equiv{\cal S}^{R}$\tabularnewline
\hline 
Function $Z_{Y}:\Omega\rightarrow\mathbb{R}^{N}$ & \textit{Selected} motion policy at time slot $t-2$,\\
 ${\bf p}^{o}\left(\cdot,t-1\right):\Omega\rightarrow{\cal S}^{R}$\tabularnewline
\hline 
Decision set ${\cal F}_{{\cal X}\left(Z_{Y}\right)}^{\mathscr{Y}}$ & Decision set ${\cal D}_{t}$\\
(precisely matched with ${\cal F}_{{\cal X}\left(Z_{Y}\right)}^{\mathscr{Y},E}$)\tabularnewline
\hline 
\end{tabular}

\caption{\label{tab:Variable-matching}Variable matching for (\ref{eq:2STAGE-SINR-2})
and the respective problem considered in Lemma \ref{lem:FUND_Lemma_FINAL}.}
\end{table}
\hfill{}\ensuremath{\bigstar}

We continue with \textbf{C3}. Since we already know that $V\left(\cdot,\cdot,t\right)$
is Carath\'eodory, it follows from (\cite{Aliprantis2006_Inf}, Lemma
4.51) that $V\left(\cdot,\cdot,t\right)$ is also jointly measurable
relative to $\mathscr{F}\otimes\mathscr{B}\left({\cal S}^{R}\right)$.
Next, let ${\bf p}\left(t\right)\equiv{\bf p}\left(\omega,t\right)\in{\cal S}^{R}$
be \textit{any} random element, measurable with respect to $\mathscr{C}\left({\cal T}_{t-1}\right)$
and, thus, $\mathscr{F}$, too. Then, from (\cite{Aliprantis2006_Inf},
Lemma 4.49), we know that the pair $\left({\bf p}\left(t,\omega\right),\left(t,\omega\right)\right)$
is also $\mathscr{F}$-measurable. Consequently, $\left|V\left(\cdot,{\bf p}\left(\cdot,t\right),t\right)\right|^{2}$
must be $\mathscr{F}$-measurable, as a composition of measurable
functions. Additionally, $V\left(\cdot,\cdot,t\right)$ is, by definition,
nonnegative. Thus, its expectation exists (Corollary 1.6.4 in \cite{Ash2000Probability}),
and we are done.\hfill{}\ensuremath{\bigstar}

Conditions \textbf{C4 }and \textbf{C6} need slightly more work, in
order to be established. To verify \textbf{C4}, we have to show existence
of a function in ${\cal L}_{1}\left(\Omega,\mathscr{F},{\cal P};\mathbb{R}\right)$,
which dominates $V\left(\cdot,\cdot,t\right)$, \textit{uniformly}
in ${\bf p}\in{\cal S}^{R}$. Everywhere in $\Omega$, again using
(\ref{eq:CONVERTER}), and with $\varsigma\triangleq\log\left(10\right)/10$
for brevity, we may write
\begin{flalign}
V\left(\left[{\bf p}_{1}^{\boldsymbol{T}}\,\ldots\,{\bf p}_{R}^{\boldsymbol{T}}\right]^{\boldsymbol{T}},t\right) & \equiv\sum_{i\in\mathbb{N}_{R}^{+}}\dfrac{P_{c}P_{0}\left|f\left({\bf p}_{i},t\right)\right|^{2}\left|g\left({\bf p}_{i},t\right)\right|^{2}}{P_{0}\sigma_{D}^{2}\left|f\left({\bf p}_{i},t\right)\right|^{2}+P_{c}\sigma^{2}\left|g\left({\bf p}_{i},t\right)\right|^{2}+\sigma^{2}\sigma_{D}^{2}}\nonumber \\
 & \le\dfrac{P_{0}}{\sigma^{2}}\sum_{i\in\mathbb{N}_{R}^{+}}\left|f\left({\bf p}_{i},t\right)\right|^{2}\nonumber \\
 & \le\dfrac{P_{0}}{\sigma^{2}}\sum_{i\in\mathbb{N}_{R}^{+}}\sup_{{\bf p}_{i}\in{\cal S}}\left|f\left({\bf p}_{i},t\right)\right|^{2}\nonumber \\
 & \equiv\dfrac{10^{\rho/10}P_{0}R}{\sigma^{2}}\sup_{{\bf p}\in{\cal S}}\exp\left(\varsigma F\left({\bf p},t\right)\right)\nonumber \\
 & \equiv\dfrac{10^{\rho/10}P_{0}R}{\sigma^{2}}\exp\hspace{-2pt}\left(\hspace{-2pt}{\displaystyle \varsigma\sup_{{\bf p}\in{\cal S}}F\left({\bf p},t\right)\hspace{-2pt}}\right)\nonumber \\
 & \equiv\dfrac{10^{\rho/10}P_{0}R}{\sigma^{2}}\exp\hspace{-2pt}\left(\hspace{-2pt}{\displaystyle \varsigma\sup_{{\bf p}\in{\cal S}}\alpha_{S}\left({\bf p}\right)\ell+\sigma_{S}\left({\bf p},t\right)+\xi_{S}\left({\bf p},t\right)\hspace{-2pt}}\right)\nonumber \\
 & \triangleq\dfrac{10^{\rho/10}P_{0}R}{\sigma^{2}}\exp\hspace{-2pt}\left({\displaystyle \hspace{-2pt}\varsigma\sup_{{\bf p}\in{\cal S}}\alpha_{S}\left({\bf p}\right)\ell+\chi_{S}\left({\bf p},t\right)\hspace{-2pt}}\right)\nonumber \\
 & \le\dfrac{10^{\rho/10}P_{0}R}{\sigma^{2}}\exp\hspace{-2pt}\left(\hspace{-2pt}{\displaystyle \varsigma\ell\sup_{{\bf p}\in{\cal S}}\alpha_{S}\left({\bf p}\right)}\hspace{-2pt}\right)\exp\hspace{-2pt}\left(\hspace{-2pt}{\displaystyle \varsigma\sup_{{\bf p}\in{\cal S}}\chi_{S}\left({\bf p},t\right)}\hspace{-2pt}\right)\nonumber \\
 & \triangleq\varphi\left(\omega,t\right)>0,\quad\forall\omega\in\Omega.
\end{flalign}
Due to the fact that $\alpha_{S}$ is continuous in ${\bf p}\in{\cal S}$
and that ${\cal S}$ is compact, the Extreme Value Theorem implies
that the deterministic term $\sup_{{\bf p}\in{\cal S}}\alpha_{S}\left({\bf p}\right)$
is finite. Consequently, it suffices to show that
\begin{equation}
\mathbb{E}\left\{ \exp\hspace{-2pt}\left(\hspace{-2pt}{\displaystyle \varsigma\sup_{{\bf p}\in{\cal S}}\chi_{S}\left({\bf p},t\right)\hspace{-2pt}}\right)\hspace{-2pt}\right\} <+\infty,\label{eq:Borell_1}
\end{equation}
provided, of course, that the expectation is meaningfully defined.
For this to happen, it suffices that the function $\sup_{{\bf p}\in{\cal S}}\chi_{S}\left({\bf p},t\right)$
is a well defined random variable. Since both $\sigma_{S}\left({\bf p},t\right)$
and $\xi_{S}\left({\bf p},t\right)$ are sample path continuous, it
follows that the sum field $\sigma_{S}\left({\bf p},t\right)+\xi_{S}\left({\bf p},t\right)$
is sample path continuous. It is then relatively easy to see that
$\sup_{{\bf p}\in{\cal S}}\chi_{S}\left({\bf p},t\right)$ is a measurable
function. See, for instance, Theorem \ref{thm:Measurability-Preservation},
or \cite{Adler2009_Random}. Additionally, the Extreme Value Theorem
again implies that $\sup_{{\bf p}\in{\cal S}}\chi_{S}\left({\bf p},t\right)$
is finite \textit{everywhere} on $\Omega$, which in turn means that
the field $\chi_{S}\left({\bf p},t\right)$ is \textit{at least} almost
everywhere bounded on the compact set ${\cal S}$.

Now, in order to prove that (\ref{eq:Borell_1}) is indeed true, we
will invoke a well known result from the theory of concentration of
measure, the \textit{Borell-TIS Inequality}, which now follows.
\begin{thm}
\textbf{\textup{(Borell-TIS Inequality \cite{Adler2009_Random}) \label{thm:Borell_TIS}}}Let
$X\left(\boldsymbol{s}\right)$, $\boldsymbol{s}\in\mathbb{R}^{N}$,
be a real-valued, zero-mean, Gaussian random field, ${\cal P}$-almost
everywhere bounded on a compact subset ${\cal K}\subset\mathbb{R}^{N}$.
Then, it is true that
\begin{flalign}
\mathbb{E}\left\{ \sup_{\boldsymbol{s}\in{\cal K}}X\left(\boldsymbol{s}\right)\right\}  & <+\infty\quad\text{and}\\
{\cal P}\left(\sup_{\boldsymbol{s}\in{\cal K}}X\left(\boldsymbol{s}\right)-\mathbb{E}\left\{ \sup_{\boldsymbol{s}\in{\cal K}}X\left(\boldsymbol{s}\right)\right\} >u\right) & \le\exp\left(\hspace{-2pt}-\dfrac{u^{2}}{{\displaystyle 2\sup_{\boldsymbol{s}\in{\cal K}}\mathbb{E}\left\{ X^{2}\left(\boldsymbol{s}\right)\right\} }}\right),
\end{flalign}
for all $u>0$.
\end{thm}
As highlighted in (\cite{Adler2009_Random}, page 50), an immediate
consequence of the Borell-TIS Inequality is that, under the setting
of Theorem \ref{thm:Borell_TIS}, we may further assert that
\begin{equation}
{\cal P}\left(\sup_{\boldsymbol{s}\in{\cal K}}X\left(\boldsymbol{s}\right)>u\right)\le\exp\left(-\dfrac{\left(u-\mathbb{E}\left\{ {\displaystyle \sup_{\boldsymbol{s}\in{\cal K}}X\left(\boldsymbol{s}\right)}\right\} \right)^{2}}{{\displaystyle 2\sup_{\boldsymbol{s}\in{\cal K}}\mathbb{E}\left\{ X^{2}\left(\boldsymbol{s}\right)\right\} }}\right),\label{eq:Borell_2}
\end{equation}
for all $u>\mathbb{E}\left\{ \sup_{\boldsymbol{s}\in{\cal K}}X\left(\boldsymbol{s}\right)\right\} $.

To show (\ref{eq:Borell_1}), we exploit the Borell-TIS Inequality
and follow a procedure similar to (\cite{Adler2009_Random}, Theorem
2.1.2). First, from the discussion above, we readily see that the
field $\chi_{S}\left({\bf p},t\right)$ does satisfy the assumptions
Theorem \ref{thm:Borell_TIS}. Also, because $\chi_{S}\left({\bf p},t\right)$
is the sum of two independent fields, it is true that 
\begin{equation}
\mathbb{E}\left\{ \chi_{S}^{2}\left({\bf p},t\right)\right\} \equiv\eta^{2}+\sigma_{\xi}^{2}.
\end{equation}
 As a result, Theorem \ref{thm:Borell_TIS} implies that $\mathbb{E}\left\{ \sup_{{\bf p}\in{\cal S}}\chi_{S}\left({\bf p},t\right)\right\} $
is finite and we may safely write
\begin{flalign}
\mathbb{E}\left\{ \exp\hspace{-2pt}\left(\hspace{-2pt}{\displaystyle \varsigma\sup_{{\bf p}\in{\cal S}}\chi_{S}\left({\bf p},t\right)\hspace{-2pt}}\right)\hspace{-2pt}\right\}  & \equiv{\scaleint{7ex}}_{\!\!\!\!\!0}^{\,\infty}{\cal P}\left(\hspace{-2pt}\exp\hspace{-2pt}\left(\hspace{-2pt}{\displaystyle \varsigma\sup_{{\bf p}\in{\cal S}}\chi_{S}\left({\bf p},t\right)\hspace{-2pt}}\right)>x\hspace{-2pt}\right)\hspace{-2pt}\text{d}x\nonumber \\
 & \equiv{\scaleint{7ex}}_{\!\!\!\!\!0}^{\,\infty}{\cal P}\left({\displaystyle \sup_{{\bf p}\in{\cal S}}\chi_{S}\left({\bf p},t\right)}>\dfrac{\log\left(x\right)}{\varsigma}\right)\hspace{-2pt}\text{d}x.\label{eq:Borell_3}
\end{flalign}
In order to exploit (\ref{eq:Borell_2}), it must hold that
\begin{equation}
\dfrac{\log\left(x\right)}{\varsigma}>\mathbb{E}\left\{ \sup_{{\bf p}\in{\cal S}}\chi_{S}\left({\bf p},t\right)\right\} \Leftrightarrow x>\exp\left(\varsigma\mathbb{E}\left\{ \sup_{{\bf p}\in{\cal S}}\chi_{S}\left({\bf p},t\right)\right\} \right)>0.
\end{equation}
Therefore, we may break (\ref{eq:Borell_3}) into two parts and bound
from above, namely,
\begin{flalign}
 & \hspace{-2pt}\hspace{-2pt}\hspace{-2pt}\hspace{-2pt}\hspace{-2pt}\hspace{-2pt}\hspace{-2pt}\hspace{-2pt}\hspace{-2pt}\hspace{-2pt}\hspace{-2pt}\hspace{-2pt}\mathbb{E}\left\{ \exp\hspace{-2pt}\left(\hspace{-2pt}{\displaystyle \varsigma\sup_{{\bf p}\in{\cal S}}\chi_{S}\left({\bf p},t\right)\hspace{-2pt}}\right)\hspace{-2pt}\right\} \nonumber \\
 & \equiv{\scaleint{7ex}}_{\!\!\!\!\!0}^{\,\exp\left(\varsigma\mathbb{E}\left\{ \sup_{{\bf p}\in{\cal S}}\chi_{S}\left({\bf p},t\right)\right\} \right)}{\cal P}\left({\displaystyle \sup_{{\bf p}\in{\cal S}}\chi_{S}\left({\bf p},t\right)}>\dfrac{\log\left(x\right)}{\varsigma}\right)\hspace{-2pt}\text{d}x\nonumber \\
 & \quad\quad\quad+{\scaleint{7ex}}_{\!\!\!\!\!\exp\left(\varsigma\mathbb{E}\left\{ \sup_{{\bf p}\in{\cal S}}\chi_{S}\left({\bf p},t\right)\right\} \right)}^{\,\infty}{\cal P}\left({\displaystyle \sup_{{\bf p}\in{\cal S}}\chi_{S}\left({\bf p},t\right)}>\dfrac{\log\left(x\right)}{\varsigma}\right)\hspace{-2pt}\text{d}x\nonumber \\
 & \le{\scaleint{7ex}}_{\!\!\!\!\!0}^{\,\exp\left(\varsigma\mathbb{E}\left\{ \sup_{{\bf p}\in{\cal S}}\chi_{S}\left({\bf p},t\right)\right\} \right)}\hspace{-2pt}\text{d}x\nonumber \\
 & \quad\quad\quad+{\scaleint{15.5ex}}_{\!\!\!\!\!\!\!\exp\left(\varsigma\mathbb{E}\left\{ \sup_{{\bf p}\in{\cal S}}\chi_{S}\left({\bf p},t\right)\right\} \right)}^{\,\infty}\exp\hspace{-2pt}\left(\hspace{-2pt}-\dfrac{\left(\dfrac{\log\left(x\right)}{\varsigma}-\mathbb{E}\left\{ {\displaystyle \sup_{{\bf p}\in{\cal S}}\chi_{S}\left({\bf p},t\right)}\right\} \right)^{2}}{{\displaystyle 2\left(\eta^{2}+\sigma_{\xi}^{2}\right)}}\right)\hspace{-2pt}\text{d}x\nonumber \\
 & \le\exp\hspace{-2pt}\left(\hspace{-2pt}\varsigma\mathbb{E}\left\{ \sup_{{\bf p}\in{\cal S}}\chi_{S}\left({\bf p},t\right)\hspace{-2pt}\right\} \right)\nonumber \\
 & \quad\quad\quad+\varsigma{\scaleint{15.5ex}}_{\!\!\!\!\!\!\!\mathbb{E}\left\{ \sup_{{\bf p}\in{\cal S}}\chi_{S}\left({\bf p},t\right)\right\} }^{\,\infty}\exp\hspace{-2pt}\left(\varsigma u\right)\exp\hspace{-2pt}\left(\hspace{-2pt}-\dfrac{\left(u-\mathbb{E}\left\{ {\displaystyle \sup_{{\bf p}\in{\cal S}}\chi_{S}\left({\bf p},t\right)}\right\} \right)^{2}}{{\displaystyle 2\left(\eta^{2}+\sigma_{\xi}^{2}\right)}}\right)\hspace{-2pt}\text{d}u.\label{eq:Borell_4}
\end{flalign}
Since both terms on the RHS of (\ref{eq:Borell_4}) are finite, (\ref{eq:Borell_1})
is indeed satisfied. Consequently, it is true that
\begin{equation}
\mathbb{E}\left\{ \varphi\left(\cdot,t\right)\right\} <+\infty\Leftrightarrow\varphi\left(\cdot,t\right)\in{\cal L}_{1}\left(\Omega,\mathscr{F},{\cal P};\mathbb{R}\right).
\end{equation}
Enough said; \textbf{C4 }is now verified.\hfill{}\ensuremath{\bigstar}

Moving on to \textbf{C6}, the goal here is to show that, for each
fixed $t\in\mathbb{N}_{N_{T}}^{2}$, the well defined random function
$H:\Omega\times{\cal S}^{R}\rightarrow\overline{\mathbb{R}}$, defined
as
\begin{equation}
H\left(\omega,{\bf p}\right)\triangleq\mathbb{E}\left\{ V\left({\bf p},t\right)\left|\mathscr{C}\left({\cal T}_{t-1}\right)\right.\right\} \left(\omega\right),
\end{equation}
is Carath\'eodory. Observe, though, that we may write
\begin{flalign}
H\left(\omega,{\bf p}\right) & \equiv{\displaystyle \sum_{i\in\mathbb{N}_{R}^{+}}H_{I}\left(\omega,{\bf p}_{i}\right)},
\end{flalign}
where the random function $H_{I}:\Omega\times{\cal S}\rightarrow\mathbb{R}$
is defined as 
\begin{equation}
H_{I}\left(\omega,{\bf p}\right)\triangleq\mathbb{E}\left\{ \left.\dfrac{P_{c}P_{0}\left|f\left({\bf p},t\right)\right|^{2}\left|g\left({\bf p},t\right)\right|^{2}}{P_{0}\sigma_{D}^{2}\left|f\left({\bf p},t\right)\right|^{2}+P_{c}\sigma^{2}\left|g\left({\bf p},t\right)\right|^{2}+\sigma^{2}\sigma_{D}^{2}}\right|\mathscr{C}\left({\cal T}_{t-1}\right)\right\} \left(\omega\right).
\end{equation}
Because a finite sum of Carath\'eodory functions (in this case, in
different variables) is obviously Carath\'eodory, it suffices to
show that $H_{I}$ is Carath\'eodory. 

First, it is easy to see that $H_{I}\left(\cdot,{\bf p}\right)$ constitutes
a well defined conditional expectation of a nonnegative random variable,
for all ${\bf p}\in{\cal S}$. Therefore, what remains is to show
that $H_{I}\left(\omega,\cdot\right)$ is continuous on ${\cal S}$,
everywhere with respect to $\omega\in\Omega$. For this, we will rely
on the sequential definition of continuity and the explicit representation
of $H_{I}$ as an integral with respect to the Lebesgue measure, which
exploits the form of the projective system of finite dimensional distributions
of $\left|f\left({\bf p},t\right)\right|^{2}$ and $\left|g\left({\bf p},t\right)\right|^{2}$.
In particular, because of the trick (\ref{eq:CONVERTER}), it is easy
to show that $H_{I}$ can be equivalently expressed as the Lebesgue
integral
\begin{equation}
H_{I}\left(\omega,{\bf p}\right)=\int_{\mathbb{R}^{2}}r\left(\boldsymbol{x}\right){\cal N}\left(\boldsymbol{x};\boldsymbol{\mu}_{2}\left(\omega,{\bf p}\right),\boldsymbol{\Sigma}_{2}\left(\omega,{\bf p}\right)\right)\text{d}\boldsymbol{x},
\end{equation}
where the continuous function $r:\mathbb{R}^{2}\rightarrow\mathbb{R}_{++}$
is defined as (recall that $\varsigma\equiv\log\left(10\right)/10$)
\begin{equation}
r\left(\boldsymbol{x}\right)\equiv r\left(x_{1},x_{2}\right)\triangleq\dfrac{P_{c}P_{0}10^{\rho/10}\left[\exp\left(x_{1}+x_{2}\right)\right]^{\varsigma}}{P_{0}\sigma_{D}^{2}\left[\exp\left(x_{1}\right)\right]^{\varsigma}+P_{c}\sigma^{2}\left[\exp\left(x_{2}\right)\right]^{\varsigma}+10^{-\rho/10}\sigma^{2}\sigma_{D}^{2}},
\end{equation}
for all $\boldsymbol{x}\equiv\left(x_{1},x_{2}\right)\in\mathbb{R}^{2}$,
and ${\cal N}:\mathbb{R}^{2}\times{\cal S}\times\Omega\rightarrow\mathbb{R}_{++}$,
corresponds to the \textit{jointly Gaussian} conditional density of
$F\left({\bf p},t\right)$ and $G\left({\bf p},t\right)$, relative
to $\mathscr{C}\left({\cal T}_{t-1}\right)$, with mean $\boldsymbol{\mu}_{2}:\left(\omega,{\bf p}\right)\rightarrow\mathbb{R}^{2\times1}$
and covariance $\boldsymbol{\Sigma}_{2}:\left(\omega,{\bf p}\right)\rightarrow\mathbb{S}_{++}^{2\times2}$
explicitly depending on $\omega$ and ${\bf p}$ as 
\begin{flalign}
\boldsymbol{\mu}_{2}\left(\omega,{\bf p}\right) & \equiv\boldsymbol{\mu}_{2}\left(\mathscr{C}\left({\cal T}_{t-1}\right)\left(\omega\right);{\bf p}\right)\quad\text{and}\\
\boldsymbol{\Sigma}_{2}\left(\omega,{\bf p}\right) & \equiv\boldsymbol{\Sigma}_{2}\left(\mathscr{C}\left({\cal T}_{t-1}\right)\left(\omega\right);{\bf p}\right),\quad\forall\left(\omega,{\bf p}\right)\in\Omega\times{\cal S}.
\end{flalign}
Via a simple change of variables, we may reexpress $H_{I}\left(\omega,{\bf p}\right)$
as
\begin{equation}
H_{I}\left(\omega,{\bf p}\right)\equiv\int_{\mathbb{R}^{2}}r\left(\boldsymbol{x}+\boldsymbol{\mu}_{2}\left(\omega,{\bf p}\right)\right){\cal N}\left(\boldsymbol{x};{\bf 0},\boldsymbol{\Sigma}_{2}\left(\omega,{\bf p}\right)\right)\text{d}\boldsymbol{x}.
\end{equation}

It is straightforward to verify that both $\boldsymbol{\mu}_{2}\left(\omega,\cdot\right)$
and $\boldsymbol{\Sigma}_{2}\left(\omega,\cdot\right)$ are continuous
functions in ${\bf p}\in{\cal S}$, for all $\omega\in\Omega$. This
is due to the fact that all functions involving ${\bf p}$ in the
wireless channel model introduced in Section \ref{sec:Spatiotemporal-Wireless-Channel}
are trivially continuous in this variable. Equivalently, we may assert
that the whole integrand $r\left(\boldsymbol{x}+\boldsymbol{\mu}_{2}\left(\omega,\cdot\right)\right){\cal N}\left(\boldsymbol{x};{\bf 0},\boldsymbol{\Sigma}_{2}\left(\omega,\cdot\right)\right)$
is a continuous function, for all pairs $\left(\omega,\boldsymbol{x}\right)\in\Omega\times\mathbb{R}^{2}$.
Next, fix $\omega\in\Omega$, and for \textit{arbitrary} ${\bf p}\in{\cal S}$,
consider \textit{any} sequence $\left\{ {\bf p}_{k}\in{\cal S}\right\} _{k\in\mathbb{N}}$,
such that ${\bf p}_{k}\underset{k\rightarrow\infty}{\longrightarrow}{\bf p}$.
Then, $H_{I}\left(\omega,\cdot\right)$ is continuous if and only
if $H_{I}\left(\omega,{\bf p}_{k}\right)\underset{k\rightarrow\infty}{\longrightarrow}H_{I}\left(\omega,{\bf p}\right)$.
We will show this via a simple application of the Dominated Convergence
Theorem. Emphasizing the dependence on ${\bf p}$ as a superscript
for the sake of clarity, we can write
\begin{align}
r\left(\boldsymbol{x}+\boldsymbol{\mu}_{2}^{{\bf p}}\right){\cal N}\left(\boldsymbol{x};{\bf 0},\boldsymbol{\Sigma}_{2}^{{\bf p}}\right) & \hspace{-2pt}\equiv\hspace{-2pt}r\left(\boldsymbol{x}+\boldsymbol{\mu}_{2}^{{\bf p}}\right)\dfrac{\exp\left(-\dfrac{1}{2}\boldsymbol{x}^{\boldsymbol{T}}\left[\boldsymbol{\Sigma}_{2}^{{\bf p}}\right]^{-1}\boldsymbol{x}\right)}{2\pi\sqrt{\det\left(\boldsymbol{\Sigma}_{2}^{{\bf p}}\right)}}\nonumber \\
 & \hspace{-2pt}\le\hspace{-2pt}r\left(\boldsymbol{x}+\boldsymbol{\mu}_{2}^{{\bf p}}\right)\dfrac{\exp\left(-\dfrac{1}{2}\lambda_{min}\left(\left[\boldsymbol{\Sigma}_{2}^{{\bf p}}\right]^{-1}\right)\left\Vert \boldsymbol{x}\right\Vert _{2}^{2}\right)}{2\pi\sqrt{\det\left(\boldsymbol{\Sigma}_{2}^{{\bf p}}\right)}}\nonumber \\
 & \hspace{-2pt}\equiv\hspace{-2pt}r\left(\boldsymbol{x}+\boldsymbol{\mu}_{2}^{{\bf p}}\right)\dfrac{\exp\left(-\dfrac{\left\Vert \boldsymbol{x}\right\Vert _{2}^{2}}{2\lambda_{max}\left(\boldsymbol{\Sigma}_{2}^{{\bf p}}\right)}\right)}{2\pi\sqrt{\det\left(\boldsymbol{\Sigma}_{2}^{{\bf p}}\right)}}\nonumber \\
 & \hspace{-2pt}\le\hspace{-2pt}\dfrac{P_{0}10^{\rho/10}}{\sigma^{2}}\left[\exp\hspace{-2pt}\left(x_{1}+\boldsymbol{\mu}_{2}^{{\bf p}}\left(1\right)\right)\right]^{\varsigma}\dfrac{\exp\left(-\dfrac{\left\Vert \boldsymbol{x}\right\Vert _{2}^{2}}{2\lambda_{max}\left(\boldsymbol{\Sigma}_{2}^{{\bf p}}\right)}\right)}{2\pi\sqrt{\det\left(\boldsymbol{\Sigma}_{2}^{{\bf p}}\right)}}\nonumber \\
 & \hspace{-2pt}\le\hspace{-2pt}\dfrac{P_{0}10^{\rho/10}}{\sigma^{2}}\left[\exp\hspace{-2pt}\left(x_{1}+\sup_{{\bf p}\in{\cal S}}\boldsymbol{\mu}_{2}^{{\bf p}}\left(1\right)\right)\right]^{\varsigma}\dfrac{\exp\left(-\dfrac{\left\Vert \boldsymbol{x}\right\Vert _{2}^{2}}{2{\displaystyle \sup_{{\bf p}\in{\cal S}}\lambda_{max}\left(\boldsymbol{\Sigma}_{2}^{{\bf p}}\right)}}\right)}{2\pi\sqrt{{\displaystyle \inf_{{\bf p}\in{\cal S}}\det\left(\boldsymbol{\Sigma}_{2}^{{\bf p}}\right)}}}\nonumber \\
 & \triangleq\dfrac{P_{0}10^{\rho/10}}{\sigma^{2}}\left[\exp\hspace{-2pt}\left(x_{1}+p_{1}\right)\right]^{\varsigma}\dfrac{\exp\left(-\dfrac{\left\Vert \boldsymbol{x}\right\Vert _{2}^{2}}{2p_{2}}\right)}{2\pi\sqrt{p_{3}}}\nonumber \\
 & \triangleq\psi\left(\omega,\boldsymbol{x}\right),\label{eq:INTEGRABLE}
\end{align}
where, due to the continuity of $\boldsymbol{\mu}_{2}\left(\omega,\cdot\right)$
and $\boldsymbol{\Sigma}_{2}\left(\omega,\cdot\right)$, the continuity
of the maximum eigenvalue and determinant operators, the fact that
${\cal S}$ is compact, and the power of the Extreme Value Theorem,
all extrema involved are finite and, of course, independent of ${\bf p}$.
It is now easy to verify that the RHS of (\ref{eq:INTEGRABLE}) is
integrable. Indeed, by Fubini's Theorem (Theorem 2.6.4 in \cite{Ash2000Probability})
\begin{flalign}
\int_{\mathbb{R}^{2}}\psi\left(\omega,\boldsymbol{x}\right)\text{d}\boldsymbol{x} & \hspace{-2pt}=\hspace{-2pt}\dfrac{P_{0}10^{\rho/10}}{\sigma^{2}}\dfrac{\exp\hspace{-2pt}\left(\varsigma p_{1}\right)}{\sqrt{p_{3}}}{\scaleint{6ex}}_{\!\!\!\!\mathbb{R}^{2}}\exp\hspace{-2pt}\left(\varsigma x_{1}\right)\dfrac{1}{2\pi}\exp\left(-\dfrac{\left\Vert \boldsymbol{x}\right\Vert _{2}^{2}}{2p_{2}}\right)\text{d}\boldsymbol{x}\nonumber \\
 & \hspace{-2pt}\equiv\hspace{-2pt}\dfrac{P_{0}10^{\rho/10}}{\sigma^{2}}\dfrac{\exp\hspace{-2pt}\left(\varsigma p_{1}\right)\hspace{-2pt}p_{2}}{\sqrt{p_{3}}}{\scaleint{6ex}}_{\!\!\!\!\mathbb{R}^{2}}\exp\hspace{-2pt}\left(\varsigma x_{1}\right)\dfrac{1}{2\pi p_{2}}\exp\hspace{-2pt}\left(\hspace{-2pt}-\dfrac{\left\Vert \boldsymbol{x}\right\Vert _{2}^{2}}{2p_{2}}\hspace{-2pt}\right)\hspace{-2pt}\text{d}\boldsymbol{x}\nonumber \\
 & \hspace{-2pt}=\hspace{-2pt}\dfrac{P_{0}10^{\rho/10}}{\sigma^{2}}\dfrac{\exp\hspace{-2pt}\left(\varsigma p_{1}\right)\hspace{-2pt}p_{2}}{\sqrt{p_{3}}}{\scaleint{6ex}}_{\!\!\!\!\mathbb{R}}\exp\hspace{-2pt}\left(\varsigma x_{1}\right)\dfrac{1}{\sqrt{2\pi p_{2}}}\exp\hspace{-2pt}\left(\hspace{-2pt}-\dfrac{x_{1}^{2}}{2p_{2}}\hspace{-2pt}\right)\hspace{-2pt}\text{d}x_{1}\nonumber \\
 & \hspace{-2pt}=\hspace{-2pt}\dfrac{P_{0}10^{\rho/10}}{\sigma^{2}}\dfrac{\exp\hspace{-2pt}\left(\varsigma p_{1}\left(\omega\right)\right)\hspace{-2pt}p_{2}\left(\omega\right)}{\sqrt{p_{3}\left(\omega\right)}}\exp\left(\dfrac{p_{2}\left(\omega\right)}{2}\varsigma^{2}\right)<+\infty,\quad\omega\in\Omega.
\end{flalign}
That is,
\begin{equation}
\psi\left(\omega,\cdot\right)\in{\cal L}_{1}\left(\mathbb{R}^{2},\mathscr{B}\left(\mathbb{R}^{2}\right),{\cal L};\mathbb{R}\right),\quad\omega\in\Omega,
\end{equation}
where ${\cal L}$ denotes the Lebesgue measure. We can now call Dominated
Convergence; since, for each $\boldsymbol{x}\in\mathbb{R}^{2}$ (and
each $\omega\in\Omega$),
\begin{equation}
r\left(\boldsymbol{x}+\boldsymbol{\mu}_{2}\left(\omega,{\bf p}_{k}\right)\right){\cal N}\left(\boldsymbol{x};{\bf 0},\boldsymbol{\Sigma}_{2}\left(\omega,{\bf p}_{k}\right)\right)\underset{k\rightarrow\infty}{\longrightarrow}r\left(\boldsymbol{x}+\boldsymbol{\mu}_{2}\left(\omega,{\bf p}\right)\right){\cal N}\left(\boldsymbol{x};{\bf 0},\boldsymbol{\Sigma}_{2}\left(\omega,{\bf p}\right)\right)
\end{equation}
and all members of this sequence are dominated by the integrable function
$\psi\left(\omega,\cdot\right)$, it is true that
\begin{multline}
H_{I}\left(\omega,{\bf p}_{k}\right)\hspace{-2pt}\equiv\hspace{-2pt}\int_{\mathbb{R}^{2}}r\left(\boldsymbol{x}+\boldsymbol{\mu}_{2}\left(\omega,{\bf p}_{k}\right)\right){\cal N}\left(\boldsymbol{x};{\bf 0},\boldsymbol{\Sigma}_{2}\left(\omega,{\bf p}_{k}\right)\right)\text{d}\boldsymbol{x}\\
\underset{k\rightarrow\infty}{\longrightarrow}\int_{\mathbb{R}^{2}}r\left(\boldsymbol{x}+\boldsymbol{\mu}_{2}\left(\omega,{\bf p}\right)\right){\cal N}\left(\boldsymbol{x};{\bf 0},\boldsymbol{\Sigma}_{2}\left(\omega,{\bf p}\right)\right)\text{d}\boldsymbol{x}\equiv\hspace{-1.5pt}H_{I}\left(\omega,{\bf p}\right).
\end{multline}
But $\left\{ {\bf p}_{k}\right\} _{k\in\mathbb{N}}$ and ${\bf p}$
are arbitrary, showing that $H_{I}\left(\omega,\cdot\right)$ is continuous,
for each fixed $\omega\in\Omega$. Hence, $H_{I}$ is Carath\'eodory
on $\Omega\times{\cal S}$.\hfill{}\ensuremath{\bigstar}

The proof to the second part of Theorem \ref{lem:C1C4_SAT} follows
easily by direct application of the Fundamental Lemma (Lemma \ref{lem:FUND_Lemma_FINAL};
also see Table \ref{tab:Variable-matching}).\hfill{}\ensuremath{\blacksquare}

\subsubsection{Proof of Lemma \ref{thm:(Big-Expectations)}}

In the notation of the statement of the lemma, the joint conditional
distribution of $\left[F\left({\bf p},t\right)\,G\left({\bf p},t\right)\right]^{\boldsymbol{T}}$
relative to the $\sigma$-algebra $\mathscr{C}\left({\cal T}_{t-1}\right)$
can be readily shown to be Gaussian with mean $\boldsymbol{\mu}_{\left.t\right|t-1}^{F,G}\hspace{-2pt}\left({\bf p}\right)$
and covariance $\boldsymbol{\Sigma}_{\left.t\right|t-1}^{F,G}\hspace{-2pt}\left({\bf p}\right)$,
for all $\left({\bf p},t\right)\in{\cal S}\times\mathbb{N}_{N_{T}}^{2}$.
This is due to the fact that, in Section \ref{sec:Spatiotemporal-Wireless-Channel},
we have implicitly assumed that the channel fields $F\left({\bf p},t\right)$
and $G\left({\bf p},t\right)$ are jointly Gaussian. It is then a
typical exercise (possibly somewhat tedious though) to show that the
functions $\boldsymbol{\mu}_{\left.t\right|t-1}^{F,G}$ and $\boldsymbol{\Sigma}_{\left.t\right|t-1}^{F,G}$
are of the form asserted in the statement of the lemma. Regarding
the proof for (\ref{eq:BOX_1}), observe that we can write
\begin{flalign}
 & \hspace{-2pt}\hspace{-2pt}\hspace{-2pt}\hspace{-2pt}\hspace{-2pt}\hspace{-2pt}\hspace{-2pt}\hspace{-2pt}\hspace{-2pt}\hspace{-2pt}\hspace{-2pt}\hspace{-2pt}\hspace{-2pt}\hspace{-2pt}\hspace{-2pt}\hspace{-2pt}\hspace{-2pt}\hspace{-2pt}\mathbb{E}\left\{ \left.\left|f\left({\bf p},t\right)\right|^{m}\left|g\left({\bf p},t\right)\right|^{n}\right|\mathscr{C}\left({\cal T}_{t-1}\right)\right\} \nonumber \\
 & \equiv10^{\left(m+n\right)\rho/20}\mathbb{E}\left\{ \left.\exp\left(\dfrac{\log\left(10\right)}{20}\left(mF\left({\bf p},t\right)+nG\left({\bf p},t\right)\right)\right)\right|\mathscr{C}\left({\cal T}_{t-1}\right)\right\} \nonumber \\
 & \equiv10^{\left(m+n\right)\rho/20}\mathbb{E}\left\{ \left.\exp\left(\dfrac{\log\left(10\right)}{20}\left[m\,n\right]\left[F\left({\bf p},t\right)\,G\left({\bf p},t\right)\right]^{\boldsymbol{T}}\right)\right|\mathscr{C}\left({\cal T}_{t-1}\right)\right\} ,
\end{flalign}
with the conditional expectation on the RHS being nothing else than
the conditional moment generating function of the conditionally jointly
Gaussian random vector $\left[F\left({\bf p},t\right)\,G\left({\bf p},t\right)\right]^{\boldsymbol{T}}$
at each ${\bf p}$ and $t$, evaluated at the point $\left(\log\left(10\right)/20\right)\left[m\,n\right]^{\boldsymbol{T}}$,
for any choice of $\left(m,n\right)\in\mathbb{Z}\times\mathbb{Z}$.
Recalling the special form of the moment generating function for Gaussian
random vectors, the result readily follows.\hfill{}\ensuremath{\blacksquare}

\subsubsection{Proof of Theorem \ref{lem:WELL_Behaved}}

It will suffice to show that both objectives of (\ref{eq:SINR_APPROX_PROG_1})
and (\ref{eq:SINR_APPROX_PROG_2}) are Carath\'eodory in $\Omega\times{\cal S}$.
But this statement may be easily shown by analytically expressing
both (\ref{eq:SINR_APPROX_PROG_1}) and (\ref{eq:SINR_APPROX_PROG_2})
using Lemma \ref{thm:(Big-Expectations)}. Now, since both objectives
of (\ref{eq:SINR_APPROX_PROG_1}) and (\ref{eq:SINR_APPROX_PROG_2})
are Carath\'eodory, we may invoke Theorem \ref{thm:Measurability-Preservation}
(Appendix B), in an inductive fashion, for each $t\in\mathbb{N}_{N_{T}}^{2}$,
guaranteeing the existence of at least one $\mathscr{C}\left({\cal T}_{t-1}\right)$-measurable
decision for either (\ref{eq:SINR_APPROX_PROG_1}), or (\ref{eq:SINR_APPROX_PROG_2}),
say $\widetilde{{\bf p}}^{*}\left(t\right)$, which solves the optimization
problem considered, for all $\omega\in\Omega$. Proceeding inductively
gives the result.\hfill{}\ensuremath{\blacksquare}

\subsubsection{Proof of Theorem \ref{lem:QoS_INCREASES}}

By assumption, $V\left({\bf p},t\right)$ is $\mathbf{L.MD.G}\diamondsuit\left(\mathscr{\mathscr{H}}_{t},\mu\right)$,
implying, for every $t\in\mathbb{N}_{N_{T}}^{+}$, the existence of
an event $\Omega_{t}\subseteq\Omega$, satisfying ${\cal P}\left(\Omega_{t}\right)\equiv1$,
such that, for every ${\bf p}\in{\cal S}^{R}$,
\begin{equation}
\mu\mathbb{E}\left\{ \left.V\left({\bf p},t-1\right)\right|\mathscr{\mathscr{H}}_{t-1}\right\} \left(\omega\right)\equiv\mathbb{E}\left\{ \left.V\left({\bf p},t\right)\right|\mathscr{\mathscr{H}}_{t-1}\right\} \left(\omega\right),\quad\forall\omega\in\Omega_{t}.\label{eq:PCID_1}
\end{equation}

Fix $t\in\mathbb{N}_{N_{T}}^{2}$. Consider any admissible policy
${\bf p}^{o}\left(t\right)$ at $t$, implemented at $t$ and decided
at $t-1\in\mathbb{N}_{N_{T}-1}^{+}$. By our assumptions, $V\left(\cdot,\cdot,t\right)$
is $\boldsymbol{SP}\diamondsuit\mathfrak{C}_{\mathscr{H}_{t}}$. Additionally,
because ${\bf p}^{o}\left(t\right)$ is admissible, it will be measurable
relative to the limit $\sigma$-algebra $\mathscr{P}_{t}^{\uparrow}$
and, hence, measurable relative to $\mathscr{\mathscr{H}}_{t}$. Thus,
there exists an event $\Omega_{t}^{{\bf p}^{o}}\subseteq\Omega$,
with ${\cal P}\left(\Omega_{t}^{{\bf p}^{o}}\right)\equiv1$, such
that, for every $\omega\in\Omega_{t}^{{\bf p}^{o}}$,
\begin{flalign}
\mathbb{E}\left\{ \left.V\left({\bf p}^{o}\left(t\right),t\right)\right|\mathscr{\mathscr{H}}_{t}\right\} \left(\omega\right) & \equiv\left.\mathbb{E}\left\{ \left.V\left({\bf p},t\right)\right|\mathscr{\mathscr{H}}_{t}\right\} \left(\omega\right)\right|_{{\bf p}={\bf p}^{o}\left(\omega,t\right)}\nonumber \\
 & \equiv h_{t}\left(\omega,{\bf p}^{o}\left(\omega,t\right)\right),\label{eq:Gen_SP_1}
\end{flalign}
where the extended real-valued random function $h_{t}:\Omega\times{\cal S}^{R}\rightarrow\overline{\mathbb{R}}$
is jointly $\mathscr{H}_{t}\otimes\mathscr{B}\left({\cal S}^{R}\right)$-measurable,
with $h_{t}\left(\omega,{\bf p}\right)\equiv\mathbb{E}\left\{ \left.V\left({\bf p},t\right)\right|\mathscr{\mathscr{H}}_{t}\right\} \left(\omega\right)$,
everywhere in $\left(\omega,{\bf p}\right)\in\Omega\times{\cal S}^{R}$. 

Also by our assumptions, $V\left(\cdot,\cdot,t\right)$ is $\boldsymbol{SP}\diamondsuit\mathfrak{C}_{\mathscr{H}_{t-1}}$,
as well. Similarly to the arguments made above, if ${\bf p}^{o}\left(t\right)$
is assumed to be measurable relative to the limit $\sigma$-algebra
$\mathscr{P}_{t-1}^{\uparrow}$, or, in other words, admissible at
$t-1$, then it will also be measurable relative to $\mathscr{\mathscr{H}}_{t-1}$.
Therefore, there exists an event $\Omega_{t^{-}}^{{\bf p}^{o}}\subseteq\Omega$,
with ${\cal P}\left(\Omega_{t^{-}}^{{\bf p}^{o}}\right)\equiv1$,
such that, for every $\omega\in\Omega_{t^{-}}^{{\bf p}^{o}}$,
\begin{flalign}
\mathbb{E}\left\{ \left.V\left({\bf p}^{o}\left(t\right),t\right)\right|\mathscr{\mathscr{H}}_{t-1}\right\} \left(\omega\right) & \equiv\left.\mathbb{E}\left\{ \left.V\left({\bf p},t\right)\right|\mathscr{\mathscr{H}}_{t-1}\right\} \left(\omega\right)\right|_{{\bf p}={\bf p}^{o}\left(\omega,t\right)}\nonumber \\
 & \equiv h_{t^{-}}\left(\omega,{\bf p}^{o}\left(\omega,t\right)\right),\label{eq:Gen_SP_2}
\end{flalign}
where the random function $h_{t^{-}}:\Omega\times{\cal S}^{R}\rightarrow\overline{\mathbb{R}}$
is jointly $\mathscr{H}_{t-1}\otimes\mathscr{B}\left({\cal S}^{R}\right)$-measurable,
with $h_{t^{-}}\left(\omega,{\bf p}\right)\equiv\mathbb{E}\left\{ \left.V\left({\bf p},t\right)\right|\mathscr{\mathscr{H}}_{t-1}\right\} \left(\omega\right)$,
everywhere in $\left(\omega,{\bf p}\right)\in\Omega\times{\cal S}^{R}$.
Note that, by construction, ${\bf p}^{o}\left(t\right)$ will also
be admissible at time $t$ and, therefore, measurable relative to
and $\mathscr{\mathscr{H}}_{t}$, as well.

Now, we combine the arguments made above. Keep $t\in\mathbb{N}_{N_{T}}^{2}$
fixed. At time slot $t-2\in\mathbb{N}_{N_{T}-2}$, let ${\bf p}^{o}\left(t-1\right)\equiv{\bf p}^{o}\left(\omega,t-1\right)$
be a $\mathscr{C}\left({\cal T}_{t-2}\right)$-measurable admissible
policy (recall that \textbf{C1-C6} are satisfied by assumption; also
recall that, if $t\equiv2$, $\mathscr{C}\left({\cal T}_{t-2}\right)\equiv\mathscr{C}\left({\cal T}_{0}\right)$
is the trivial $\sigma$-algebra). At the \textit{next} time slot
$t-1\in\mathbb{N}_{N_{T}-1}^{+}$, let us choose ${\bf p}^{o}\left(t\right)\equiv{\bf p}^{o}\left(\omega,t-1\right)$;
in this case, ${\bf p}^{o}\left(t\right)$ will also be $\mathscr{C}\left({\cal T}_{t-2}\right)$-measurable
and result in the \textit{same final position for the relays at time
slot} $t\in\mathbb{N}_{N_{T}}^{2}$. As a result, the relays just
stay still. Under these circumstances, at time slot $t-1\in\mathbb{N}_{N_{T}-1}^{+}$,
the expected network QoS will be $\mathbb{E}\left\{ V\left({\bf p}^{o}\left(t-1\right),t-1\right)\right\} $,
whereas, at the next time slot $t\in\mathbb{N}_{N_{T}}^{2}$, it will
be $\mathbb{E}\left\{ V\left({\bf p}^{o}\left(t-1\right),t\right)\right\} $.
Exploiting (\ref{eq:PCID_1}), we may write
\begin{align}
\mu h_{t-1}\left(\omega,{\bf p}\right) & \equiv h_{t^{-}}\left(\omega,{\bf p}\right),\quad\forall\left(\omega,{\bf p}\right)\in\Omega_{t}\bigcap\Omega_{t-1}^{{\bf p}^{o}}\bigcap\Omega_{t^{-}}^{{\bf p}^{o}}\times{\cal S}^{R},
\end{align}
where, obviously, ${\cal {\cal P}}\left(\Omega_{t}\bigcap\Omega_{t-1}^{{\bf p}^{o}}\bigcap\Omega_{t^{-}}^{{\bf p}^{o}}\right)\equiv1$.
Consequently, it will be true that
\begin{equation}
\mu h_{t-1}\left(\omega,{\bf p}^{o}\left(\omega,t-1\right)\right)\equiv h_{t^{-}}\left(\omega,{\bf p}^{o}\left(\omega,t-1\right)\right),\quad\forall\omega\in\Omega_{t}\bigcap\Omega_{t-1}^{{\bf p}^{o}}\bigcap\Omega_{t^{-}}^{{\bf p}^{o}}.
\end{equation}
From (\ref{eq:Gen_SP_1}) and (\ref{eq:Gen_SP_2}), it is also true
that
\begin{equation}
\mu\mathbb{E}\left\{ \left.V\left({\bf p}^{o}\left(t-1\right),t-1\right)\right|\mathscr{\mathscr{H}}_{t-1}\right\} \left(\omega\right)\equiv\mathbb{E}\left\{ \left.V\left({\bf p}^{o}\left(t-1\right),t\right)\right|\mathscr{\mathscr{H}}_{t-1}\right\} \left(\omega\right),
\end{equation}
almost everywhere with respect to ${\cal P}$. This, of course, implies
that
\begin{equation}
\mu\mathbb{E}\left\{ V\left({\bf p}^{o}\left(t-1\right),t-1\right)\right\} \equiv\mathbb{E}\left\{ V\left({\bf p}^{o}\left(t-1\right),t\right)\right\} ,\label{eq:no_decrease}
\end{equation}
and for all $t\in\mathbb{N}_{N_{T}}^{2}$, since $t$ was arbitrary.

Since (\ref{eq:no_decrease}) holds for all admissible policies decided
at time slot $t-2\in\mathbb{N}_{N_{T}-2}$, it will also hold for
the respective optimal policy, that is,
\begin{equation}
\mu\mathbb{E}\left\{ V\left({\bf p}^{*}\left(t-1\right),t-1\right)\right\} \equiv\mathbb{E}\left\{ V\left({\bf p}^{*}\left(t-1\right),t\right)\right\} ,\quad\forall t\in\mathbb{N}_{N_{T}}^{2}.
\end{equation}
Next, as discussed above, the choice ${\bf p}^{o}\left(t\right)\equiv{\bf p}^{*}\left(\omega,t-1\right)$
constitutes an admissible policy decided at time slot $t-1\in\mathbb{N}_{N_{T}-1}^{+}$;
it suffices to see that ${\bf p}^{*}\left(\omega,t-1\right)\in{\cal C}\left({\bf p}^{*}\left(\omega,t-1\right)\right)$,
by definition of our initial $2$-stage problem, because ``staying
still'' is always a feasible decision for the relays. Consequently,
because the optimal policy ${\bf p}^{*}\left(t\right)$ results in
the highest network QoS, \textit{among all admissible policies}, it
will be true that
\begin{equation}
\mu\mathbb{E}\left\{ V\left({\bf p}^{*}\left(t-1\right),t-1\right)\right\} \le\mathbb{E}\left\{ V\left({\bf p}^{*}\left(t\right),t\right)\right\} ,\quad\forall t\in\mathbb{N}_{N_{T}}^{2},
\end{equation}
completing the proof of Theorem \ref{lem:QoS_INCREASES}.\hfill{}\ensuremath{\blacksquare}

\bibliographystyle{IEEEbib}
\bibliography{IEEEabrv}

\end{document}